\documentclass[10pt]{article}
\usepackage{amsfonts}
\usepackage{amsmath,amsthm,mathrsfs}
\usepackage{hyperref}

\usepackage[all]{xy}
\usepackage{graphicx}

\usepackage{amssymb}
\usepackage{latexsym}

\DeclareMathAlphabet\EuFrak{U}{euf}{m}{n}	
\SetMathAlphabet\EuFrak{bold}{U}{euf}{b}{n}	

\newcommand{\bo}{{\partial_0 b}}
\newcommand{\bl}{{\partial_1 b}}

\newcommand{\co}{{\partial_0 c}}
\newcommand{\cl}{{\partial_1 c}}
\newcommand{\cz}{{\partial_2 c}}

\newcommand{\tto}{ {\tiny{\to}} }

\newcommand{\ovl}{\overline}
\newcommand{\unl}{\underline}
\newcommand{\wa}{\widehat}
\newcommand{\wt}{\widetilde}

\newcommand{\bC} {{\mathbb C}}

\newcommand{\bR} {{\mathbb R}}
\newcommand{\bT} {{\mathbb T}}
\newcommand{\bU} {{\mathbb U}}
\newcommand{\bSU} {{\mathbb{SU}}}

\newcommand{\bZ} {{\mathbb Z}}

\newcommand{\bN} {{\mathbb N}}

\newcommand{\veps}{\varepsilon}

\newcommand{\tbe}{{\emph{\textbf{e}}}}
\newcommand{\tbf}{{\emph{\textbf{f }}}}
\newcommand{\tbfu}{{\emph{\textbf{f\,}}}}
\newcommand{\tbu}{{\emph{\textbf{u\,}}}}
\newcommand{\tbv}{{\emph{\textbf{v\,}}}}
\newcommand{\tbi}{{\emph{\textbf{i\,}}}}

\newcommand{\bs}{\boldsymbol}

\newcommand{\mA}{\mathcal A}
\newcommand{\mB}{\mathcal B}
\newcommand{\mC}{\mathcal C}

\newcommand{\mE}{\mathcal E}
\newcommand{\mF}{\mathcal F}
\newcommand{\mG}{\mathcal G}
\newcommand{\mH}{\mathcal H}

\newcommand{\mK}{\mathcal K}

\newcommand{\mM}{\mathcal M}

\newcommand{\mO}{\mathcal O}

\newcommand{\mS}{\mathcal S}

\newcommand{\mX}{\mathcal X}

\newcommand{\mU}{\mathcal U}

\newcommand{\sA}{\mathscr{A}}
\newcommand{\sH}{\mathscr{H}}
\newcommand{\sK}{\mathscr{K}}

\newcommand{\ad}{{\mathrm{ad}}}
\newcommand{\Hol}{{\mathrm{Hol}}}
\newcommand{\hhol}{\jmath \, }
\newcommand{\Nat}{{N}}

\newcommand{\Calg}{{\bf C^*alg}}
\newcommand{\Hilb}{{\bf Hilb}}

\newcommand{\obj}{{\bf obj \, }}

\newcommand{\Cgrb}{{\bf C^*grb}}

\newcommand{\tdyn}{{\bf tC^*dyn}}

\newcommand{\pbun}{{\bf pbun }}
\newcommand{\bun}{{\bf bun }}
\newcommand{\Aut}{{\bf aut }}
\newcommand{\Inn}{{\bf inn }}

\newtheorem{thm}{Theorem}[section]
\newtheorem{cor}[thm]{Corollary}
\newtheorem{lem}[thm]{Lemma}
\newtheorem{prop}[thm]{Proposition}

\newtheorem{defn}[thm]{Definition}

\theoremstyle{definition}
\newtheorem{ex}{Example}[section]
\newtheorem{rem}[thm]{Remark}

\theoremstyle{remark}

\numberwithin{equation}{section}

\begin{document}

\author{{\sf Ezio Vasselli}
\\{\small{Dipartimento di Matematica,}}
\\{\small{Universit\`a di Roma ``Tor Vergata'',}}
\\{\small{Via della Ricerca Scientifica 1, I-00133 Roma, Italy.}}
\\{\sf ezio.vasselli@gmail.com} \\
\\{\emph{Dedicated to Damiano Vasselli and his curiosity}}
}

\title{Gerbes over posets and twisted C*- dynamical systems}
\maketitle

\begin{abstract}
A base $\Delta$ generating the topology of a space $M$
becomes a partially ordered set (poset), when ordered under inclusion of open subsets.
Given a precosheaf over $\Delta$ of fixed-point spaces (typically $C^*$-algebras) under the action of a group $G$,
in general one cannot find a precosheaf of $G$-spaces having it as fixed-point precosheaf.
Rather one gets a gerbe over $\Delta$, that is, 
a "twisted precosheaf" whose twisting is encoded by a cocycle with coefficients in a suitable 2--group.
We give a notion of holonomy for a gerbe, in terms of a non-abelian cocycle over the fundamental group $\pi_1(M)$.
At the $C^*$-algebraic level, holonomy leads to a general notion of twisted $C^*$-dynamical system,
based on a generic 2-group instead of the usual adjoint action on the underlying $C^*$-algebra.
As an application of these notions, we study presheaves of group duals (DR-presheaves)
and prove that the dual object of a DR-presheaf is a group gerbe over $\Delta$.
It is also shown that any section of a DR-presheaf defines a twisted action of $\pi_1(M)$ on a Cuntz algebra.

\


\end{abstract}

\newpage
\tableofcontents
\markboth{Contents}{Contents}

\newpage

\section{Introduction}
\label{sec.intro}

The geometry of partially ordered sets (\emph{posets}) is a research line originated from ideas of J.E. Roberts and
has its roots in the early days of quantum field theory,
when Bohr and Rosenfeld pointed out that it is not possible to evaluate a quantum field over a point 
as a consequence of the uncertainty relations \cite[\S I.5.2]{Haa}.
Rather, what can be done is to define an "averaged field" localized in a spacetime region,
and this point is reflected in the form that quantum fields take in mathematically rigorous approaches
as those of Wightman \cite[Chap.II.1]{Haa} and Haag-Kastler \cite[Chap.III]{Haa}.

It is in the latter one, known as \emph{algebraic quantum field theory}, that the inclusion relation between spacetime regions
takes a fundamental role. In this setting the object describing a quantum system is a precosheaf, say $\mA$,
with inclusion *-morphisms
\begin{equation}
\label{eq.intro.1}
\jmath_{\omega o} : A_o \to A_\omega  \ \ , \ \ o \subseteq \omega  \subset M \ ,
\end{equation}
where $M$ is a spacetime and each $A_o$ is a *-algebra of operators on a fixed Hilbert space.
%
%
The idea is that $A_o$ is generated by observable field operators localized in the region $o \subset M$,
that for convenience is chosen in a suitable base $\Delta$.
Thus the spacetime regions $o \in \Delta$ take the role of the points, 
and $\Delta$  -- that is a poset under the inclusion relation -- replaces $M$
for what concerns the geometric aspects.

This led to consider notions for posets with a geometric flavour, such as the fundamental group \cite{Ruz05}, cohomology \cite{RR06,RRV07},
bundles \cite{RR06,RRV07,RRV08}, and their connections \cite{RR06}. 
These combinatorial notions yield truly geometric invariants when applied to a "good"
{\footnote{That is, a base of arcwise and simply connected subsets.}}
base of a manifold, as the fundamental group \cite{Ruz05}, 
${\mathrm{1}}^{\mathrm{st}}$ locally constant \cite{RRV07} and de Rham \cite{RRV08} cohomology,
the category of flat bundles \cite{RRV07,RRV08}.

These ideas have been applied to quantum field theory, for aspects concerning gauge theories \cite{CRV1,CRV2}.
Anyway a further topic is the analysis of \emph{sectors}, that are the physically relevant Hilbert space representations of $\mA$.
In the case of the Minkowski spacetime this question led
to the characterization of the set of sectors (\emph{superselection structure}) as a symmetric tensor category with conjugates 
(\emph{DR-category}, after S. Doplicher and J.E. Roberts). 
%
%
It is a pivotal result that any DR-category is isomorphic to the category of representations (\emph{dual}) of a compact group,
that is interpreted as the global gauge group in the case of a superselection structure \cite{DR89,DR90}.

The generalization of this schema to generic curved spacetimes is still an open problem enriched with interesting geometric aspects.
The most recent proposal of superselection structure is the one given by R. Brunetti and G. Ruzzi \cite{BR08}, 
and is defined as the set $Z^1(\mA)$ of 1-cocycles with values in the precosheaf of unitary operators in $\mA$. 
Instead of Hilbert space representations, $Z^1(\mA)$ defines representations of $\mA$ on Hilbert bundles over $\Delta$
that correspond to flat Hilbert bundles over $M$ \cite{RV11,RV14a,RV14b}. 
This approach has been proposed to describe situations in which the geometry of $M$ interacts with the physics described by $\mA$: 
the paradigmatic example is the Aharonov-Bohm effect, where the monodromy of a classical electromagnetic potential, closed as a de Rham 1-form, 
appears as a phase factor affecting the wave functions of charged quantum particles. 
The fact that Aharonov-Bohm type effects can be really described by representations of $\mA$ on Hilbert bundles over $\Delta$ has been proved in \cite[Sect.4.2]{VasQFT}
{\footnote{A complete discussion of these representations from the point of view of the theory of sectors is the object
           of a work in progress by C. Dappiaggi, G. Ruzzi and the author.}}.

In the same paper it is proved that $Z^1(\mA)$ can be represented as the category of sections of a presheaf $\mS(\mA)$ of DR-categories (\emph{DR-presheaf}) over $\Delta$.
The advantage of this representation is that each DR-category $S_o$ of the presheaf can be interpreted as a "local" superselection structure for $A_o$, $o \in \Delta$.
This leads to the mathematical question of finding the right dual object playing the role of the dual compact group of a single DR-category.
We already proved in a previous work that, if a "coherent" embedding of a DR-presheaf into the category of Hilbert spaces is given, 
then the category of sections is isomorphic to the one of Hilbert $\mG$-bundles,
where $\mG$ is a group bundle over $\Delta$ \cite{Vas12}.  
Yet $\mG$ depends on the embedding, and embeddings may be many or do not exist.
Thus the question of the search of the canonical dual object arises.

In the present paper we prove that the dual object of the category of sections of a DR-presheaf is a group gerbe over $\Delta$:
roughly speaking, a "twisted group bundle", whose obstruction to be an actual bundle is measured by a non-abelian cocycle.
Any bundle over $\Delta$ is a particular kind of gerbe, and non-isomorphic bundles may become isomorphic when regarded as gerbes,
because in this setting we have a more general notion of morphism.
This fact is the dual counterpart of the multiplicity of embeddings for DR-presheaves.

Our approach is intentionally oversized with respect to the physical scenario,
as we expect that the dual gerbe of $Z^1(\mA)$ is a trivial group bundle \S \ref{sec.D4}.
Yet our level of generality brings into the light the rich mathematical structures of DR-presheaves,
where 2--groups, gerbes, non-abelian cocycles and twisted $C^*$-dynamical systems appear. 
These structures are the object of the present investigation.

In the sequel we illustrate the organization of the present paper and the main results.

\

In Section \ref{sec.A} we recall some notions.
We start with \emph{2--groups}, that we adopt in the form of crossed modules: these are pairs, 
denoted by ${}^2G = (G \tto N)$, of group morphisms
\[
i : G \to N \ \ \ , \ \ \ \beta : N \to \Aut G
\]
fulfilling certain equivariance conditions. 
Our notation for $N$ is not so standard: 
in the present paper, it indicates that $N$ is often defined as the normalizer of $i(G)$ in some bigger group.

\noindent 
Afterwards we recall some definitions in geometry of posets: 
paths (and loops in particular), the fundamental group, cohomology, bundles.
In particular, bundles over posets are defined as precosheaves such that the inclusion morphisms are isomorphisms;
in the present paper we shall consider bundles of $C^*$-algebras, Hilbert spaces and groups.

\

In Section \ref{sec.B} we define a non-abelian cohomology for the poset $\Delta$ with coefficients in a generic 2--group 
${}^2G = (G \tto N)$.
%
%
This generalizes the one in \cite{RR06} in which $N$ is the group of inner automorphisms of $G$.
Given a non-abelian ${}^2G$-cocycle $q$ over $\Delta$ and a category $C$ of $G$-spaces 
($G$-$C^*$-dynamical systems, $G$-Hilbert spaces, ...), we introduce the notion of \emph{gerbe} over $\Delta$: 
a family of objects of $C$ endowed with inclusion morphisms of the type (\ref{eq.intro.1}),
whose obstruction to fulfill the usual precosheaf relations 
%
%
is encoded by $q$.
%
%
When the inclusion morphisms are invertible we use the term \emph{bundle gerbe}; 
when $q$ is trivial, a bundle gerbe is nothing but a bundle over $\Delta$.
We derive the twisted holonomy of a gerbe, that defines a non-abelian cocycle over the fundamental group $\pi_1(\Delta)$
describing the parallel transport over loops in $\Delta$ (Prop.\ref{prop.B3.1} and Eq.\ref{eq.B3.3}).
%
%
Finally, the relation between gerbes over posets and gerbes over spaces is briefly discussed, Eq.\ref{eq.comp.03}.

\

Section \ref{sec.C} concerns twisted $C^*$-dynamical systems and their relation with $C^*$-gerbes. 
Given a group $\Pi$, a $C^*$-algebra $B$ and a 2--group ${}^2G = (G \tto N)$ with $N \subseteq \Aut B$,
in our approach a twisted $C^*$-dynamical system is defined as a non-abelian ${}^2G$-cocycle over $\Pi$, see \S \ref{sec.C3}.
Thus twisted $C^*$-dynamical systems in the standard sense \cite{ZM68,BS70,PR89} arise
when $G = U(B)$, the unitary group of $B$, and $N = \Aut B$.
We study the basic properties of our twisted $C^*$-dynamical systems, in particular their equivariant representations
that in our setting involve representations of $G$ besides twisted representations of $\Pi$.
A crossed product is constructed, Theorem \ref{thm.cp}.

\noindent
On the other side, $C^*$-gerbes pop up in a natural way.
If a $C^*$-precosheaf $\mA$ is given and we know that every $C^*$-algebra of the precosheaf $A_o$, $o \in \Delta$,
is the fixed-point algebra under a $G$-action on a $C^*$-algebra $F_o$,
then we may ask whether we can arrange the family $F = \{ F_o \}_{o \in \Delta}$ to get a $C^*$-precosheaf.
This problem may have many inequivalent solutions or have no solution at all \cite{Vas12},
but both these situations can be encoded by a $C^*$-gerbe 
$\check{\mF}$,
isomorphic to any (when existing) $C^*$-precosheaf with fixed-point sub-precosheaf $\mA$.
This is the case of the Haag-Kastler precosheaf defined by the free Dirac field,
where solutions correspond to background, closed electromagnetic potentials, \S \ref{sec.C1}.

\noindent Now, in \cite{RV11} it is proved that $C^*$-bundles over $\Delta$ form a category equivalent to the one of
$C^*$-dynamical systems over $\pi_1(\Delta)$, and $C^*$-dynamical systems are a well-developed machinery \cite{Ped,Wil}. 
The geometric picture of this construction is that
the $*$-automorphisms of the $\pi_1(\Delta)$-action are the parallel transport around loops in the "base space" $\Delta$.
We prove analogous results for $C^*$-gerbes,
assigning a twisted $C^*$-dynamical system to any bundle $C^*$-gerbe (Theorem \ref{thm.C4.1}).
%
%

\noindent We note that, using the equivalence with the category of flat bundles,
we can recast these ideas outside the scenario of posets. 
For example, by Prop.\ref{prop.C1.1} and Theorem \ref{thm.C4.1}
any flat $C^*$-bundle $A \to M$ with fibre a fixed-point algebra $A_\bullet = F_\bullet^G$ 
and structure group $N/G$, where $N \subseteq \Aut F_\bullet$ is in the normalizer of $G$,
defines a twisted action of $\pi_1(M)$ on $F_\bullet$.

\

In Section \ref{sec.D} we prove our duality theorems for DR-presheaves.
To this end we make some preliminary remarks:
first, given a 2-group ${}^2G$ and a ${}^2G$-cocycle $q$ over $\Delta$, 
there is a standard group gerbe $\check{\mG}_q$ associated to $q$, Example \ref{ex.Gadj};
secondly, the category $\Hilb_q(\Delta)$ of Hilbert gerbes with twisting $q$ 
can be regarded in a natural way as the dual of $\check{\mG}_q$, Example \ref{ex.Hilb}.
This allows us to prove the following result, Theorem \ref{thm.D2.1}:
given a DR-presheaf $\mS$ over $\Delta$, 
there are a 2-group ${}^2G$ and a ${}^2G$-cocycle $q$ over $\Delta$
such that the category of sections of $\mS$ embeds in $\Hilb_q(\Delta)$.
All the involved objects, ${}^2G$, $q$ and $\check{\mG}_q$, are unique up to isomorphism.
When the restriction morphisms of $\mS$ are isomorphisms, $q$ is a complete invariant, Theorem \ref{cor.ThmEmb.1}.

\noindent We also prove, Prop.\ref{prop.DR.tw}, that any section $\varrho$ of $\mS$ defines a twisted $\pi_1(\Delta)$-action on the Cuntz algebra $\mO_d$,
where $d \in \bN$ is the dimension of $\varrho$ \cite[\S 2]{DR89}.

\noindent Finally we apply our construction to the DR-presheaf $\mS(\mA)$ defined by a Haag-Kastler precosheaf (\ref{eq.intro.1}),
assigning to it a "gauge" 2-group gerbe ${}^2\check{\mG}(\mA)$, Theorem \ref{thm.HK}.
We give a qualitative argument to explain why we expect that ${}^2\check{\mG}(\mA)$ should collapse to a trivial group bundle.
The fact that things might really go this way is object of a work in progress, in which $\mS(\mA)$ plays a crucial role \cite[\S 5]{VasQFT}.


\section{Background notions}
\label{sec.A}

\paragraph{Some conventions.}
Several kinds of groups and algebras shall be considered. 
Unless otherwise stated we write $1$ for the identity of a group,
as for example the unitary group of a Hilbert space.
When $R$ is an algebra ($C^*$-algebra, Von Neumann algebra), we write in the same way $1 \in R$ for the unit.
No confusion should arise, since the context will be clarifying.
Moreover, we use the symbol $id$ for the identity map of a set.

With the term \emph{topological group} we always mean a Hausdorff topological group which,
unless otherwise stated, is assumed to be locally compact.
Morphisms between topological groups (in particular, unitary representations) are always assumed to be continuous.

\paragraph{Operator algebras.}
Some notions on Hilbert space representations of topological groups, 
$C^*$-algebras, Von Neumann algebras and $C^*$-dynamical systems are assumed \cite{Ped,Wil}.
Here we just give the following terminologies.
A \emph{$C^*$-dynamical system} is given by a pointwise continuous action $\alpha : G \to \Aut A$,
where $A$ is a $C^*$-algebra and $G$ a topological group
{\footnote{
Throughout this paper, $\Aut A$ shall always be endowed with the pointwise convergence topology.
}}.
When we want to emphasize the role of $G$, we say that $A$ is a \emph{$G$-dynamical system}.
Analogously, given a unitary representation $\pi : G \to U(H)$, we say that $H$ is a \emph{$G$-Hilbert space}.
We assume that *-morphisms between unital $C^*$ and Von Neumann algebras are unital.

The notion of \emph{multiplier $C^*$-algebra} $M(B)$ of a $C^*$-algebra $B$ is used \cite[\S 3.12]{Ped}.
We write $U(B)$ for the unitary group of $M(B)$; since $M(B) = B$ for $B$ unital, no confusion should arise.

\paragraph{Categories.}
The elementary notions that we use (\emph{categories}, \emph{functors}, \emph{natural transformations})
can be found in \cite[\S I.1-4]{ML}.
Some terminologies follow.
%
%
Given the topological group $G$, a \emph{$G$-category} $C$ is a category having objects topological spaces carrying a $G$-action
(\emph{$G$-spaces}) and arrows $G$-equivariant continuous maps (\emph{$G$-morphisms}). Given $\rho,\sigma \in \obj C$,
we denote the set of their $G$-morphisms by $(\rho,\sigma)$. The examples that we have in mind are given by
the category $\Calg_G$ of $G$-dynamical systems with arrows $G$-equivariant *-morphisms, and
the category $\wa{G}$ of $G$-Hilbert spaces with arrows $G$-equivariant linear operators.
If $\pi : G \to U(H)$, $\eta : G \to U(K)$ are unitary representations (that is, $H,K \in \obj \wa{G}$),
then we write
\begin{equation}
\label{def.iG}
(\pi,\eta) \, := \, \{ t : H \to K \, : \, \eta(g) \circ t = t \circ \pi(g) \, , \, \forall g \in G  \} \ .
\end{equation}
The definition of \emph{DR-category} (symmetric tensor $C^*$-category with conjugates, having simple identity object)
can be found in \cite[\S 1]{DR89}. Here we just need a quick description: a DR-category $T$ is a category where
the spaces of arrows $(\rho,\sigma)$ are complex Banach spaces, endowed with:
(1) an antilinear involution $*$ fulfilling $\| t^* \circ t \| = \| t \|^2,$ for all arrows $t$;
(2) a tensor product whose commutativity is described by a unitary natural transformation, the symmetry;
(3) an identity $\iota \in \obj T$ for the tensor product with $(\iota,\iota) \simeq \bC$;
(4) conjugate objects $\bar{\rho}$, $\forall \rho \in \obj T$.

The prototype of DR-category is $\wa{G}$, where $G$ is a compact group.
Here the tensor product is the usual one 
$H,K \mapsto H \otimes K$
defined for $G$-Hilbert spaces, the symmetry is defined by the flip operators
$\theta_{HK}(v\otimes w) := w \otimes v$, $\forall v \in H$, $w \in K$,
and the conjugate is given by the conjugate Hilbert space $\bar{H}$ carrying the conjugate representation.
There are other examples of DR-categories: we just mention semigroups of
symmetric *-endomorphisms of $C^*$-algebras with trivial centre \cite{DR89A}.


\subsection{2--Groups}
\label{sec.A0}
A \emph{2--group} (that we present in the guise of a \emph{crossed module}) is given by a pair of group morphisms 
$i : G \to N$, $\beta : N \to \Aut G$,
such that
\[
i \circ \beta_u  = \ad u \circ i
\ \ \ , \ \ \
\beta_{i(g)} \circ i = i \circ \ad g
\ \ \ \ , \
\forall u \in N \, , \, g \in G \ .
\]
To be concise, sometimes in the sequel we will use the notations
\[
\unl g := i(g) \in N
\ \ \ \ , \ \ \ \ 
\wa u(g) := \beta_u(g) \in G
\ .
\]
A 2--group is denoted by 
${}^2G \equiv ( G \tto N )$,
often leaving the maps $i$, $\beta$ implicit.
We say that ${}^2G$ is \emph{topological} whenever $G$ and $N$ are topological groups (and $i,\beta$ continuous);
in this case ${}^2G$ defines the $C^*$-dynamical system $\beta_* : N \to \Aut C^*G$ 
(about the \emph{group $C^*$-algebra} $C^*G$ see \cite[\S 7.1]{Ped}).

\begin{ex}
Given a group $G$ we have the 2--group $G \tto \Aut G$, where the map from $G$ to $\Aut G$
is given by the adjoint action. In the sequel we shall write 
${}^2\Aut G := (G \tto \Aut G)$.
\end{ex}

\begin{defn}
A \textbf{2--group morphism}
$\eta : (G' \tto N') \to {}^2G$
is defined by a pair $\eta = (\eta_0,\eta_1)$ of morphisms
$\eta_0 : G' \to G$, $\eta_1 : N' \to N$,
equivariant in the sense that
\[
\eta_1 \circ i \, = \, i' \circ \eta_0
\ \ , \ \ 
\eta_0(\wa{u'}(g')) \, = \, \wa{\eta_1(u')} (\eta_0(g'))
\ \ , \ \ 
\forall g' \in G' \, , \, u' \in N' \ .
\]
\end{defn}
Automorphisms of ${}^2G$ form the group $\Aut {}^2G$.
Sometimes in the sequel we will drop the zero and one symbols, writing $\eta(u)$, $\eta(g)$.
Note that any $u \in N$ defines the adjoint 2--group automorphism 
$\ad u(g) := \wa{u}(g)$, $\ad u(v) = uvu^{-1}$, $g \in G$, $v \in N$.
%
Thus we get actions
\begin{equation}
\label{eq.2G.1}
\ad : N \to \Aut {}^2G
\ \ \ , \ \ \
\ad \circ i : G \to \Aut {}^2G \ .
\end{equation}

When $G$ is a normal subgroup of $N$, the inclusion and the adjoint action yield the 2--group $G \tto N$.
In particular, when $G=N$ and $i : G \to N$ is the identity we write
\begin{equation}
\label{eq.1G}
{}^1G \, := \, ( G \tto G ) \ .
\end{equation}
More in general, given any 2--group $G \tto N$ we have that $i(G)$ is a normal subgroup of $N$.
To be concise, we write
\begin{equation}
\label{eq.mod}
N/G \, := \, N/i(G) 
\ \ \ \ {\mathrm{and}} \ \ \ \
u_G := u \, {\mathrm{mod}} \, i(G) \in N/G  \, , \, \forall u \in N \ .
\end{equation}
We write $\tbe := 1_G$ for the identity of $N/G$.

\begin{rem}
Let $A$ be a $C^*$-algebra. Then any unitary $u \in U(A)$ defines the inner
automorphism $\ad u \in \Aut A$. Since inner automorphisms form a normal subgroup of $\Aut A$,
we get the 2--group ${}^2\Aut A := (U(A) \tto \Aut A)$.
Let $\alpha : G \to \Aut A$ be a $C^*$-dynamical system and
$N_\alpha A$ denote the normalizer of $\alpha(G)$ in $\Aut A$.
Then we have the \emph{normalizer 2--group}
\[
{}^2G_\alpha := (\alpha(G) \tto N_\alpha A) \ .
\]
Let ${}^2G := (G \tto N)$ be a 2--group. A morphism $\alpha : {}^2G \to {}^2\Aut A$
is a \emph{Green $C^*$-dynamical system} in the terminology of \cite[\S 7.4.1]{Wil}.
\end{rem}

\subsection{Geometry of posets}
\label{sec.A1}

A \emph{poset} is given by a set $\Delta$ endowed with an order relation $\leq$.
%
%
We will often denote a poset simply by $\Delta$, leaving the symbol $\leq$ implicit.
Examples of posets arise from topology: 
any family of open, non-empty proper subsets of a topological space, ordered under inclusion, defines a poset.
%
%
%

\paragraph{Simplicial sets.}
An important invariant of $\Delta$ is the \emph{simplicial set} $\Sigma_*(\Delta)$ (\cite[\S 2]{RR06} and \cite[\S 3.1]{RRV07}).
We comment the first three degrees and simply define the others.
\begin{itemize}
\item	$\Sigma_0(\Delta) := \Delta$. With a $na\ddot{\imath}f$ geometric idea, we think $\Sigma_0(\Delta)$ as a "space".
\item	By definition, elements of $\Sigma_1(\Delta)$ are triples
	\[
	b \, = \, (\partial_0b,\partial_1b,|b| \in \Delta)
	\ \ : \ \ 
	\partial_1b , \partial_0b \leq |b| \ .
	\]
	The idea is that $b$ is a "line" starting at the "point" $\partial_1b$ and ending at $\partial_0b$.
	The previous notation automatically yields maps $\partial_* : \Sigma_1(\Delta) \to \Sigma_0(\Delta)$, $*=0,1$.
\item	The set $\Sigma_2(\Delta)$ of "triangles" has elements quadruples
	\[
	c \, = \, ( \co , \cl , \cz \in \Sigma_1(\Delta) \, , \, |c| \in \Sigma_0(\Delta) ) \ ,
	\]
	fulfilling the properties
	$\partial_{hk}c = \partial_{k \, h+1}c$, $|\partial_hc| \leq |c|$, $\forall h \geq k \in \{ 0,1,2 \}$,
	where $\partial_{hk} := \partial_h \circ \partial_k$ are the compositions of the boundary maps.
	Here the idea is that the sides of the triangle $c$ are given by $\co$, $\cl$ and $\cz \in \Sigma_1(\Delta)$,
	and the above relations identify the common vertices $\partial_{hk}c$.
\end{itemize}
Inductively, at higher degrees a $n$-simplex is a $n+2$-tuple $d = ( \bo d , \ldots , \partial_nd , |d| )$,
where $\partial_kd \in \Sigma_{n-1}(\Delta)$, $k=0,\ldots,n$, $|d| \in \Delta$, fulfil
\begin{equation}
\label{eq.A3}
|\partial_id| \leq |d| \ , \ \partial_{hk}d = \partial_{k,h+1}d
\  \, \ \ 
i,h = 0,\ldots,n \, , \, k \leq h \ .
\end{equation}
There is a further simplicial set $N_*(\Delta)$, called the \emph{nerve} of $\Delta$ (\cite[\S 3.1]{RRV07}, \cite[\S A.2]{RRV08}),
that turns out to be a simplicial subset of $\Sigma_*(\Delta)$. 
We give the definition of the first three degrees: $N_0(\Delta) := \Delta$, whilst
\[
N_1(\Delta) \ := \ \{  b = (\partial_1b,\partial_0b) \, : \, \partial_1b \leq \partial_0b   \} \ .
\]
Note that $N_1(\Delta)$ encodes the order relation of $\Delta$, and any $b \in N_1(\Delta)$ can be regarded as the element
$b' := ( \partial_0b' := \partial_0b , \partial_1b' := \partial_1b , |b'| := \partial_0b  )$
of $\Sigma_1(\Delta)$. Finally, $N_2(\Delta) \subseteq \Sigma_2(\Delta)$ is given by triples
$c = ( \cl \leq \cz \leq \co \in \Delta)$.

\paragraph{Paths and homotopy.}
A \emph{path} is a finite sequence 
$p =  b_n * \ldots * b_1$
of elements of $\Sigma_1(\Delta)$ such that
$\partial_0b_k = \partial_1b_{k+1}$, $\forall k = 1 , \ldots , n-1$.
If $a = \partial_1b_1 \in \Delta$ and $o = \partial_0b_n \in \Delta$, then we write
\[
p : a \to o \ .
\]
For simplicity, we shall always assume that $\Delta$ is \emph{connected}, that is, 
that for any pair $o,a \in \Delta$ there is a path $p : a \to o$.
In particular, when $a=o$ we say that $p$ is a \emph{loop}.
Paths can be composed in a natural way: if $p' : o \to e$, $p' = b'_m * \ldots * b'_1$, starts where $p$ finishes, then we define
\[
p'*p \, := \, b'_m * \ldots * b'_1 * b_n * \ldots * b_1  \ .
\]
The \emph{opposite} of $p$ is given by the path $\ovl{p} : o \to a$,
$\ovl{p} \, := \, \ovl{b}_1 * \ldots * \ovl{b}_n$,
where
\[
\ovl{b}_k \, := \, ( \partial_0\ovl{b}_k := \partial_1b_k \, , \, \partial_1\ovl{b}_k := \partial_0b_k \, , \, |\ovl{b}_k| := |b_k| )
\ \ , \ \ \ \
\forall k = 1 , \ldots , n-1 \, .
\]
Let $a \in \Delta$. A \emph{path frame} is a family 
\begin{equation}
\label{eq.pf}
p_a \, = \, \{ p_{oa} : a \to o \}_o \ ,
\end{equation}
such that $p_{aa} = \iota_a := (a,a,a) \in \Sigma_1(\Delta)$, the \emph{constant path} based on $a$.
To be concise, in the sequel we shall write $p_{ao} := \ovl{p}_{oa} : o \to a$.
Path frames allow one to relate loop compositions with the order relation of $\Delta$, by defining
\begin{equation}
\label{eq.loops}
\ell_{\omega o} \, := \, p_{a\omega } * (\omega o) * p_{oa} : a \to a 
\ \ \ \Rightarrow \ \ \
\ell_{\xi \omega } * \ell_{\omega o} = \ell_{\xi o} \ , \ \forall o \leq \omega  \leq \xi  \ .
\end{equation}
There is an equivalence relation $\sim$ on the set of paths, called \emph{homotopy} \cite[\S 2.2]{Ruz05},
%
%
%
modeled on the case in which $\Delta$ is a base of simply connected open sets.
%
%
The set of loops is a semigroup when endowed with the composition operation, 
and the quotient under the homotopy relation is a group
\[
\pi_1^a(\Delta) \ ,
\]
called the \emph{homotopy group of $\Delta$ based on $a$}. 
It turns out that $[p]^{-1} = [\ovl{p}]$, where $[p] \in \pi_1^a(\Delta)$ denotes the homotopy class of the loop $p : a \to a$.
The identity $1 \in \pi_1^a(\Delta)$ is given by the class of the constant path $\iota_a$.
The isomorphism class of $\pi_1^a(\Delta)$ is independent of $a \in \Delta$, 
thus when $a$ will not play a role we shall write $\pi_1(\Delta)$ for $\pi_1^a(\Delta)$.
%

\paragraph{Cohomology.}
Let $N$ be a group with unit $1$. A \emph{1-cocycle} is a map
\begin{equation}
\label{eq.A4}
u : \Sigma_1(\Delta) \to N
\ \ \ {\mathrm{:}} \ \ \
u_\cl \, = \, u_\co \, u_\cz  \ \ , \ \ \ \ \forall c \in \Sigma_2(\Delta) \ .
\end{equation}
When $b=(o,\omega ) \in N_1(\Delta)$ we write $u_b \equiv u_{\omega o}$.
We denote the set of cocycles by $Z^1(\Delta,N)$. A 1-cocycle $u'$ is said to be \emph{cohomologous} to $u$
if there is $v : \Sigma_0(\Delta) \to N$ such that
\begin{equation}
\label{eq.A5}
v_\bo \, u'_b \, v_\bl^{-1} \ = \  u_b 
\ \ , \ \ \ \ 
\forall b \in \Sigma_1(\Delta) \ .
\end{equation}
In this case we write
$v \in (u,u')$.
We denote the set of cohomology classes by $H^1(\Delta,N)$.  
It is an abelian group when $N$ is abelian, but in general it is a set with distinguished element
the cohomology class of the \emph{trivial cocycle} ${\bf 1}_b :=1$, $b \in \Sigma_1(\Delta)$. 
There is an isomorphism
\begin{equation}
\label{eq.A6}
H^1(\Delta,N) \, \simeq \, H^1(\pi_1(\Delta),N) \ ,
\end{equation}
where, for any group $\Pi$, by $H^1(\Pi,N)$ we mean the quotient of the set of morphisms $\phi : \Pi \to N$ 
by the equivalence relation $\phi \simeq \ad v \circ \phi$, $v \in N$ \cite[Prop.3.8]{RR06}.

Given an abelian group $A$ with unit $1$, we can define cohomology groups $H^n(\Delta,A)$ in higher degrees $n \in \bN$,
in the following way. We consider maps $u : \Sigma_n(\Delta) \to A$ fulfilling the condition
$u(\partial_0k) u(\partial_1k)^{-1}  u(\partial_2k) \cdots = 1$, $\forall k \in \Sigma_{n+1}(\Delta)$.
These form the group $Z^n(\Delta,A)$ of $n$-cocycles.
Then we define $H^n(\Delta,A)$ as the quotient under the subgroup of $n$-coboundaries
$dw(h) := w(\partial_0h) w(\partial_1h)^{-1}  w(\partial_2h) \cdots $, $\forall h \in \Sigma_n(\Delta)$,
where $w$ is any map $w : \Sigma_{n-1}(\Delta) \to A$.

\paragraph{Posets and topology.}
Let $M$ be a Hausdorff space and $\Delta$ a base, ordered under inclusion, of arcwise and simply connected proper subsets generating the topology of $M$.
In this case we say that $\Delta$ \emph{is a good base for} $M$.
By \cite[Theorem 2.18]{Ruz05}, given $a \in \Delta$ and $x \in a$ there is an isomorphism of fundamental groups
\begin{equation}
\label{eq.A1}
\pi_1^a(\Delta) \to \pi_1(M) \ \ , \ \  [p] \mapsto [p_{top}] \ ,
\end{equation}
see also \cite[\S II.13]{BT}.
The proof is intuitive: the closed curve $p_{top} : [0,1] \to M$, $p_{top}(0) = x$, is constructed in such a way
that its support is contained in $\cup_n |b_n| \subseteq M$, where $p = b_n * \ldots * b_1$,
and the homotopy class $[p_{top}] \in \pi_1(M)$ can be shown to be independent of the choice of $p$ in the homotopy class 
$[p] \in \pi_1^a(\Delta)$. 
As a consequence, one can prove that if $M$ is a manifold then $H^1(\Delta,\bR) \simeq H_{dR}^1(M)$ \cite[Cor.3.9]{RRV08},
where $H_{dR}$ denotes de Rham cohomology.

\subsection{Precosheaves and bundles}
\label{sec.A3}
Let ${}^2G = (G \tto N)$ be a 2-group, $u \in Z^1(\Delta,N)$ and $C$ a $G$-category.
A \emph{$u$-precosheaf}, written $\mB = (B,\jmath)_\Delta$, is given by a family 
$\jmath_{\omega o} :  B_o \to B_\omega $, $o\leq \omega  \in \Delta$,
of morphisms between objects of $C$, such that
\begin{equation}
\label{eq.A7}
\jmath_{\xi o} \, = \, \jmath_{\xi \omega } \circ \jmath_{\omega o} 
\ \ \ , \ \ \
\jmath_{\omega o} \circ \alpha_o(g) \ = \ \alpha_\omega (\wa{u}_{\omega o}(g)) \circ \jmath_{\omega o}
\ ,
\end{equation}
for all $\forall o \leq \omega  \leq \xi  \in \Delta$ and $g \in G$.
Here we explicitly wrote the actions
\[
\alpha_o : G \to \Aut B_o \ \ , \ \ \forall o \in \Delta \ .
\]
We call $B$ the \emph{set of fibres} and $\jmath$ the \emph{precosheaf structure} that is formed by the \emph{inclusion morphisms} $\jmath_{\omega o}$.
The first equations in (\ref{eq.A7}) are called the \emph{precosheaf relations}. 
When $u$ is trivial $N$ plays no role and each $\jmath_{\omega o}$ is a $G$-morphism, thus we use the term \emph{$G$-precosheaf}.
When $G$ is trivial, we simply use the term \emph{precosheaf}.
A \emph{morphism} $\eta : \mB \to \mB' = (B',{\jmath \, }')_\Delta$ of $u$-precosheaves is given by a family of $G$-morphisms 
$\eta_o : B_o \to B'_o$, $o \in \Delta$,
such that 
\begin{equation}
\label{eq.A7a}
\eta_\omega  \circ \jmath_{\omega o}  \, = \, {\jmath \, }'_{\omega o} \circ \eta_o 
\ \ , \ \ 
\forall o \leq \omega 
\ .
\end{equation}
We say that $\eta$ is an isomorphism whenever each $\eta_o$ is an isomorphism.

A remark on the notation: we may rewrite
$j_{\omega o}  \, \equiv \, j_b$,
where $b := (o\leq \omega ) \in N_1(\Delta)$. In this way, the first of (\ref{eq.A7}) becomes
$\jmath_\cl  =  \jmath_\co \circ \jmath_\cz$, $c \in N_2(\Delta)$.
\begin{ex}
Let $G$ be a group and $M$ a space with trivial $G$-action having a good base $\Delta$. 
If $E \to M$ is a $G$-vector bundle then for any $o \in \Delta$ we define
$SE_o$
as the vector space of continuous sections $s : M \to E$ vanishing in $M \setminus o$.
For any $o \subseteq \omega $ we have the inclusion morphism $\jmath_{\omega o} : SE_o \to SE_\omega $
and (\ref{eq.A7}) are obviously fulfilled. Thus $\mS\mE = (SE,\jmath)_\Delta$ is a $G$-precosheaf.
If $f : E \to E'$ is a $G$-morphism, then a morphism $f_* : \mS\mE \to \mS\mE'$ is induced,
$f_{*,o}(s) := f \circ s$, $\forall s \in SE_o$, $o \in \Delta$.
\end{ex}

\paragraph{Bundles.}
We call \emph{$u$-bundles} those $u$-precosheaves $\mB = (B,\jmath)_\Delta$ such that any $j_{\omega o}$ is an isomorphism.
In this case we can define 
$j_b := j_{|b|\bo}^{-1} \circ \jmath_{|b|\bl}$, $\jmath_b : B_\bl \to B_\bo$,
for arbitrary $b \in \Sigma_1(\Delta)$. Thus, given $a \in \Delta$ and a loop $p : a \to a$, $p = b_n * \ldots * b_1$,
by composition we get an automorphism
$\jmath_p : B_a \to B_a$.
These automorphisms are independent of the homotopy class of $p$, and defining $B_* := B_a$ we obtain an action
\begin{equation}
\label{eq.A8}
\jmath_* : \pi_1^a(\Delta) \to \Aut B_* \ \ , \ \ [p] \mapsto \jmath_p \ .
\end{equation}
This action characterizes $\mB$ up to isomorphism: the set $\bun(\Delta,B_*)$ of isomorphism classes of bundles
with fibres isomorphic to $B_*$ is isomorphic to $H^1( \pi_1(\Delta) , \Aut B_* )$ \cite[Th.3.2]{Vas12}.

\begin{ex}[Group bundles defined by cocycles]
\label{ex.Gb}
Let ${}^2G_\bullet = (G_\bullet \tto N)$ be a 2-group and $u \in Z^1(\Delta,N)$.
We define $G_o := G_\bullet$, $\alpha_o(g) := \ad g$ and $\jmath_{\omega o} := \wa{u}_{\omega o}$ for all $g \in G_\bullet$, $o \leq \omega  \in \Delta$.
It is easily seen that $\mG_u := (G,\jmath)_\Delta$ is a group $u$-bundle.
\end{ex}

\begin{rem}[Choice of the standard fibre]
\label{rem.standardfibre}
Any bundle $\mB = (B,\jmath)_\Delta$ is isomorphic to a bundle $\mB' = (B',\jmath^{ \,'})_\Delta$
where any $B'_o$ equals a fixed object $F$, called \emph{the standard fibre of $\mB$} 
(see \cite[\S 2.2]{VasQFT}, $F = B_a$ for a fixed $a \in \Delta$).
We call $\mB'$ a \emph{choice of the standard fibre} of $\mB$, and use the term \emph{$F$-bundle} for $\mB$.
\end{rem}

\paragraph{$C^*$-precosheaves.}
When $C$ is the category of $G$-dynamical systems we make use of the term \emph{$C^*$-precosheaf}. 
Some classes of examples follow.
\begin{itemize}
\item Let $A$ be a $C^*$-algebra. Then the set $\Delta$ of (closed, two-sided, proper, non-trivial) ideals of $A$
      forms a poset under the inclusion relation and generates the topology of the Jacobson spectrum of $A$ \cite[\S 4.1]{Ped}. 
      For any $o \in \Delta$ we (tautologically) denote the corresponding ideal by $A^\triangleleft_o$ and the inclusion *-morphisms by 
      $\jmath_{\omega o} : A_o \to A_\omega $, $o \subseteq \omega $.
      Thus the $C^*$-precosheaf $\mA^\triangleleft = (A^\triangleleft,\jmath)_\Delta$ comes canonically associated with $A$.
      In particular, if $X$ is a locally compact Hausdorff space with base $\Omega$ and $A$ is $C_0(X)$-algebra \cite[\S C.1]{Wil},
      then defining $A^X_U := C_0(U)A$, $U \in \Omega$, we get the $C^*$-precosheaf $\mA^X = (A^X,\jmath)_\Omega$.
      Similar arguments apply to the lattice of hereditary $C^*$-subalgebras of a $C^*$-algebra \cite{AB15}.
\item Let $E$ be a directed graph and $C^*E$ denote the graph $C^*$-algebra of $E$. It is well-known that $C^*E$ is a $\bU(1)$-dynamical system.
      We denote the poset of hereditary and saturated subsets of the set of vertices of $E$ by $\Delta$.
      Any $o \in \Delta$ generates a $\bU(1)$-stable ideal $C^*E^\triangleleft_o \subset C^*E$ \cite[Chap.2]{graphs}. 
      In this way, defining the inclusion *-morphisms 
      $\jmath_{\omega o} : C^*E^\triangleleft_o \to C^*E^\triangleleft_\omega$, $o \subseteq \omega$,
      we get the $\bU(1)$-precosheaf $\mC^*\mE^\triangleleft = ( C^*E^\triangleleft,\jmath)_\Delta$.
\item Let $M$ be a space with good base $\Delta$. Then the category of $C^*$-bundles over $\Delta$ is equivalent
      to those of $\pi_1(M)$-dynamical systems \cite{RV11} and of
      flat $C^*$-bundles over $M$ \cite[\S 2.2]{RV14a}.
\item Let $M$ be a globally hyperbolic spacetime and $\Delta$ denote the base of diamonds. 
      Then any field net $\mF$ with gauge group $G$  is a $G$-precosheaf over $\Delta$.
      The observable subnet $\mA$ defined by $G$-invariant operators in $\mF$ is a $C^*$-precosheaf \cite[\S 3]{GLRV01}.
\item Let ${}^2G = (G \tto N)$ be a 2-group, $u \in Z^1(\Delta,N)$ and 
      $\alpha : {}^2G \to {}^2\Aut A_*$ 
      a Green $C^*$-dynamical system \cite[\S 7.4.1]{Wil}. 
      Then defining $A_o := A_*$, $\alpha_o(g) := \alpha(\unl{g})$ and $\jmath_{\omega o} := \alpha(u_{\omega o})$ for all $g \in G$, $o \leq \omega  \in \Delta$,
      yields the $u$-bundle $\mA := (A,\jmath)_\Delta$.
\end{itemize}


\section{Gerbes over posets}
\label{sec.B}

In the present section we study gerbes.
As a first step we introduce a non-abelian cohomology for a poset $\Delta$ with coefficients in a 2-group.
Then we define gerbes over posets as a generalization of $G$-precosheaves.
Finally we study the holonomy of a gerbe which, instead of a $\pi_1(\Delta)$-action as in (\ref{eq.A8}),
yields a non-abelian $\pi_1(\Delta)$-cocycle.

\subsection{Non-abelian cocycles}
\label{sec.B1}

Let $\Delta$ be a poset and ${}^2G := (G \tto N)$ a 2--group.
A \emph{non-abelian 2-cocycle} on $\Delta$ with coefficients in ${}^2G$ is given by maps
$u : \Sigma_1(\Delta) \to N$, $g : \Sigma_2(\Delta) \to G$,
such that
\begin{equation}
\label{eq.B1.01}
\unl{g_c} \, = \, u_\co \, u_\cz \, u_\cl^{-1}
\ \ \ \ , \ \ \ \
g_{\partial_0d} \, g_{\partial_2d} \, = \, \wa{u}_{\partial_{01}d}(g_{\partial_3d}) \, g_{\partial_1d} \ ,
\end{equation}
for all $c \in \Sigma_2(\Delta)$ and $d \in \Sigma_3(\Delta)$.
We denote a non-abelian 2-cocycle by $q \equiv (u,g)$.
The \emph{trivial cocycle}, written ${\bf 1}$, is given by the constant maps with values the identities of $G$ and $N$.
We denote the set of non-abelian 2-cocycles by 
\[
\check{Z}^2(\Delta,{}^2G) \ .
\]
The terminology of $\check{Z}^2$ as a \emph{second} cohomology is not universally used, 
as some authors prefer to give emphasis to the 1-cochain $u$ and talk about non-abelian 1-cocycles \cite{BS}.
The notation we use here is coherent with \cite{RR06}, where the subject of non-abelian cohomology for posets
has been discussed in a systematic way.

Some remarks follow.
\begin{itemize}
\item When the map $g \mapsto \unl{g} \in N$, $g \in G$, is injective we may define $g_c$ by the first of (\ref{eq.B1.01}),
      and the second of (\ref{eq.B1.01}) is automatically verified.
\item Let $\eta : {}^2G' \to {}^2G$ be a 2--group morphism and $q'=(u',g') \in \check{Z}^2(\Delta,{}^2G')$.
      Then $\eta_*q' := (\eta_1 \circ u' , \eta_0 \circ g')$ is a non-abelian 2-cocycle 
      with coefficients in ${}^2G$.
\item Let $\Inn G$ denote the group of inner automorphisms of $G$ and ${}^2\Inn G =$ $(G \tto \Inn G)$ the 2--group defined by adjoint action.
      By \cite[Eq.19 and Eq.31]{RR06}, we have that our notion of 2--cocycle with coefficients in ${}^2\Inn G$ agrees with the one of \cite{RR06}.
\end{itemize}

\begin{defn}[Twisted connections]
\label{rem.tc}
Writing the first of (\ref{eq.B1.01}) for $c = ( \iota_a , \iota_a , \iota_a ; a )$, where $\iota_a := (a,a,a)$,
we conclude that $u_{\iota_a} = \unl{g_c} \in i(G)$ for all $a \in \Delta$.
We say that $q = (u,g)$ is \textbf{normalized} whenever $u_{\iota_a} = 1$ for all $a \in \Delta$. 
The set of normalized cocycles is denoted by 
\[
\check{Z}^2_1(\Delta,{}^2G) \ .
\]
Now, we say that $q = (u,g) \in \check{Z}^2(\Delta,{}^2G)$ is a \textbf{twisted connection} whenever
$u_{\ovl{b}} = u_b^{-1}$ for all $b \in \Sigma_1(\Delta)$.
We denote the set of twisted connections by 
\[
\check{U}^2(\Delta,{}^2G) \ .
\]
Since $\iota_a = \ovl{\iota_a}$ for all $a \in \Delta$, any twisted connection is normalized,
\emph{i.e.} $\check{U}^2(\Delta,{}^2G) \subseteq \check{Z}^2_1(\Delta,{}^2G)$.
Let now $u \in Z^1(\Delta,N)$; then $u$ defines the non-abelian 2-cocycle 
\[
du \, := \, (u,1) \ ,
\]
where $1$ is the constant map assigning the identity $1 \in G$. 
Since $u_{\ovl{b}} = u_b^{-1}$ for all $b$ (see the argument of \cite[Lemma 2.7]{Ruz05}, $du$ is a twisted connection.
\end{defn}

\paragraph{Non-abelian Cohomology.}
We say that $q=(u,g)$, $q'=(u',g')$ are \emph{cohomologous} whenever 
there is a pair $(v,h)$, where $v : \Sigma_0(\Delta) \to N$, $h : \Sigma_1(\Delta) \to G$,
such that
\begin{equation}
\label{eq.B1.01a}
u'_b \ = \  \unl{h_b} \, v_\bo \, u_b \, v_\bl^{-1}
\ \ \ \ , \ \ \ \
g'_c   \ = \  \wa{u}'_\co(h_\cz) \, h_\co \, \wa{v}_{\partial_{01}c}(g_c) \, h_\cl^{-1}
\ ,
\end{equation}
%
%
%
for all $b \in \Sigma_1(\Delta)$ and $c \in \Sigma_2(\Delta)$. In this case we write $(v,h) \in (q,q')$.
When the map $G \to N$ is injective, we may define $h$ in terms of $u$, $u'$ and $v$, and the second of (\ref{eq.B1.01a})
is automatically verified. 
%
%
Elementary computations show that if $q'' = (u'',g'')$ and $(v',h') \in (q',q'')$ then
\[
(1,1) \in (q,q)
\ \ \ , \ \ \ 
( v^{-1} \, , \, \wa{v}_{\partial_0}^{-1}(h) ) \in (q',q)
\ \ \ , \ \ \ 
( v'v \, , \, h' \wa{v}'_{\partial_0}(h) ) \in (q,q'') \ ,
\]
where $v_{\partial_0}$ stands for the mapping $b \mapsto v_\bo$.
Thus being cohomologous is an equivalence relation, and the corresponding set of equivalence classes, denoted by
\[
\check{H}^2(\Delta,{}^2G) \ ,
\]
is called the \emph{non-abelian cohomology of $\Delta$ with coefficients in ${}^2G$}.
We denote the class of $q=(u,g)$ in $\check{H}^2(\Delta,{}^2G)$ by $[q]=[u,g]$.

\begin{rem}
Let $q=(u,g) \in \check{Z}^2(\Delta, {}^2G)$. Since $u_{\iota_a} \in i(G)$ for all $a \in \Delta$,
we have that $q$ is cohomologous to a normalized cocycle.
In fact, we may apply to $q$ a pair $(1,h)$ such that $\unl{h_{\iota_a}} = u_{\iota_a}^{-1}$, $\forall a \in \Delta$,
obtaining the desired normalized cocycle.
\end{rem}


\

Let now $u,u' \in Z^1(\Delta,N)$ and $v \in (u,u')$. Then $(v,1) \in (du,du')$ and we obtain the map
\begin{equation}
\label{eq.B1.02}
d_* : H^1(\Delta,N) \to \check{H}^2(\Delta,{}^2G)
\ \ , \ \ 
d_*[u] := [du]
\ .
\end{equation}
%
%
%
In the following Lemma we will establish a relation between $\check{H}^2(\Delta,{}^2G)$ and $H^1(\Delta,N/G)$.
Before that we make some preparatory remarks, starting from the following notion.
\begin{defn}[Polarizations]
Let $X$ be a set with a period two bijection $x \mapsto x^c$ (so that $x^{c,c} = x$ for all $x \in X$).
A \textbf{polarization} of $X$ is given by a partition $X \, = \, X_+ \, \dot{\cup} \, X_0 \, \dot{\cup} \, X_-$ such that: \\
(1) $X_0 :=$ $\{ x \in X : x = x^c \}$;
(2) if $x \notin X_0$, then $x \in X_+$ if and only if $x^c \in X_-$.
\end{defn}
That polarizations exist it is easily seen by applying Zorn's Lemma to the projection of $X$ onto
the quotient $X_\simeq$ by the relation $x \simeq x'$ $\Leftrightarrow$ $x'=x$ or $x' = x^c$:
if $s : X_\simeq \to X$ is a section, then we may set
$X_+ := \{ s(y) , y \in X_\simeq \setminus (X_0)_\simeq \}$ and $X_- := X \setminus (X_+ \cup X_0)$.
In the next remark we apply the notion of polarization to sections of group morphisms.

\begin{rem}
\label{rem.section}
Let $\eta : N \to Q$ denote a group epimorphism. By Zorn's Lemma there is a unital section 
$s : Q \to N$, $\eta \circ s = id_Q$, $s(\tbe) = 1$.
In general $s$ is not a morphism, anyway we may always pick $s$ such that $s(\tbf^{-1}) = s(\tbf)^{-1}$ 
at least for all $\tbf \in Q$ such that $\tbf \neq \tbf^{-1}$.
For, we may consider, for any group $Y$, the polarization $Y = Y_+ \dot{\cup} Y_0 \dot{\cup} Y_-$ 
with respect to the operation $y \mapsto y^{-1}$.
Then, given an arbitrary unital section $s' : Q \to N$, we set
$s(\tbf) := s'(\tbf)$ for $\tbf \in Q_+ \cup Q_0$ and
$s(\tbf) := s'(\tbf^{-1})^{-1}$ for $\tbf \in Q_-$.
Note that if $\tbf \in Q_0$, \emph{i.e.} $\tbf = \tbf^{-1}$, 
then $s(\tbf)^{-1} = g_\tbf s(\tbf)$ for some $g_\tbf \in \ker \eta$.
\end{rem}

\begin{lem}
\label{lem.B1.1}
Let ${}^2G = (G \tto N)$ be a 2-group and $\Delta$ a poset. Then there is a diagram
\begin{equation}
\label{eq.B1.05}
\xymatrix{
   H^1(\Delta,N) 
    \ar[r]^-{d_*}
    \ar[dr]_-{ \mu_{1,*} }
&  \check{H}^2(\Delta,{}^2G)
    \ar[d]^-{ \mu_{2,*} }
\\ & H^1(\Delta,N/G)
}
\end{equation}
which commutes. If the map $i: G \to N$ is injective, then $\mu_{2,*}$ is bijective
and any $\nu \in \check{Z}^2(\Delta,{}^2G)$ is equivalent to a twisted connection.
\end{lem}

\begin{proof}
With the notation (\ref{eq.mod}), for any $(u,g) \in \check{Z}^2(\Delta,{}^2G)$, $v \in Z^1(\Delta,N)$ we define
\[
\{ \mu_1(v) \}_b \, := \, v_{b,G}
\ \ , \ \ 
\{ \mu_2(u,g) \}_b \, := \, u_{b,G} \, \in N/G
\ \ , \ \ 
b \in \Sigma_1(\Delta)
\ .
\]
We have $\mu_2 \circ d = \mu_1$ and $\mu_1 , \mu_2 , d$ preserve coboundaries,
thus the diagram (\ref{eq.B1.05}) is commutative as desired.
We now assume that $i$ is injective and for simplicity make the identification $G \equiv i(G)$.
To construct an inverse for $\mu_{2,*}$ we pick a unital section 
$s : N/G \to N$, $s(\tbf)_G = \tbf$.
%
%
Given $w \in Z^1(\Delta,N/G)$, we compute
\[
\{ s(w_\co) \, s(w_\cz) \, s(w_\cl)^{-1} \}_G  \ = \ 
w_\co \, w_\cz \, w_\cl^{-1} \ = \
\tbe
\ \ , \ \ 
\forall c \in \Sigma_2(\Delta) \ .
\]
Thus $\delta(w)_c :=$ $s(w_\co)  s(w_\cz)  s(w_\cl)^{-1} \in$ $G$, and we set
\[
\nu(w) := ( s(w) , \delta(w) ) \in \check{Z}^2(\Delta,{}^2G) \ \ \ , \ \ \ \forall w \in Z^1(\Delta,N/G) \ .
\]
By construction $\mu_2 \circ \nu$ is the identity of $Z^1(\Delta,N/G)$ so $\mu_{2,*}$ is surjective.

Now, let us consider $q=(u,g)$, $q' = (u',g') \in \check{Z}^2(\Delta,{}^2G)$.
If $\mu_{2,*}([q]) = \mu_{2,*}([q'])$ then there is a 0-cochain $\zeta = \{ \zeta_o \in N/G \}$ such that
$\zeta_\bo \, u'_{b,G} \, \zeta_\bl^{-1} = u_{b,G}$, $\forall b$.
Defining
$v_o := s(\zeta_o) \in N$, $\forall o \in \Delta$,
we note that
\[
h_b^{-1} \ := \ v_\bo \, u'_b \, v_\bl^{-1} \, u_b^{-1} \, \in G \ \ \ , \ \ \ \forall b \in \Sigma_1(\Delta) \ .
\]
At this point it is trivial to check that $(v,h) \in (q,q')$, and $\mu_{2,*}$ is injective.

We now prove that $\nu(w)$ is equivalent a twisted connection.
To this end we note that $w_{\ovl{b}} = w_b^{-1}$ for all $b$, and by Remark \ref{rem.section} we may assume
$s(w_{\ovl{b}}) =$ $s(w_b^{-1}) =$ $s(w_b)^{-1}$
for all $b$ such that $w_b \neq w_b^{-1}$. Instead, the case $w_b = w_b^{-1}$ may occur for two possible reasons:
(1) $b = \ovl{b}$: in this case we have necessarily $w_b = w_{\ovl{b}} = \tbe$, because $b$ can be regarded as a 
    loop from $\bl$ to $\bl = \bo$ homotopic to the constant path $\iota_\bl$ 
    (see \S \ref{sec.A1} and references cited therein);
    thus $s(w_b) = s(w_{\ovl{b}}) = 1$ and no problem arises; 
(2) $b \neq \ovl{b}$ and $w_b = w_b^{-1}$ ($= w_{\ovl{b}}$): in this case 
    $w_{\ovl{b}} w_b = \tbe$ and $s(w_{\ovl{b}}) s(w_b) \in G$.
These considerations imply that the only 1--simplices that are an obstacle to make $\nu(w)$ a twisted connection are 
those of the type (2).
Given a polarization 
$\Sigma_1(\Delta) =$ $\Sigma_1(\Delta)_+ \dot{\cup} \Sigma_1(\Delta)_0 \dot{\cup} \Sigma_1(\Delta)_-$
defined with respect to the operation $b \mapsto \ovl{b}$, we set
\[
h_{\ovl{b}} \, := \, 
\left\{
\begin{array}{l}
s(w_b)^{-1} s(w_{\ovl{b}})^{-1} \, \in G 
\ \ \ , \ \ \ 
\forall b \in \Sigma_1(\Delta)_+ \, : \, w_b = w_b^{-1} \, , \\
1 \ \ \ {\mathrm{otherwise}} \ .
\end{array}
\right.
\]
Note that if $b \in \Sigma_1(\Delta)_+$ (so that $b \neq \ovl{b}$) and $w_b = w_b^{-1}$, then $h_b = 1$.
In this way we get the cochain $h : \Sigma_1(\Delta) \to G$, and defining 
$u_b := h_b s(w_b)$ for all $b$
we get, in particular for $b \in \Sigma_1(\Delta)_+ $ with $w_b = w_b^{-1}$,
\[
u_{\ovl{b}} \ = \ h_{\ovl{b}} \, s(w_{\ovl{b}}) \ = \ s(w_b)^{-1} \ = \ u_b^{-1} \ .
\]
In the other cases $h_b = 1$ and $s(w_{\ovl{b}}) = s(w_b)^{-1}$, so $u_{\ovl{b}} = u_b^{-1}$.
In conclusion, $(1,h)$ makes $\nu(w)$ equivalent to the twisted connection 
$(u,g')$, $g'_c := \wa{u}'_\co(h_\cz) h_\co \delta(w)_c h_\cl^{-1}$.
\end{proof}

\paragraph{Abelian coefficients and Dixmier-Douady classes.}
With the above notation, let $G$ be abelian and $N$ act trivially on $G$. 
Then any non-abelian 2-cocycle $(u,g)$ defines the 2-cocycle $g \in Z^2(\Delta,G)$.
%
%
If $q=(u,g)$, $q'=(u',g') \in \check{Z}^2(\Delta,{}^2G)$ and $(v,h) \in (q,q')$,
then by the second of (\ref{eq.B1.01a}) we have $h \in (g,g')$ and we obtain the map
\begin{equation}
\label{eq.B1.03}
\delta : \check{H}^2(\Delta,{}^2G) \to H^2(\Delta,G) 
\ \ , \ \ 
\delta[u,g] := [g]
\ ,
\end{equation}
that we call the \emph{Dixmier-Douady class}. 
Note that $\delta \circ d_*[v] = [{\bf 1}]$ for all $[v] \in H^1(\Delta,N)$, 
thus $\delta$ measures the obstruction for $(u,g)$ being equivalent to a $N$-cocycle.

\begin{ex}[$C^*$-bundles with fibre the compact operators]
Let $H_\bullet$ be a Hilbert space and $K(H_\bullet)$ denote the $C^*$-algebra of compact operators.
It is well-known that $\Aut K(H_\bullet) \simeq PU(H_\bullet) := U(H_\bullet) / \bT$, where $\bT$ is identified with the centre of $U(H_\bullet)$.
By (\ref{eq.A6}) and the results of \S \ref{sec.A3}, the set of isomorphism classes of $C^*$-bundles over $\Delta$ with fibre $K(H_\bullet)$ is
\[
\bun(\Delta,K(H_\bullet)) \, \simeq \, 
H^1(\Delta,PU(H_\bullet)) \, \simeq \, 
H^1( \pi_1(\Delta) , PU(H_\bullet)) \ .
\]
We have the 2-group ${}^2\bT := ( \bT \tto U(H_\bullet) )$, with $i : \bT \to U(H_\bullet)$ injective and $U(H_\bullet)$ acting trivially on $\bT$ by adjoint action.
The maps of the previous Lemma take the form
\[
\xymatrix{
H^1(\Delta,U(H_\bullet)) 
 \ar[r]^-{d_*} 
 \ar[dr]_-{ \mu_{1,*} } &
\check{H}^2(\Delta,{}^2\bT)  
 \ar[r]^-{\delta} 
 \ar[d]^-{ \mu_{2,*} }  &
H^2(\Delta,\bT) \\
& H^1(\Delta,PU(H_\bullet))  & 
}
\]
with $\mu_{2,*}$ bijective and $\mu_{1,*}[v] := [\ad v]$, where $\ad$ is the adjoint action.
The argument of \cite[end of \S 4]{Vas12} shows that if the group cohomology $H^2( \pi_1(\Delta) , \bT )$ is not trivial,
then $\delta$ is not trivial (and $\mu_{1,*}$ is not surjective).
%
%
%
%
In terms of poset geometry we have:
given a $C^*$-bundle $\mA = (A,\jmath)_\Delta$ with fibre $K(H_\bullet)$ and associated class $[w] \in H^1(\Delta,PU(H_\bullet))$, 
if $\delta \circ \mu_{2,*}^{-1}[w] \neq [{\bf 1}]$ then $\mA$ cannot be of the form $\mK(\mH) = (K(H),\ad u)_\Delta$
for some Hilbert bundle $\mH=(H,u)_\Delta$.
Thus $\delta$ plays for $K(H_\bullet)$-bundles over $\Delta$ the same role 
the classical Dixmier-Douady invariant $\delta_{DD}$ plays for $K(H_\bullet)$-bundles over a manifold $M$ \cite[\S 10.7-8]{Dix}.
When $\Delta$ is a good base for $M$, $\delta$ can be interpreted as the obstrution for a \emph{flat} $K(H_\bullet)$-bundle 
$\sA \to M$ being of the type $\sK(\sH) \to M$ for a \emph{flat} Hilbert bundle $\sH \to M$, 
thus $\delta(\sA) = [{\bf 1}]$ implies $\delta_{DD}(\sA) = 0$. Note that $\delta_{DD}(\sA) \in H^3(M,\bZ) \simeq$ $H^2(M,\bT)$,
thus it should be of interest to write a map 
$H^2(\Delta,\bT) \to H^2(M,\bT)$
describing the relation between $\delta$ and $\delta_{DD}$.
\end{ex}

%
%
%
%

\paragraph{Holonomy of non-abelian cocycles.}
We now construct non-abelian cocycles over $\pi_1(\Delta)$ starting from non-abelian cocycles over $\Delta$.
To this end we give a notion of non-abelian cohomology of a group $\Pi$ with coefficients in a 2--group ${}^2G = (G \tto N)$. 
A reference for this subject is \cite{Bro}, nevertheless the definitions we give here are modeled on the theory of twisted $C^*$-dynamical systems \cite{ZM68,BS70,PR89}.
A \emph{non-abelian 2--cocycle} is a pair $( \nu , \gamma )$, where the maps
$\nu : \Pi \to N$, $\gamma :\Pi \times \Pi \to G$
fulfill, for all $l,m,n \in \Pi$,
\begin{equation}
\label{eq.B3.4}
\left\{
\begin{array}{ll}
\nu_l \, \nu_m \, = \, \unl{\gamma(l,m)} \, \nu_{lm} \ , \\ \\
\gamma(l,m) \, \gamma(lm,n) \, = \, \wa{\nu}_l(\gamma(m,n)) \, \gamma(l,mn) \ , \\ \\
\gamma(l,1) \, = \, \gamma(1,l) \, = \, 1 \ .
\end{array}
\right.
\end{equation}
When $i : G \to N$ is injective the second of (\ref{eq.B3.4}) follows from the first one.
We say that $(\nu',\gamma')$ is cohomologous to $(\nu,\gamma)$ whenever there is a pair
$(w,k)$, with $w \in N$ and $k : \Pi \to G$,
such that
\begin{equation}
\label{eq.B3.4a}
\nu'_l \ = \   \unl{k_l} \, w \, \nu_l \, w^{-1}
\ \ \ , \ \ \ 
\gamma'(l,m) \ = \ \wa{\nu}'_l(k_m) \, k_l \, \wa{w}(\gamma(l,m)) \, k_{lm}^{-1} \ ,
\end{equation}
for all $l,m \in \Pi$. As for cocycles over posets, it can be proved that (\ref{eq.B3.4a}) defines an equivalence relation.
The corresponding set of equivalence classes is denoted by $\check{H}^2(\Pi,{}^2G)$.

Let now $q = (u,g) \in \check{Z}^2_1(\Delta,{}^2G)$ and $a \in \Delta$. We set
\begin{equation}
\label{def.up}
u_p \, := \, u_{b_n} \circ \cdots \circ u_{b_1} \in N
\ \ , \ \ 
p : a \to a \, , \, p = b_n * \ldots * b_1 \ .
\end{equation}
Since $q$ is normalized, we have $u_{\iota_a} = 1$.
Now, $u_{b,G} \in N/G$, $b \in \Sigma_1(\Delta)$, defines a cocycle $u_G \in Z^1(\Delta,N/G)$ (Lemma \ref{lem.B1.1}),
thus $u_{p,G} = u_{p',G}$ for any loop $p'$ homotopic to $p$, and the so-obtained map 
$u_G : \pi_1(\Delta) \to N/G$
is multiplicative, \cite[Prop.3.8]{RR06}.

We now set $\Pi_1^a(\Delta) := \{ p : a \to a \}$. Then the homotopy relation yields a surjective map 
$\Pi_1^a(\Delta) \to \pi_1(\Delta)$, $p \mapsto [p]$,
and we pick a section
$s : \pi_1(\Delta) \to \Pi_1^a(\Delta)$, $[s(l)] = l$, $\forall l \in \pi_1(\Delta)$,
such that $s(1) = \iota_a$. Thus we define the map
$\bar{u} : \pi_1(\Delta) \to N$, $\bar{u}_l := u_{s(l)}$.
By the previous remarks, for all $l,m \in \pi_1(\Delta)$ we have
\begin{equation}
\label{eq.B3.5}
\bar{u}_1 \, = \, 1
\ \ \ , \ \ \
\bar{u}_l \, \bar{u}_m \, \bar{u}_{lm}^{-1} \, =: \, \gamma(l,m)  \, \in i(G)
\ .
\end{equation}
Note that $i(G)$ is normal in $N$, so the 2--group 
${}^2G_i = (i(G) \tto N)$
is defined. There is an obvious morphism ${}^2G \to {}^2G_i$ defined by $i$: it is an isomorphism when $i$ is injective.
\begin{prop}[The holonomy map]
\label{prop.B1.1}
The assignment $(u,g) \mapsto ( \bar{u} , \gamma )$ induces a map
\begin{equation}
\label{eq.B1.1}
\check{H}^2(\Delta,{}^2G) \to \check{H}^2(\pi_1(\Delta),{}^2G_i) \ ,
\end{equation}
which is independent of the choice of the section $s : \pi_1(\Delta) \to \Pi_1^a(\Delta)$, $s(1) = \iota_a$.
\end{prop}

\begin{proof}
If $q=(u,g)$ is normalized, then by (\ref{eq.B3.5}) we have that (\ref{eq.B3.4}) is fulfilled and $(\bar{u},\gamma) \in \check{Z}^2( \pi_1(\Delta) , {}^2G_i )$.
If $q$ is not normalized, then we may find a normalized cocycle $q_1 = ( u_1 , g_1 )$ cohomologous to $q$, and assign to $q$ the group cocycle 
$( \bar{u}_1 , \gamma_1 )$ defined as before.
To show that our definition is independent of the normalized cocycle
we pick a further $q'_1 = ( u'_1 , g'_1 ) \in$ $\check{Z}^2_1(\Delta,{}^2G)$ cohomologous to $q$,
and verify that $( \bar{u}'_1 , \gamma'_1 )$ is cohomologous to $( \bar{u}_1 , \gamma_1 )$.
To this end, we note that by hypothesis there is a pair 
$(v,h) \in ( q_1 , q'_1 )$.
Passing to maps modulo $G$ we have
$\bar{u}'_{1,b,G} =$ $v_{\bo,G} \, \, \bar{u}_{1,b,G} \, \, v_{\bl,G}^{-1}$,
thus
$\bar{u}'_{1,p,G} =$ $v_{a,G} \, \, \bar{u}_{1,p,G} \, \, v_{a,G}^{-1}$
for any $p : a \to a$. This implies 
\[
k_l \ := \ \bar{u}'_{1,l} \, v_a \, \bar{u}_{1,l}^{-1} \, v_a^{-1}  \, \in i(G) \ ,
\]
%
%
and $(v_a,k)$ makes $( \bar{u}_1 , \gamma_1 )$ and $( \bar{u}'_1 , \gamma'_1 )$ cohomologous, as desired.

We now check that (\ref{eq.B1.1}) is well-defined: if $q' \in \check{Z}^2(\Delta,{}^2G)$ is cohomologous to $q$,
then $q'$ is also cohomologous to the normalized cocycle $q_1$,
thus by the previous considerations the group cocycles assigned to $q$ and $q'$ are cohomologous.
We conclude that (\ref{eq.B1.1}) is well-defined.

As a final step we verify independence of $s$.
Let $s' : \pi_1(\Delta) \to \Pi_1^a(\Delta)$ be a section such that $s'(1) = \iota_a$. We set 
$\nu'_l := u_{s'(l)}$
and
$h_l := \nu'_l \bar{u}_l^{-1}$, $\forall l \in \pi_1(\Delta)$.
We note that, since by construction $s(l)$ is homotopic to $s'(l)$, we have $h_l \in i(G)$.
Then we compute
\[
\begin{array}{lcl}
\gamma'(l,m) & = & {\nu'}_l \, {\nu'}_m \, {\nu'}_{lm}^{-1}  \\ & = & 
{\nu'}_l \, h_m \bar{u}_m \, \bar{u}_{lm}^{-1}  h_{lm}^{-1}   \\ & = & 
\wa{\nu}'_l(h_m) \, {\nu'}_l \bar{u}_m \, \bar{u}_{lm}^{-1}  h_{lm}^{-1}  \\ & = &
\wa{\nu}'_l(h_m) \, h_l \bar{u}_l \bar{u}_m \, \bar{u}_{lm}^{-1}  h_{lm}^{-1}  \\ & = &
\wa{\nu}'_l(h_m) \, h_l \, \gamma(l,m) \, h_{lm}^{-1}  \ .
\end{array}
\]
Thus the pair $(1,h)$ makes $(\bar{u},\gamma)$ and $(\nu',\gamma')$ cohomologous, as desired.
\end{proof}



\subsection{Gerbes}
\label{sec.B2}

The basic idea of a gerbe over a poset is that we have a family $\{ B_o \}_{o \in \Delta}$ of $G$-spaces, with inclusion morphisms
%
%
that may fail to fulfill the precosheaf relations due to an obstruction described by a ${}^2G$-cocycle on $\Delta$.
A characteristic property of gerbes is that they define precosheaves when the family of fixed-point subspaces of the fibres is considered.

Let ${}^2G = (G\tto N)$ be a 2--group and $C$ a $G$-category.
Given a 2-cocycle $q = (u,g) \in \check{Z}^2(\Delta,{}^2G)$,
a \emph{$q$-gerbe} (or simply \emph{gerbe}) is given by a pair
$\check{\mB} = (B,\jmath)_\Delta$, 
where the families of $G$-spaces
$B = \{ B_a \in \obj C \}_{a \in \Delta}$
and of morphisms
$\jmath = \{ \jmath_b : B_\bl \to B_\bo \}_{b \in N_1(\Delta)}$
fulfil
\begin{equation}
\label{eq.B2.1}
\alpha_{\partial_{01}c}(g_c) \circ \jmath_\cl \ = \ \jmath_\co \circ \jmath_\cz
\ \ \ , \ \ \ 
\jmath_b \circ \alpha_\bl(g) \ = \ \alpha_\bo(\wa{u}_b(g)) \circ \jmath_b \ ,
\end{equation}
for all $c \in N_2(\Delta)$, $g \in G$, $b \in N_1(\Delta)$.
We call $B$ the \emph{set of fibres} and $\jmath$ the \emph{gerbe structure}.
Some remarks follow.
\begin{itemize}
\item For any $o \in \Delta$ we set $A_o := \{ v \in B_o : \{\alpha_o(g)\}(v) = v \, , \, \forall g \in G \}$. 
      The second of (\ref{eq.B2.1}) implies $\jmath_b(A_\bl) \subseteq A_\bo$, 
      and defining $\imath_b := \jmath_b | A_\bl$, $b \in N_1(\Delta)$, we find that $\imath$ fulfils the precosheaf relations.
      Thus we get the \emph{fixed-point precosheaf} $\mA := (A,\imath)_\Delta$.
\item When $\jmath_b$ is invertible for all $b \in \Sigma_1(\Delta)$ we say that $\check{\mB}$ is a \emph{bundle gerbe}. 
      In this case we can extend $\jmath$ from $N_1(\Delta)$ to $\Sigma_1(\Delta)$, by defining
      $\jmath_b := \jmath_{|b| \bo}^{-1} \circ \jmath_{|b| \bl}$, $\forall b \in \Sigma_1(\Delta)$.
      In this way, (\ref{eq.B2.1}) extends to arbitrary simplices 
      $c \in \Sigma_2(\Delta)$, $b \in \Sigma_1(\Delta)$.
      %
      %
      Note that by definition 
      \begin{equation}
      \label{eq.tc1}
      \jmath_{\ovl{b}} \, = \, \jmath_b^{-1} \ \ \ , \ \ \ \forall b \in \Sigma_1(\Delta) \ .
      \end{equation}
\item If $g_c \equiv 1$ then $u \in Z^1(\Delta,N)$ and $\check{\mB}$ is a $u$-precosheaf.
      In this case we write $\mB \equiv \check{\mB}$.
      %
      %
      %
\end{itemize}

\begin{defn}
\label{def.gmor}
Let ${}^2G' = (G \tto N')$ be a 2--group
{\footnote{Note that $G$ is left unchanged with respect to the definition of ${}^2G$ used in the previous lines.}},
$q'=(u',g') \in \check{Z}^2(\Delta,{}^2G')$ and
$\check{\mB}' = (B',{\jmath \, }')_\Delta$
a $q'$-gerbe. A \textbf{gerbe morphism} $\eta_\beta : \check{\mB} \to \check{\mB}'$ is given by families
$\eta_o : B_o \to B'_o$,
$\beta_b \in \alpha'_\bo(G)$,
$o \in \Delta$, $b \in N_1(\Delta)$,
such that
\begin{equation}
\label{eq.B2.1a}
\beta_b \circ \eta_\bo \circ \jmath_b \, = \, {\jmath \, }'_b \circ \eta_\bl
\ \ \ \ , \ \ \ \
\alpha'_o(g) \circ \eta_o \, = \, \eta_o \circ \alpha_o(g) \ ,
\end{equation}
for all $g \in G$, where $\alpha'_o := \alpha_{B'_o}$.
When $\beta_b$ is the identity for all $b$ we write $\eta \equiv \eta_\beta$.
We say that $\eta_\beta$ is an \textbf{isomorphism} whenever any $\eta_o$ is an isomorphism.
\end{defn}

\begin{rem}
The factors $\beta_b$ in the previous definition are motivated by the wish of making the notion of gerbe isomorphism
coherent with the one of cocycle equivalence (\ref{eq.B1.01a});
their presence implies that non-isomorphic precosheaves $\mB$, $\mB'$
may become isomorphic when regarded as gerbes, compare (\ref{eq.B2.1a}) with (\ref{eq.A7a}).
\end{rem}

\begin{ex}[2--group gerbes]
\label{ex.Gadj}
Let ${}^2G = \{ {}^2G_o := (G_o \tto N_o) \}_{o \in \Delta}$ be a collection of 2--groups.
We assume that there is a fixed 2--group
${}^2G_\bullet = (G_\bullet \tto N_\bullet)$
with isomorphisms $\nu_o : {}^2G_\bullet \to {}^2G_o$, $o \in \Delta$.
By (\ref{eq.2G.1}) any $\nu_o$ defines the adjoint action
\[
\alpha_o : G_\bullet \to \Aut {}^2G_o
\ \ , \ \ 
\{ \alpha_o(g) \} v := \{ \ad \nu_o(g) \} (v)  
\ , \
g \in G_\bullet , v \in N_o
\ .
\]
Let us now consider  $q = (u,g) \in \check{Z}^2( \Delta , {}^2G_\bullet )$. A \emph{2--group $q$-gerbe} is a pair
${}^2\check{\mG}  =$ $( {}^2G , \jmath )_\Delta$,
where $\jmath_b : {}^2G_\bl \to {}^2G_\bo$, $b \in \Sigma_1(\Delta)$, are isomorphisms fulfilling (\ref{eq.B2.1}).
${}^2\check{\mG}$ defines:
\begin{enumerate}
\item the group gerbe $\check{\mG} = (G,\jmath)_\Delta$, where $\jmath_b : G_\bl \to G_\bo$, $b \in \Sigma_1(\Delta)$, 
      are (with an abuse of notation) the group isomorphisms induced by the initial 2--group isomorphisms;
\item the group \emph{bundle} $Q\mG = (QG,\bar{\jmath})_\Delta$, having defined $QG_o := N_o/G_o$ for all $o \in \Delta$, and 
      $\bar{\jmath}_b : QG_\bl \to QG_\bo$, $b \in \Sigma_1(\Delta)$, as the group isomorphisms induced by $\jmath$.
      %
      %
\end{enumerate}
Let now, in particular, $q$ be a twisted connection.
We set ${}^2G_o := {}^2G_\bullet$ for all $o \in \Delta$ and $\jmath_b := \ad u_b$ for all $b \in \Sigma_1(\Delta)$
(note that, by (\ref{eq.2G.1}), $\jmath_b \in \Aut {}^2G_\bullet$).
Then the pair
\[
{}^2\check{\mG}_q \, = \, ( {}^2G_\bullet , \ad u )_\Delta
\]
is a 2--group $q$-gerbe, that we call the \emph{standard $q$-gerbe}.
More in general, we say that ${}^2\check{\mG} = ( {}^2G , \jmath )_\Delta$ is a standard $q$-gerbe whenever 
$\jmath_b = \nu_\bo \circ \ad u_b \circ \nu_\bl^{-1}$, $b \in \Sigma_1(\Delta)$;
these gerbes are all isomorphic to the standard $q$-gerbe, in fact the family $\nu = \{ \nu_o \}$ realizes 
the isomorphism $\nu : {}^2\check{\mG}_q \to {}^2\check{\mG}$.
\end{ex}

\subsection{Transition maps and twisted holonomy}
\label{sec.B3}

As we mentioned in \S \ref{sec.A} holonomy plays a fundamental role in the theory of bundles over posets.
Thus we discuss the analogue of holonomy for gerbes, recognizing that it takes the form of a
non-abelian cocycle in the cohomology of $\pi_1(\Delta)$.

Let ${}^2G = (G \tto N)$ be a 2--group, $q = (u,g) \in \check{Z}^2(\Delta,{}^2G)$,
$C$ a $G$-category and $\check{\mB} = (B,\jmath)_\Delta$ a bundle $q$-gerbe as in (\ref{eq.B2.1}).
We pick $a,o \in \Delta$ and define
\begin{equation}
\label{eq.B3.1}
\jmath_p \, := \, \jmath_{b_n} \circ \ldots \circ \jmath_{b_1} \, : B_a \to B_o
\ \ , \ \ 
\forall p = b_n * \ldots * b_1 \, , \, p : a \to o \ .
\end{equation}
Iterating (\ref{eq.tc1}) on $b_1 , \ldots , b_n$ we get
\begin{equation}
\label{eq.tc2}
\jmath_{\ovl{p}} \, = \, \jmath_p^{-1} \ \ \ , \ \ \ \forall p : a \to o \ .
\end{equation}
Let $p' : a \to o$ be homotopic to $p$.
Then the first of (\ref{eq.B2.1}) implies that in general $\jmath_p \neq \jmath_{p'}$
unless $g_c \equiv 1$, that is, when $\check{\mB}$ is a bundle, see \S \ref{sec.A3};
anyway, since $\jmath$ fulfils (\ref{eq.A7}) up to elements of $\alpha_o(G)$, there is $g(p,p') \in G$ such that
$\jmath_p  =  \alpha_o(g(p,p')) \circ \jmath_{p'}$.

We define the \emph{holonomy group} 
\[
\Hol_a\check{\mB} \, := \, \alpha_a(G) \vee \{ \jmath_p \in \Aut B_a \, : \, p : a \to a \} \, \subseteq \Aut B_a \ ,
\]
where $\vee$ stands for the group generated by the two factors.
The following facts hold:
\begin{itemize}
\item Let $b \in \Sigma_1(\Delta)$. Then for any $g \in G$ and $p : \bl \to \bl$ we have
      $\jmath_b \circ \alpha_\bl(g) \circ \jmath_p \circ \jmath_b^{-1} = 
      \alpha_\bo(\wa{u}_b(g)) \circ \jmath_{b*p*\ovl{b}}$.
      Thus there are isomorphisms
      \begin{equation}
      \label{eq.B3.1c}
      \ad \jmath_b : \Hol_\bl\check{\mB} \to \Hol_\bo\check{\mB} 
      \ \ \ , \ \ \ 
      \ad \jmath_b(z) := \jmath_b \circ z \circ \jmath_b^{-1}
      \ .
      \end{equation}
\item Let $g \in G$ and $p : a \to a$, $p = b_n * \ldots * b_1$. Then applying the second of (\ref{eq.B2.1}) we have
      \begin{equation}
      \label{eq.B3.1e}
      \jmath_p \circ \alpha_a(g) \circ \jmath_p^{-1} \, = \, \alpha_a(\wa{u}_p(g)) \, \in \alpha_a(G) \ ,
      \end{equation}
      where 
      $\wa{u}_p := \wa{u}_{b_n} \circ \ldots \circ \wa{u}_{b_1} \in \Aut G$.
      Thus $\alpha_a(G)$ is normal in $\Hol_a\check{\mB}$ and we have the \emph{holonomy 2--group at $a \in \Delta$}
      \[
      {}^2G_{a,\hhol} \, := \, ( \alpha_a(G) \tto \Hol_a\check{\mB} ) 
      \ .
      \]
      %
      %
      Applying  Lemma \ref{lem.B1.1} and (\ref{eq.A6}) we obtain isomorphisms
      \begin{equation}
      \label{eq.B3.2}
      \check{H}^2(\Delta,{}^2G_{a,\hhol} ) \, \stackrel{\simeq}{\rightarrow} \,
      H^1(\Delta,QG_{a,\hhol}) \, \stackrel{\simeq}{\rightarrow} \,
      H^1( \pi_1(\Delta) , QG_{a,\hhol} ) \ ,
      \end{equation}
      where $QG_{a,\hhol} := \Hol_a\check{\mB} / \alpha_a(G)$.
      %
      %
\item Defining $NB_a$ as the normalizer of $\alpha_a(G)$ in $\Aut B_a$ we get the \emph{normalizer 2--group}
      ${}^2G_{a,\alpha} =$ $( \alpha_a(G) \tto NB_a )$.
      Since $\Hol_a\check{\mB} \subseteq NB_a$, there is an inclusion morphism
      ${}^2G_{a,\hhol} \to {}^2G_{a,\alpha}$
      inducing a canonical map at the level of cohomology sets
      \begin{equation}
      \label{eq.B3.1bb}
      \check{H}^2( \Delta , {}^2G_{a,\hhol} ) \, \to \, \check{H}^2( \Delta , {}^2G_{a,\alpha} ) 
      \ \ , \ \ 
      [\tau] \mapsto [\tau]_\alpha
      \ .
      \end{equation}
\end{itemize}
%
%
We now define a non-abelian cocycle with coefficients in ${}^2G_{a,\hhol}$ describing the "transition maps" of $\check{\mB}$.
To this end we consider a path frame $p_a = \{ p_{ao} : o \to a \}$ (see (\ref{eq.pf})) and,
using (\ref{eq.loops}), we define the loops
\begin{equation}
\label{eq.ellb}
\ell_b : a \to a
\ \ \ , \ \ \ 
\ell_b := p_{a\bo} * b * p_{\bl a}
\ \ \ , \ \ \
b \in \Sigma_1(\Delta)
\ .
\end{equation}
Then we define the holonomies
$\jmath^{\, a}_b  :=  \jmath_{\ell_b} \in$ $\Hol_a\check{\mB}$,
$b \in \Sigma_1(\Delta)$.
To study the behaviour of $\jmath^{\, a}$ over a 2-simplex $c \in \Sigma_2(\Delta)$, we consider the path
$p_{a\partial_{01}c} = b_n * \ldots * b_1$,
$p_{a\partial_{01}c} : a \to \partial_{01}c$,
and note that since $\partial_{00}c = \partial_{01}c$ (see \S \ref{sec.A1}), we have 
$p_{a\partial_{01}c} = p_{a\partial_{00}c}$.
Then we compute, using repeatedly (\ref{eq.B2.1}),
\[
\begin{array}{lcl}
\jmath^{\, a}_\co \circ \jmath^{\, a}_\cz \circ (\jmath^{\, a}_\cl)^{-1} & = &
\jmath_{p_{a\partial_{00}c}} \circ \jmath_\co \circ \jmath_\cz \circ \jmath_{\cl}^{-1} \circ \jmath_{p_{\partial_{01}ca}} \\ & = &
\jmath_{p_{a\partial_{00}c}} \circ \alpha_{\partial_{01}c}(g_c) \circ \jmath_{p_{\partial_{01}ca}} \\ & = &
\jmath_{p_{a\partial_{01}c}} \circ \alpha_{\partial_{01}c}(g_c) \circ \jmath_{p_{\partial_{01}ca}} \\ & = &
\alpha_a( \wa{u}_{b_n} \circ \ldots \circ \wa{u}_{b_1} (g_c) ) \ .
\end{array}
\]
Thus 
$\delta^a_c := \jmath^{\, a}_\co \circ \jmath^{\, a}_\cz \circ (\jmath^{\, a}_\cl)^{-1} \in \alpha_a(G)$, 
for all $c \in \Sigma_2(\Delta)$.
\begin{prop}[Transition maps]
\label{prop.B3.1}
Let $\check{\mB} = (B,\jmath)_\Delta$ be a bundle $q$-gerbe and $a \in \Delta$. Then
$\tau^a = ( \jmath^{\, a} , \delta^a )$
is a twisted connection in $\check{Z}^2(\Delta,{}^2G_{a,\hhol} )$ whose cohomology class is independent of the path frame.
A bundle gerbe $\check{\mB}'$ is isomorphic to $\check{\mB}$ if, and only if,
$\check{\mB}$ and $\check{\mB}'$ have isomorphic standard fibres and
$\check{\mB}'$ defines a cocycle $\tau'^a$ such that $[\tau'^a]_\alpha =$ $[\tau^a]_\alpha \in \check{H}^2( \Delta , {}^2G_{a,\alpha} )$.
\end{prop}

\begin{proof}
By construction $\tau^a$ is a ${}^2G_{a,\hhol}$-cocycle.
%
%
Moreover (\ref{eq.tc2}) implies $\jmath^{\, a}_{\ovl{b}} = \jmath^{\, a , -1}_b$, $b \in \Sigma_1(\Delta)$, 
thus $\tau^a$ is a twisted connection.
Let now $p'_a$ denote a path frame. Then we have a new cocycle 
$\tau'^a = ( {\jmath \, }'^{\, a} , \delta'^a ) \in \check{Z}^2(\Delta,{}^2G_{a,\hhol} )$,
where ${\jmath \, }'^a_b := \jmath_{p'_{a\bo}} \circ \jmath_b \circ \jmath_{p'_{\bl a}}$.
We define the holonomies
$v_o := \jmath_{p'_{ao}} \circ \jmath_{oa} \in \Hol_a\check{\mB}$, $o \in \Delta$;
a simple check shows that
${\jmath \, }'^a_b  =  v_\bo \circ \jmath^{\, a}_b \circ v_\bl^{-1}$ 
and
$\delta'^a_c  =  \wa{v}_{\partial_{01}c}(\delta^a_c)$,
for all $b \in \Sigma_1(\Delta)$ and $c \in \Sigma_1(\Delta)$.
This implies $(v,{\bf 1}) \in ( \tau^a , \tau'^a )$ and $[\tau^a]$ is independent of the path frame.

Let now ${}^2G' = (G \tto N')$ be a 2--group, $q' = (u',g') \in \check{Z}^2(\Delta,{}^2G')$ and 
$\check{\mB}' = (\mB',{\jmath \, }')_\Delta$ a $q'$-bundle gerbe with fibres isomorphic to those of $\check{\mB}$.
Without loss of generality we assume that $B'_o = B_o$ for all $o \in \Delta$, so the normalizer 2--groups are the same,
${}^2G_{a,\alpha} = {}^2G_{a,\alpha' \, }$.
Let $\eta_\beta : \check{\mB} \to \check{\mB}'$ be an isomorphism. Then
${\jmath \, }'_b  =  \beta_b \circ \eta_\bo \circ \jmath_b \circ \eta_\bl^{-1}$ for all $b \in \Sigma_1(\Delta)$,
and the map
\[
\lambda_b(\beta) \ := \ 
\eta_\bo \circ \jmath_b \circ \eta_\bl \circ \beta \circ \eta_\bl^{-1} \circ \jmath_b^{-1} \circ \eta_\bo^{-1}
\ \ , \ \ 
\beta \in \alpha_\bl(G) \ ,
\]
defines an isomorphism 
$\lambda_b : \alpha_\bl(G) \to \alpha_\bo(G)$
because of (\ref{eq.B2.1}) and (\ref{eq.B2.1a}).
Iterating over the loop (\ref{eq.ellb}) $\ell_b = b_n * \ldots * b_1$ we obtain
\[
{\jmath \, }'^a_b 
\ = \ 
h_b \circ \eta_a \circ \jmath^{\, a}_b \circ \eta_a^{-1}
\ \ \ , \ \ \
h_b := \beta_{b_n} \lambda_{b_{n-1}}( \beta_{b_{n-1}} \lambda_{b_{n-2}}(\beta_{b_{n-2}}  \ldots  \beta_{b_1})) \in \alpha_a(G)
\, . 
\]
By the identification $B_a = B'_a$ and (\ref{eq.B2.1a}) we have $\eta_a \in NB_a$,
thus $(\eta_a,h) \in ( \tau^a , \tau'^a )$ as desired.
On the converse, let the twisted connection $\tau'^a = ({\jmath \, }'^a,\delta'^a)$ defined by $\check{\mB}'$ be equivalent to $\tau^a$.
Then there is $(v,h) \in (\tau^a,\tau'^a)$ such that, in particular, 
${\jmath \, }'^a_b \, = \, h_b \circ v_\bo \circ \jmath^{\, a}_b \circ v_\bl^{-1}$ for all $b \in \Sigma_1(\Delta)$.
For any $o \in \Delta$ and $b \in \Sigma_1(\Delta)$ we set
\[
\eta_o \, := \, \jmath_{p_{oa}} \circ v_a \circ \jmath_{p_{ao}}
\ \ \ , \ \ \ 
\beta_b \, := \, \jmath_{p_{\bo a}} \circ h_b \circ \jmath_{p_{a\bl}} \ ,
\]
and a trivial computation shows that $\eta_\beta : \check{\mB} \to \check{\mB}'$ is an isomorphism. 
\end{proof}
%

\begin{cor}
\label{cor.B3.1}
Let $\mA = (A,\imath)_\Delta$ be the fixed-point bundle of the $q$-gerbe $\check{\mB} = (B,\jmath)_\Delta$.
Then $\mA$ defines the cocycle $\imath^a \in Z^1(\Delta,QG_{a,\hhol})$, where
$\imath^a_b := \jmath^{\, a}_b {\mathrm{mod}}\alpha_a(G)$, $b \in \Sigma_1(\Delta)$.
\end{cor}

\begin{proof}
It suffices to note that $\jmath^{\, a}_b | A_a$ equals $\imath^a_b$, whilst $\delta^a_c$ is the identity of $A_a$.
\end{proof}

As an application of the previous results we assign a \emph{twisted holonomy} to the bundle gerbe 
$\check{\mB} = (B,\jmath)_\Delta$
by defining the cohomology class
\begin{equation}
\label{eq.B3.3}
[\bar{\tau}^a] \, = \, [\bar{\jmath}^{\, a},\gamma^a] \, \in \check{H}^2( \pi_1(\Delta) , {}^2G_{a,\hhol} )
\end{equation}
as the one associated to $[\tau^a] \in \check{H}^2(\Delta,{}^2G_{a,\hhol} )$
in the sense of Proposition \ref{prop.B1.1} (note that $\alpha_a(G) \subseteq \Hol_a\check{\mB}$,
thus in the notation of Proposition \ref{prop.B1.1} we have ${}^2G_{a,\hhol} = {}^2G_{a,\hhol,i}$).
%

\paragraph{Associated gerbes.}
Let $q = (u,g) \in$ $\check{U}^2(\Delta,{}^2G)$ be a twisted connection and $\check{\mB} = (B,\jmath)_\Delta$ a bundle $q$-gerbe.
We say that $\check{\mB}$ \emph{is associated to} $q$ whenever:
\begin{enumerate}
\item the action $\alpha_a : G \to \Aut B_a$ extends to a 2--group morphism 
      $\alpha_a : {}^2G \to {}^2G_{a,\alpha}$;
\item the twisted connection $\tau^a$ defined by $\check{\mB}$ is equivalent to $\alpha_{a,*}q$ in $\check{Z}^2(\Delta,{}^2G_{a,\alpha} )$,
      where $\alpha_{a,*}q := ( \alpha_a \circ u , \alpha_a \circ g )$ as in \S \ref{sec.B1}.
\end{enumerate}
In essence the previous definition says that $\check{\mB}$ is determined by $q$ and the 2--action $\alpha_a$.
For, let us define
$\check{\mF} = (F,\imath)_\Delta$, 
$F_o \equiv B_a$,
$\imath_b := \alpha_a(u_b)$,
for all
$o \in \Delta$, $b \in \Sigma_1(\Delta)$.
It is easily seen that $\check{\mF}$ is a $q$-gerbe, and it is constructed only using $q$ and $\alpha_a$.
Its transition maps can be easily computed,
\[
\imath^a_b \, = \, v_\bo \circ \alpha_a(u_b) \circ v_\bl^{-1}
\ \ \ \ , \ \ \ \
\xi^a_c \, := \, \imath^a_\co \circ \imath^a_\cz \circ (\imath^a_\cl)^{-1} \, = \, 
v_{ \partial_{01}c } \circ \alpha_a(\unl{g_c}) \circ v_{ \partial_{01}c }^{-1} \ ,
\]
where $v_o := \alpha_a(u_{p_{ao}})$, see (\ref{eq.pf}) and (\ref{def.up}).
The above relations say that the pair $(v,1)$ makes $(\imath^a,\xi^a)$ cohomologous to $\alpha_{a,*}q$,
thus $\check{\mF}$ is isomorphic to $\check{\mB}$ by Prop.\ref{prop.B3.1}.

\begin{ex}[Hilbert gerbes]
\label{ex.Hilb}
Let ${}^2G_\bullet = (G_\bullet \tto N_\bullet)$ be a 2--group and $q = (u,g) \in \check{Z}^2(\Delta,{}^2G_\bullet)$.
By definition, $q$-Hilbert bundle gerbes are given by pairs
$\check{\mH} = (H,U)_\Delta$,
where $U_b : H_\bl \to H_\bo$, $b \in \Sigma_1(\Delta)$, are unitary operators such that
\begin{equation}
\label{eq.B2.2}
\left\{
\begin{array}{ll}
\pi_{\partial_{01}c}(g_c) \circ U_\cl \ = \ U_\co \circ U_\cz \ \ , \ \ c \in \Sigma_2(\Delta) \ ,
\\ \\
U_b \circ \pi_\bl(g) \ = \ \pi_\bo(\wa{u}_b(g)) \circ U_b \ \ , \ \ g \in G_\bullet \, , \, b \in \Sigma_1(\Delta) \ ,
\end{array}
\right.
\end{equation}
and $\pi_o : G_\bullet \to U(H_o)$, $o \in \Delta$, are unitary representations.
%
%
Any $q$-Hilbert bundle gerbe $\check{\mH}$ defines a unitary $q$-group gerbe, in the following way.
We set $U(H) := \{ U(H_o) \}_{o \in \Delta}$ and consider the $G_\bullet$-actions
$\{ \ad \pi_o(g) \}(y) :=$ $\pi_o(g) y \pi_o(g)^*$, $g \in G_\bullet$, $y \in U(H_o)$;
then we consider the isomorphisms
$\ad U_b(y) := U_b y U_b^*$, $y \in U(H_\bl)$, $b \in \Sigma_1(\Delta)$.
Using (\ref{eq.B2.2}) we obtain that $\mU(\check{\mH}) := ( U(H) , \ad U )_\Delta$ is a $q$-group gerbe,
that we call the \emph{unitary gerbe of $\check{\mH}$}.
Let now $\check{\mG} = (G,\jmath)_\Delta$ denote the $q$-group gerbe defined by a standard $q$-gerbe Ex.\ref{ex.Gadj}.
Then the second of (\ref{eq.B2.2}) says that we have the gerbe morphism 
\[
\ad \pi : \check{\mG} \to \mU(\check{\mH}) \ .
\]
The category of $q$-Hilbert gerbes has naturally defined tensor product and conjugates, induced by those of $\wa{G}_\bullet$.
The identity of the tensor product is the trivial gerbe $\mC = ( C , \imath )_\Delta$
where $C_o \equiv \bC$ carry the trivial $G_\bullet$-action and $\imath_b \equiv 1$.
We have $(\mC,\mC) \simeq \bC$ when $\Delta$ is connected \cite[Prop.5.5]{RRV08}, a property that we assume in the present paper.
Thus in this case the category of $q$-Hilbert gerbes a DR-category.

\noindent Let now $q$ be a twisted connection , $a \in \Delta$ and $\check{\mH} = (H,U)_\Delta$ a Hilbert gerbe associated to $q$.
Then, given the normalizer $NH^\pi_a$ of $\pi_a(G_\bullet)$ in $U(H_a)$,
%
%
by definition the representation $\pi_a$ extends to a morphism $\pi_a : {}^2G_\bullet \to {}^2G_{a,\pi}$, where
${}^2G_{a,\pi} = ( \pi_a(G_\bullet) \tto NH^\pi_a )$.
Up to equivalence, $\check{\mH}$ has transition maps $\pi_{a,*}q = ( \pi_a \circ u , \pi_a \circ g )$.
We denote the category of Hibert gerbes associated to $q$ by 
\[
\Hilb_q(\Delta) \ \ \ , \ \ \ q \in \check{U}^2(\Delta,{}^2G_\bullet) \ .
\]
The above remarks imply that $\Hilb_q(\Delta)$ is a DR-category (having assumed $\Delta$ connected). 
\end{ex}

\paragraph{Gerbes over posets and gerbes over spaces.}
Let $\Delta$ be a good base for a Hausdorff paracompact space $X$ and $Q$ a topological group.
In \cite[\S 8]{RRV08} it is explained how to get $Q$-cocycles in the ordinary geometric sense,
that form a set $Z^1(X,Q)$,
%
%
from cocycles in $Z^1(\Delta,Q)$. The so-obtained cocycles are locally constant 
(that is, they are given by maps that are constant on the connected components of their domains)
and yield transition maps for eventually associated flat bundles over $X$.
This defines a map 
\begin{equation}
\label{eq.comp.01}
{\mathrm{T}}^1 : Z^1(\Delta,Q) \, \to \, Z^1(X,Q) \ .
\end{equation}
Applying the cohomology relation we get a map
${\mathrm{T}}^1_* : H^1(\Delta,Q) \, \to \, H^1(X,Q)$
that in general is neither injective nor surjective:
this is the cohomological counterpart of  well-known properties of flat bundles over $X$
\cite[\S I.2]{Kob}, \cite[\S 2.2 and \S 4.3]{RV14a}, \cite[\S A]{RV14b}.

Let now ${}^2G = (G \tto N)$ be a 2--group where $G$, $N$ and $N/G$ are topological groups.
We make the simplifying hypothesis that the map $i : G \to N$ is injective, and identify $G$ with $i(G) \subseteq N$.
Then there is a map
\begin{equation}
\label{eq.comp.02}
S : Z^1(X,N/G) \ \to \ \check{Z}^2(X,{}^2G) \ ,
\end{equation}
where $\check{Z}^2(X,{}^2G)$ is the set of non-abelian cocycles over $X$ with coefficients in ${}^2G$.
The map (\ref{eq.comp.02}) induces an isomorphism at the level of the corresponding cohomologies,
see \cite[Lemma 14]{Vas09} or \cite[Lemma 2]{BS} (warning: in \cite{BS} $\check{H}^2(X,{}^2G)$ is written $\check{H}^1(X,{}^2G)$).
Combining (\ref{eq.comp.02}) with the map $\mu_2$ of Lemma \ref{lem.B1.1} we get the map
\begin{equation}
\label{eq.comp.03}
\check{\mathrm{T}}^2 : \check{Z}^2(\Delta,{}^2G) \to \check{Z}^2(X,{}^2G)
\ \ , \ \ 
\check{\mathrm{T}}^2 := S \circ {\mathrm{T}}^1 \circ \mu_2
\ .
\end{equation}
This yields a bridge from gerbes over posets to gerbes in the ordinary sense (\emph{principal ${}^2G$-2-bundles}, \cite[\S 2]{BS}).
It is not difficult to verify that $\check{\mathrm{T}}^2(q)$ is locally constant for all $q \in \check{Z}^2(\Delta,{}^2G)$.
When $X$ is a manifold and $G \simeq \bU(1)$, gerbes over $X$ obtained by the map $\check{\mathrm{T}}^2$ are expected to be flat in the sense of \cite{BTW04}.


\begin{ex}[Morita bundle gerbes]
\label{ex.pennig}
With the above notation, let $\alpha : {}^2G \to {}^2\Aut A$ be a Green $C^*$-dynamical system.
Then any $\zeta \in \check{Z}^2(X,{}^2G)$ defines by composition a cocycle $\alpha_*\zeta \in \check{Z}^2(X,{}^2\Aut A)$.
Given $q \in \check{Z}^2(\Delta,{}^2G)$, we get a cocycle
$\check{\mathrm{T}}^2(q) \in \check{Z}^2(X,{}^2G)$,
and, by composition, a cocycle
$\alpha_* \check{\mathrm{T}}^2(q) \in \check{Z}^2(X,{}^2\Aut A)$.
If $A$ is the stabilization of a unital $C^*$-algebra, then $\alpha_* \check{\mathrm{T}}^2(q)$ defines a Morita bundle gerbe $\mA_{\alpha,q}$, \cite[Cor.3.42]{Pen11},
that is expected to be the gerbe analogue of a locally constant (also called \emph{flat}) $C^*$-bundle over $X$ \cite[\S 2.2]{RV14a}, with structure group $G$.
\end{ex}


\section{Twisted $C^*$-dynamical systems and $C^*$-gerbes}
\label{sec.C}

In the present section we focus on $C^*$-algebras and introduce a notion of twisted $C^*$-dynamical system, generalizing the classical one.
We discuss the representation theory of our twisted $C^*$-dynamical systems and show that they can arise as twisted holonomies of $C^*$-gerbes.
Some of the following results may be proven for generic gerbes, but we prefer to discuss them in the case of $C^*$-algebras,
that is the one in which we are mostly interested.

\subsection{How $C^*$-gerbes arise}
\label{sec.C1}

To better illustrate our motivations we present some situations where $C^*$-gerbes arise.

A first scenario is given by the "reconstruction problem", that may be summarized by the question 
whether given a $C^*$-bundle of fixed point algebras, there exists a $C^*$-bundle having it as a fixed-point bundle.
In particular, the case of bundles of fixed-point algebras of the Cuntz algebra nicely illustrate the general results.

The second scenario is given by quantum Dirac fields interacting with classical electromagnetic potentials
defining vanishing electromagnetic fields (closed de Rham 1--forms).

\paragraph{Fixed-point $C^*$-bundles.}
As we already seen, bundle $C^*$-gerbes $\check{\mF}$ define fixed-point bundles $\mA$.
We now make the reverse path, starting from $\mA$ and reconstructing $\check{\mF}$.

Let $\mA = (A,\jmath)_\Delta$ be a $C^*$-bundle. Without loss of generality we assume that
there is a fixed standard fibre $A_\bullet \equiv A_o$ for all $o \in \Delta$, Remark \ref{rem.standardfibre}.
Now, let us suppose that there is a $C^*$-algebra $F_\bullet$ with an automorphism group $G \subseteq \Aut F_\bullet$
such that
{\footnote{
The idea is that $G = \alpha(G')$ for some action $\alpha : G' \to \Aut F_\bullet$, 
anyway in the context of the present paragraph $G$ is the group of interest.
}}
$A_\bullet = F_\bullet^G$; in this case we say that $\mA$ is a \emph{bundle of fixed-point algebras}.
Note that we have the normalizer 2--group
${}^2G = ( G \tto NF_\bullet )$,
and $QG := NF_\bullet/G$ acts on $A_\bullet$ by *-automorphisms.

Now, we ask whether there is a bundle structure
$\imath_b : F_\bullet \to F_\bullet$, $b \in \Sigma_1(\Delta)$,
making $\mF = (F,\imath)_\Delta$ a $C^*$-bundle (with $F_o \equiv F_\bullet$ for all $o$), such that 
\begin{equation}
\label{eq.C1.1}
\imath_b \circ g \, = \, \nu_b(g) \circ \imath_b
\ \ \ , \ \ \
\imath_b | A_\bullet \, = \, \jmath_b 
\ \ \ , \ \ \
\forall b \in \Sigma_1(\Delta) \, , \, g \in G \, ,
\end{equation}
for some $\nu \in Z^1(\Delta,\Aut G)$.
The second of (\ref{eq.C1.1}) says that there is an inclusion morphism $\mA \to \mF$.
When both the conditions in (\ref{eq.C1.1}) hold, we say that $\mA$ is the \emph{fixed-point bundle of} $\mF$ and we write $\mA = \mF^\alpha$.

Let us now fix $a \in \Delta$.
Then Cor.\ref{cor.B3.1} gives the condition for $\mA$ being the fixed-point bundle of a bundle gerbe:
in particular, since $QG_{a,\hhol} \subseteq QG$, a necessary condition is
\begin{equation}
\label{eq.C1.1a}
\jmath_p \in QG \ \ \ , \ \ \ \forall p : a \to a \ .
\end{equation}
\begin{prop}
\label{prop.C1.1}
Let $\mA$ be a bundle of fixed-point $C^*$-algebras fulfilling (\ref{eq.C1.1a}).
Then there are a twisted connection $\tau = (\upsilon,\delta) \in \check{U}^2(\Delta,{}^2G)$ (unique up to equivalence)
with associated $C^*$-gerbe $\check{\mF}_\mA$ (unique up to isomorphism)
such that $\mA$ is isomorphic to the fixed-point bundle of $\check{\mF}_\mA$.
\end{prop}

\begin{proof}
By hypothesis $\mA$ has transition maps $\jmath^{\, a}_b \in QG$, $b \in \Sigma_1(\Delta)$,
fulfilling (\ref{eq.A4}) and defining a cohomology class $[\jmath^{\, a}] \in H^1(\Delta,QG)$.
Since the inclusion map $G \to NF_\bullet$ is injective, by Lemma \ref{lem.B1.1} the map $\mu_{2,*}$ of is one-to-one
and we may take $\tau = (\upsilon,\delta)$ in $\mu_2^{-1}(\jmath^{\, a})$ a twisted connection.
We denote the $C^*$-gerbe associated to $\tau$ by $\check{\mF}_\mA$.
By construction 
$\upsilon_b \, {\mathrm{mod}} G =$ $\jmath^{\, a}_b$, $b \in \Sigma_1(\Delta)$, 
thus reasoning as in Cor.\ref{cor.B3.1} we conclude that the fixed-point bundle of $\check{\mF}_\mA$ is isomorphic to $\mA$.
Since $\mu_2$ is one-to-one up to cohomology, 
by Cor.\ref{cor.B3.1} any $C^*$-gerbe $\check{\mF}'$ with fixed-point bundle $\mA$ defines a cocycle cohomologous to $\tau$,
thus by Prop.\ref{prop.B3.1} we have $\check{\mF}' \simeq \check{\mF}_\mA$.
We conclude that $\check{\mF}_\mA$ is unique up to isomorphism.
\end{proof}

\begin{cor}
\label{cor.C1.1}
A $C^*$-bundle $\mF' = (F,\imath')_\Delta$ fulfils $\mA = \mF'^\alpha$ if, and only if, 
$\mF'$ is isomorphic to $\check{\mF}_\mA$ as a gerbe. 
%
%
\end{cor}

\paragraph{The case of Cuntz algebras.} 
Given $d \in \bN$, the Cuntz algebra $\mO_d$ is defined as the universal $C^*$-algebra
generated by isometries $\psi_1 , \ldots$ $\psi_d$ fulfilling the Cuntz relations
$\psi_h^*\psi_k = \delta_{hk}1$, $\sum_k \psi_k \psi_k^* = 1$.
The Cuntz algebra has two remarkable dynamics:
the faithful action $\alpha : G \to \Aut \mO_d$, where $G \subseteq \bU(d)$ is a compact Lie group,
and the canonical *-endomorphism $\sigma : \mO_d \to \mO_d$,
\begin{equation}
\label{eq.Od.dyn}
\alpha_g(\psi_h) \, := \, \sum_k g_{hk}\psi_k  \ \ , \ \ \sigma(t) \, := \, \sum_k \psi_k t \psi_k^* 
\ \ \ , \ \ \
\forall g=(g_{hk}) \in G \, , \, t \in \mO_d \ .
\end{equation}
The fixed-point algebra of $\mO_d$ under the $G$-action is denoted by $\mO_G$.
Since $\alpha_g \circ \sigma = \sigma \circ \alpha_g$ for all $g \in G$, 
the *-endomorphism $\sigma_G := \sigma|\mO_G :$ $\mO_G \to$ $\mO_G$ is well-defined, \cite{DR87}. 
Let $N_\alpha\mO_d$ denote the normalizer of $\alpha(G)$ in $\Aut\mO_d$;
a 2--group of interest is given by ${}^2G_\sigma = ( G \tto N_\sigma G )$, where
\[
N_\sigma G \, := \, \{ \beta \in N_\alpha \mO_d  : \beta \circ \sigma = \sigma \circ \beta  \} \ .
\]
Bundles with fibre $\mO_G$ and holonomy in $N_\sigma G/G$ have been studied in \cite[\S 4]{Vas12}, and two opposite situations have been illustrated:
\begin{enumerate}
\item $G = \bSU(d)$ and $\pi_1(\Delta) = \bZ$: there are bundles $\mA$ with fibres $A_o \simeq \mO_G$
      admitting several, non-isomorphic $C^*$-bundles $\mF$ such that $\mA = \mF^\alpha$.
      By the above results, all these $C^*$-bundles are isomorphic as gerbes.
\item $G \simeq \bU(1)$, $G \subset \bU(d)$, acting by scalar multiplication on $\bC^d$, and $H^2(\pi_1(\Delta),\bT) \neq 0$:
      there are bundles with fibres $A_o \simeq \mO_G$ that do not admit $C^*$-bundles having them as fixed-point bundles.
      Thus we only get the $C^*$-gerbe of Prop.\ref{prop.C1.1}.
\end{enumerate}
As we shall see (\S \ref{sec.D}), any section of a DR-presheaf defines a bundle with fibre $\mO_G$ of the above kind.

\paragraph{$C^*$-gerbes in quantum field theory.}
$C^*$-gerbes carrying actions by arbitrary compact groups have been constructed in \cite[\S 4.4]{VasQFT} by means of Wightman fields,
nevertheless at the present time their physical meaning is questionable, \S \ref{sec.D4}.
Instead, in the following lines we illustrate a construction involving the free Dirac field interacting with background potentials
closed as de Rham 1--forms \cite[\S 4.2]{VasQFT}. 
This situation is analogous to what happens in the Aharonov-Bohm effect,
where charged quantum particles "interact" with a classical potential,
which is closed (\emph{i.e.} with vanishing electromagnetic field) in a suitable region, the one that lies outside an ideally infinite solenoid.

Let $M$ denote a 4--dimensional globally hyperbolic spacetime and $\Delta$ the base of diamonds generating the topology of $M$ \cite{GLRV01,BR08}.
Then the spinor bundle $DM \to M$ is defined, with fibre $\bC^4$. 
We denote the space of smooth, compactly supported sections of $DM$ by $\mS(DM)$.
The free Dirac field on $M$ defines a Hilbert space $H$ and a map
\begin{equation}
\label{eq.Dirac}
\psi : \mS(DM) \to B(H) \ ,
\end{equation}
such that any functional
$s \mapsto \left\langle v , \psi(s) w \right\rangle$, $v,w \in H$,
is continuous under the Schwartz topology.
Given $o \in \Delta$, we consider the weak closure of the *-algebra generated by operators of the type 
$\psi(s)$, ${\mathrm{supp}}(s) \subseteq o$,
and denote it by $F_o$. Any $F_o$ is a Von Neumann algebra, and there are obvious inclusion *-morphisms
$\jmath_{\omega o} : F_o \to F_\omega$, $o \subseteq \omega$,
defining the $C^*$-precosheaf $\mF = (F,\jmath)_\Delta$. We call $\mF$ the \emph{Dirac precosheaf}.
We have a $\bU(1)$-action
\[
\alpha_\zeta(\psi(s)) \, := \, \zeta \psi(s) \, , \, \zeta \in \bU(1) \, , \, s \in \mS(DM) \ ;
\]
since $\alpha_\zeta(F_o) = F_o$ for all $o \in \Delta$, we conclude that $\mF$ a $\bU(1)$-precosheaf.
%
%
%
We denote the fixed-point bundle of $\mF$ by $\mA = (A,\imath)_\Delta$.

Let now $Z_{dR}^1(M)$ denote the space of closed de Rham 1--forms.
Each $\sA \in$ $Z_{dR}^1(M)$ can be interpreted as a vector potential giving rise 
to a trivial electromagnetic field $d\sA = 0$. If $H_{dR}^1(M) \neq 0$ then non-exact potentials of this kind exist.
%
%
%
Using the method of \cite[\S 4.2]{VasQFT}, we can associate to $\sA$ a real cocycle $\lambda \in Z^1(\Delta,\bR)$,
and in particular we write $\lambda_{\omega o} \equiv \lambda_b$ for $o \leq \omega$ and $b = ( \omega , o ; \omega )$.
We define the *-morphisms
\[
\jmath^{\, \lambda}_{\omega o} : F_o \to F_\omega
\ \ \ , \ \ \ 
\jmath^{\, \lambda}_{\omega o} := \alpha_{ \exp i \lambda_{\omega o} }
\ \ \ , \ \ \ 
o \subseteq \omega
\ ,
\]
which yield the $C^*$-precosheaf $\mF^\lambda = (F,\jmath^{\, \lambda})_\Delta$.
%
%
We remark that, since any $\jmath^{\, \lambda}_{\omega o}$ is the identity on elements of $A_o$,
any $\mF^{\, \lambda}$ defines the same fixed-point bundle $\mA$.

Let $\lambda'$ be the 1--cocycle defined by $\sA' \in Z_{dR}^1(M)$.
It is easily verified that if $\lambda$ and $\lambda'$ are not cohomologous then in general $\mF^\lambda$ and $\mF^{\lambda'}$ are not isomorphic.
Nevertheless we have the gerbe isomorphisms
\[
\eta_{\beta} : \mF^{\, \lambda} \to \mF^{\, \lambda'}
\ \ \ , \ \ \
\eta_o(t) := t
\ \ \ , \ \ \ 
\beta_b := \alpha_{ \exp i ( \lambda'_{\omega o} - \lambda_{\omega o} ) }
\ ,
\]
where $t \in \mF^{\, \lambda}_o$, $o \subseteq \omega$ and $b := ( \omega , o , \omega ) \in N_1(\Delta)$.
Thus all the precosheaves $\mF^{\, \lambda}$, $\lambda \in Z^1(\Delta,\bR)$, define a unique "Dirac gerbe" having fixed-point bundle $\mA$.


\subsection{Twisted $C^*$-dynamical systems}
\label{sec.C3}

In the present section we introduce a notion of twisted $C^*$-dynamical system.
The idea is that we have an ordinary $G$-dynamical system, on which a second group $\Pi$
acts by a twisted action which lies in the normalizer of the initial $G$-action.
Twisted $C^*$-dynamical systems in the usual sense \cite{ZM68,BS70,AdL87,PR89} arise when $G$ is the unitary group of the underlying $C^*$-algebra.

Of course the techniques that we use are not new, and are borrowed from the well-established theory of (twisted) $C^*$-dynamical systems \cite{Wil}.
Yet, besides considering a generic group $G$, we introduce a flexible notion of morphism
and discuss the relation between twisted $C^*$-dynamical systems and non-abelian group cohomology, 
in particular by relaxing the properties of the twisting map $\gamma$ in the second of (\ref{eq.B3.4}).

\

Let $\alpha : G \to \Aut A$ be a $C^*$-dynamical system, with $G$ not necessarily locally compact.
As we saw in \S \ref{sec.A0}, we have the normalizer 2--group 
${}^2G_\alpha = (\alpha(G) \tto N_\alpha A)$.
For convenience sometimes we write $\alpha_g \equiv \alpha(g) \in \Aut A$, $g \in G$.
Moreover, let $\Pi$ denote a locally compact group.
\begin{defn}
With the above notation, given Borel maps
$\nu : \Pi \to N_\alpha A \subseteq \Aut A$, $\gamma : \Pi \times \Pi \to G$,
we say that $\mA = (A,\alpha,\nu,\gamma)$ is a \textbf{twisted $C^*$-dynamical system} over $G,\Pi$ whenever 
\begin{equation}
\label{eq.tds}
\nu_1 \, = \, id
\ \ \ , \ \ \
\nu_l \circ \nu_m \, = \, \alpha_{\gamma(l,m)} \circ \nu_{lm}
\ \ \ , \ \ \
\forall l,m \in \Pi \ .
\end{equation}
\end{defn}
Note that if $\alpha_{\gamma(l,m)} = id$ for all $l,m \in \Pi$ then $\nu$ is a group morphism.
When $\nu$ is also continuous it defines an action $\nu : \Pi \to \Aut A$ in the ordinary sense.
The technically convenient requirement of $\nu , \gamma$ being Borel rather that continuous is standard in the theory
of "usual" twisted $C^*$-dynamical systems.

By (\ref{eq.tds}) we have that $\alpha_{\gamma(l,m)}$ fulfils
\begin{equation}
\label{eq.tds'}
\alpha_{\gamma(l,m)} \circ \alpha_{\gamma(lm,n)} \, = \, \wa{\nu}_l(\alpha_{\gamma(m,n)}) \circ \alpha_{\gamma(l,mn)}
\ \ \ , \ \ \ 
\alpha_{\gamma(l,1)} \, = \, \alpha_{\gamma(1,l)} \, = \, id \ ,
\end{equation}
for all $l,m \in \Pi$, where $\wa{\nu}_l(\alpha_g) :=$ $\nu_l \circ \alpha_g \circ \nu_l^{-1} \in$ $\alpha(G)$ for all $l$ and $g$.
Thus $(\nu,\alpha_\gamma) \in$ $\check{Z}^2(\Pi,{}^2G_\alpha)$, where $\alpha_\gamma := \alpha \circ \gamma$.
Note that in our approach it is the map $\alpha_\gamma$, rather than $\gamma$, that plays a role.
The idea is that in general $\gamma$ is an auxiliary tool which may be arbitrarily chosen at the only condition that (\ref{eq.tds}) is fulfilled,
see Remark \ref{rem.gamma} below.
%

\paragraph{Unitary actions: inner, Busby-Smith and spatial systems.}
Any $C^*$-algebra $A$ carries the adjoint action
{\footnote{Recall, \S \ref{sec.A}, that $U(A)$ is by definition the unitary group of $M(A)$.}}
$\ad : U(A) \to \Aut A$.
Here, $\ad$ is continuous with respect to the strict topology on $U(A)$ generated by the seminorms
$p_t(u) := \| ut \| + \| tu \|$, $t \in A$.
Since $N_\ad A = \Aut A$, the normalizer 2--group is       
${}^2U(A)_\ad =$ $( \Inn A \to \Aut A )$, where $\Inn A := \ad(U(A))$.      
The corresponding systems $\mA = (A,\ad,\nu,\gamma)$ are given by maps      
\[      
\nu : \Pi \to \Aut A \ \ \ , \ \ \ \gamma : \Pi \times \Pi \to U(A) \ ,      
\]      
such that       
\begin{equation}      
\label{eq.inn.tds}      
\nu_1 \, = \, id      
\ \ \ , \ \ \      %
\nu_l \circ \nu_m \, = \, \ad_{\gamma(l,m)} \circ \nu_{lm}
\ \ \ , \ \ \      %
\forall l,m \in \Pi \ .
\end{equation}      
That is,      
$(\nu,\ad_\gamma) \in$ $\check{Z}^2(\Pi,{}^2U(A)_\ad)$, $\ad_\gamma := \ad \circ \gamma$,
%
%
%
and in this case we say that $\mA$ is \emph{inner}.      
In particular $\mA$ is inner when $(\nu,\gamma) \in$ $\check{Z}^2(\Pi,{}^2\Aut A)$,       
where the 2--group ${}^2\Aut A$ is defined in \S \ref{sec.A0}:      
this means that a stronger version of (\ref{eq.tds'}) holds,      
\begin{equation}      
\label{eq.tds''}      
\gamma(l,m) \cdot \gamma(lm,n) \, = \, \nu_l(\gamma(m,n)) \cdot \gamma(l,mn)      
\ \ \ , \ \ \       
\gamma(l,1) \, = \, \gamma(1,l) \, = \, 1 \ ,      
\end{equation}      
for all $l,m \in \Pi$.      
In this case we get the usual notion of \emph{Busby-Smith twisted $C^*$-dynamical system} \cite[Def.2.1]{PR89}.

\begin{ex}[Twisted group $C^*$-algebras]
\label{ex.tgC*}
Let $G$ be compact and $(\nu,\gamma) \in \check{Z}^2(\Pi,{}^2\Aut G)$.
Then there is a faithful action $\Aut G \to \Aut C^*G$ and we identify $\Aut G$ with its image in $\Aut C^*G$.
In this way $\mC^*G := ( C^*G , \ad , \nu ,\gamma )$ becomes a twisted $C^*$-dynamical system.
Considering the morphism $\phi : G \to U(C^*G)$ \cite[\S 7.1]{Ped} we get the cocycle 
$(\nu,\phi\circ\gamma) \in \check{Z}^2(\Pi,{}^2\Aut  C^*G)$, 
defining a system of Busby-Smith type.
\end{ex}

\begin{ex}[Weak actions]
Let $\mA = (A,\ad,\nu,\gamma)$ be of Busby-Smith type. Then the triple $(A,\nu,\gamma,1)$,
where $1 \in M(A)$ is the identity, is a weak action in the sense of \cite[\S 3.13]{BMZ13}.
\end{ex}

\begin{ex}[Strict crossed module actions]
Let $n \in \bN$ and $\Pi := \bR^n$ (additive group). 
Given a $C^*$-algebra $A$ and continuous maps
$\nu : \Pi \to \Aut A$, $u : \Lambda^2 \bR^n \to U(A)$,
fulfilling the properties
\begin{equation}
\label{eq.MP}
\nu_0 \, = \, id      
\ \ , \ \
\nu_l \circ \nu_m \, = \, \ad_{u(l \wedge m)} \circ \nu_{l+m}
\ \ , \ \ 
l,m \in \Pi \ ,
\end{equation}
we have that $\mA = (A,\ad,\nu,\gamma)$, with $\gamma(l,m) := u(l \wedge m)$,
is an inner twisted $C^*$-dynamical system.
In \cite[Theorem 4.8]{MP16} it is proved that any pair $(\nu,u)$ fulfilling (\ref{eq.MP})
(and some additional properties) defines a strict action $\beta : {}^2\mH \to {}^2\Aut A$ 
(Green $C^*$-dynamical system),
where the 2--group ${}^2\mH$ is the one defined in \cite[Eq.1]{MP16}.
\end{ex}

A particular case of Busby-Smith system is the following.
Given the $C^*$-algebra $A$ we have the 2--group ${}^1U(A)$, see (\ref{eq.1G}).
Then any unital Borel map $u : \Pi \to U(A)$, $u(1) = 1$, defines the cocycle
$(u,du) \in$ $\check{Z}^2(\Pi,{}^1U(A))$,
where
\[
du(l,m) \, := \, u_l u_m u_{lm}^*
\ \ \ , \ \ \ 
\forall l,m \in \Pi
\ .
\]
Since $u$ is Borel, and since products in $\Pi$ and $U(A)$ are continuous, 
we have that $du$ is Borel.
Clearly $(\ad_u ,du) \in$ $\check{Z}^2(\Pi,{}^2\Aut A)$, $\ad_u := \ad \circ u$,
and $\mA = ( A , \ad , \ad_u , du )$ is a Busby-Smith twisted $C^*$-dynamical system.
In this case we say that $\mA$ is \emph{spatial}.

\paragraph{Morphisms.}
Let $\mA' = (A',\alpha',\nu',\gamma')$ a twisted $C^*$-dynamical system over $G',\Pi$.
A \emph{morphism}
$\eta^\phi : \mA \to \mA'$
is given by a *-morphism $\eta : A \to A'$ and a Borel group morphism $\phi : G \to G'$, such that
\begin{equation}
\label{eq.C3.1}
\eta \circ \nu_l \, = \, \nu'_l \circ \eta
\ \ \ , \ \ \
\eta \circ \alpha_g \, = \, \alpha'_{\phi(g)} \circ \eta
\ ,
\end{equation}
for all $l,m \in \Pi$ and $g \in G$.
The notion of morphism defines the category $\tdyn_\Pi$.
We may keep $G$ fixed and consider morphisms with $\phi = id$, writing $\eta \equiv \eta^{id}$:
this yields the subcategory $\tdyn_{\Pi,G}$.

Now, the first of (\ref{eq.C3.1}) implies $\eta \circ \alpha_{\gamma(l,m)} =$ $\alpha'_{ \gamma'(l,m) } \circ \eta$,
and writing the second of (\ref{eq.C3.1}) for $g = \gamma(l,m)$ yields
\begin{equation}
\label{eq.C3.1a}
\alpha'_{ \phi(\gamma(l,m)) } \circ \eta \ = \ \alpha'_{ \gamma'(l,m) } \circ \eta
\ \ \ , \ \ \ 
l,m \in \Pi
\ ;
\end{equation}
thus a stronger version of (\ref{eq.C3.1a}) is 
\begin{equation}
\label{eq.C3.1a'}
\phi(\gamma(l,m)) \ = \ \gamma'(l,m) \ \ \ , \ \ \ l,m \in \Pi \ ,
\end{equation}
that will be required in the cases in which we are interested to preserve the map $\gamma$.

It is instructive to write (\ref{eq.C3.1}) and (\ref{eq.C3.1a}) when 
$\mA' = (A' , \ad , \ad_u , du)$
is spatial:
\begin{equation}
\label{eq.C3.1b}
\eta(\nu_l(t)) \, = \,  u_l \, \eta(t) \, u_l^* 
\ \ \ , \ \ \ 
\eta \circ \alpha_g \, = \, \ad_{\phi(g)} \circ \eta
\ \ \ , \ \ \ 
\ad_{\phi(\gamma(l,m))} \circ \eta \, = \,  \ad_{du(l,m)} \circ \eta 
\ ,
\end{equation}
for all $t \in A$, $l,m \in \Pi$ and $g \in G$.
We say that $\eta^\phi$ is \emph{spatial} whenever (\ref{eq.C3.1a'}) holds:
\begin{equation}
\label{eq.C3.1c}
\phi(\gamma(l,m)) \ = \  du(l,m)
\ \ \ , \ \ \ 
\forall l,m \in \Pi
\ .
\end{equation}
\begin{rem}
\label{rem.gamma}
Let $\mA' = (A,\alpha,\nu,\gamma')$ be a twisted $C^*$-dynamical system. 
Then, for any map $\gamma : \Pi \times \Pi \to G$ such that $\alpha_{\gamma'} = \alpha_\gamma$,
the identity $id \in \Aut A$ defines an isomorphism $\iota : \mA \to \mA'$, 
where $\mA := (A,\alpha,\nu,\gamma)$.
As an example, we may take
\[
\gamma(l,m) \, := \, \gamma'(l,1)^{-1} \gamma'(1,1) \gamma'(1,m)^{-1} \gamma'(l,m) \ \ \ , \ \ \ \forall l,m \in \Pi \ .
\]
By the second of (\ref{eq.tds'}) we have that $\alpha_\gamma = \alpha_{\gamma'}$, with $\gamma$ fulfilling
\begin{equation}
\label{eq.gamma}
\gamma(l,1)  \, = \, \gamma(1,l) \, = \, 1 \ \ \ , \ \ \ \forall l \in \Pi \ .
\end{equation}
Thus, up to isomorphism, any $\mA = (A,\alpha,\nu,\gamma)$ is such that $\gamma$ fulfils (\ref{eq.gamma}). 
%
%
\end{rem}
\begin{rem}
\label{rem.BS}
Let $\mA = (A,\ad,\nu,\gamma)$ denote a Busby-Smith twisted $C^*$-dynamical system.
We consider the crossed product $C^*$-algebra $B := A \times_{\nu,\gamma} \Pi$ \cite[Def.2.4,Rem.2.5]{PR89}.
Then by definition there are a unital Borel map $u : \Pi \to U(B)$ and a non-degenerate *-morphism $\eta : A \to M(B)$,
defining the spatial morphism $\eta^\phi : \mA \to \mM(\mB) := (M(B),\ad,\ad_u,du)$ with $\phi = \eta |_{U(A)}$.
\end{rem}

\

We now give a different notion of morphism coherent with the cohomology relation (\ref{eq.B3.4a}).
This will turn out to be a weaker version of the notion of \emph{exterior equivalence} \cite[Def.3.1]{PR89}.
\begin{defn}
\label{def.tmor}
Let
$\mA = (A,\alpha,\nu,\gamma)$, $\mA' = (A',\alpha',\nu',\gamma')$
be twisted $C^*$-dynamical systems.
A morphism $\eta_\kappa : \mA \to \mA'$ is given by a *-morphism $\eta : A \to A'$ and a Borel map
$\kappa : \Pi \to G$
such that, for all $l,m \in \Pi$ and $g \in G$,
\begin{equation}
\label{eq.C3.mk}
\eta \circ \alpha_g \, = \, \alpha'_g \circ \eta 
\ \ \ , \ \ \ 
\alpha'_{\kappa(l)} \circ \eta \circ \nu_l \ = \ \nu'_l \circ \eta \ .
\end{equation}
This defines the category $\tdyn^c_{\Pi,G}$.
\end{defn}
Note that, for simplicity, here we did not use the term $w \in N$ appearing in (\ref{eq.B3.4a}).

Let now in particular
$\mA = (A,\alpha_0,\alpha_1\circ\bs\nu,\gamma)$, $\mA' = (A',\alpha'_0,\alpha'_1\circ\bs\nu',\gamma')$
be systems having the same structure 2--group ${}^2G = (G \tto N)$.
If $A = A'$ and $\eta = id$, then $\alpha = \alpha'$ and the second of (\ref{eq.C3.mk}) implies
\[
\alpha_{\gamma'(l,m)} \ = \ \wa{\nu}'_l(\alpha(\kappa_m)) \circ \alpha(\kappa_l) \circ \alpha_{\gamma(l,m))} \circ \alpha(\kappa_{lm})^{-1} \ ,
\]
for all $l,m \in \Pi$. Thus a stronger version of the previous equality is
\begin{equation}
\label{eq.ext-eq}
\gamma'(l,m) \ = \ \wa{\bs\nu}'_l(\kappa_m) \, \kappa_l \, \gamma(l,m)) \, \kappa_{lm}^{-1} \ ,
\end{equation}
corresponding to the second of (\ref{eq.B3.4a}).

Now, we make the further assumption that $\mA , \mA'$ are Busby-Smith.
If $\mA , \mA'$ are exterior equivalent by means of a map $\kappa : \Pi \to U(A)$, see \cite[Def.3.1]{PR89},
then $id_\kappa$ defines a morphism of the type (\ref{eq.C3.mk}), fulfilling the stronger property (\ref{eq.ext-eq}).

\

Of course, we may combine (\ref{eq.C3.1}), (\ref{eq.C3.mk}) and define a notion of morphism involving both $\phi$ and $\kappa$,
nevertheless here we will do not need such a generality.

\paragraph{Non-abelian cocycles and twisted systems.}
Let ${}^2G = (G \tto N)$ be a topological 2--group,
$(\bs\nu,\gamma) \in \check{Z}^2(\Pi,{}^2G)$ with $\bs\nu , \gamma$ Borel maps
{\footnote{Here we use a bold character for $\bs\nu : \Pi \to N$ to emphasize that it \emph{does not} take values in $\Aut A$,
           differently from the map $\nu$ appearing in the previous lines.}},
and $A$ a $C^*$-algebra.
If ${}^1\Aut A$ is the 2--group defined as in (\ref{eq.1G}) and 
$\alpha : {}^2G \to {}^1\Aut A$
is a morphism, then the  twisted $C^*$-dynamical system 
\begin{equation}
\label{eq.A1a}
{}^1\mA_\alpha \, := \, ( A , \alpha_0 , \alpha_1 \circ \bs\nu , \gamma )
\end{equation}
is defined.
If we have a Green $C^*$-dynamical system 
$\alpha : {}^2G \to {}^2\Aut A$
\S \ref{sec.A0}, then we can define 
${}^2\mA_\alpha := ( A , \ad , \alpha_1 \circ \bs\nu , \alpha_0 \circ \gamma )$
which is of Busby-Smith type.

In the following definition we relax the properties defining (\ref{eq.A1a}), 
by not requiring that the second and the third of (\ref{eq.B3.4}) necessarily hold for $\gamma$:
\begin{defn}
\label{def.sg}
Let ${}^2G = (G \tto N)$ be a topological 2--group and $\alpha : {}^2G \to {}^1\Aut A$ a morphism.
Given a twisted $C^*$-dynamical system $\mA = (A,\alpha_0,\nu,\gamma)$ over $G,\Pi$,
we say that ${}^2G$ is a \textbf{structure 2--group} for $\mA$ whenever there is a mapping
$\bs\nu : \Pi \to N$
such that 
\begin{equation}
\label{eq.sg}
\nu \ = \ \alpha_1 \circ \bs\nu
\ \ \ , \ \ \
\bs\nu_1 = 1
\ \ \ , \ \ \ 
\bs\nu_l \, \bs\nu_m \, = \, \unl{\gamma(l,m)} \, \bs\nu_{lm}
\ \ \ , \ \ \ 
\forall l,m \in \Pi
\ .
\end{equation}
\end{defn}
Clearly any system of the type (\ref{eq.A1a}) has ${}^2G$ as a structure group.
More in general, we shall say that $\mA'$ has structure 2--group ${}^2G$ whenever 
there is a morphism $\eta^\phi : \mA' \to \mA$, with $\mA$ as in the previous definition
and $\eta$ a *-isomorphism.
\begin{rem}
\label{rem.sg}
Let $\mA = (A,\alpha,\nu,\gamma)$ be a twisted $C^*$-dynamical system.
Set
${}^1\mA_\iota = ( A , \iota_0 , \iota_1 \circ \nu , \alpha_\gamma )$,
where 
$\iota_0 : \alpha(G) \to \Aut A$ and $\iota_1 : N_\alpha A \to \Aut A$
are the inclusion maps and define the 2--group morphism $\iota : {}^2G_\alpha \to {}^1\Aut A$. 
Then $\bs\nu := \iota_1 \circ \nu = \nu$ obviously fulfils (\ref{eq.sg}) and $(\bs\nu,\alpha_\gamma) \in \check{Z}^2(\Pi,{}^2G_\alpha)$.
We have the morphism $id^\alpha : \mA \to {}^1\mA_\iota$, where $id \in \Aut A$ is the identity.
This proves that any twisted $C^*$-dynamical system has "at worst" ${}^2G_\alpha$ as a structure 2--group.
\end{rem}

\begin{ex}
\label{eq.sg.inn}
Let $\mA = (A,\ad,\nu,\gamma)$ be inner.
Then (\ref{eq.sg}) is fulfilled with the 2--group 
${}^2\Aut A$,
$\alpha_0 : U(A) \to \Aut A$ the adjoint action,
$\alpha_1 : \Aut A \to \Aut A$ the identity map,
and
$\bs \nu = \nu$.
Thus $\mA$ has ${}^2\Aut A$ as a structure 2--group.
\end{ex}

\begin{rem}
\label{rem.BMZ1}
Let $\mA = (A,\alpha_0,\alpha_1\circ\bs\nu,\gamma)$ be a twisted $C^*$-dynamical system with structure 2--group ${}^2G = (G \tto N)$.
Then we may restrict $\alpha$ to a morphism
$\alpha_{\bs\nu} : {}^2G_{\bs\nu} \to {}^1\Aut A$,
where ${}^2G_{\bs\nu} = (G \tto G \vee_{\bs\nu} \Pi)$ is defined by taking the subgroup $G \vee_{\bs\nu} \Pi \subseteq N$
generated by elements of the type
$\unl{g} \bs\nu_l$, $g \in G$, $l \in \Pi$,
with the maps $G \to G \vee_{\bs\nu} \Pi$ and $G \vee_{\bs\nu} \Pi \to \Aut G$ obtained by restriction.
\end{rem}

\paragraph{Embeddings in Busby-Smith systems.}
We now give some conditions to embed a twisted system into a spatial (Busby-Smith) one.
We shall see that in general, to get the desired embedding, one pays the price of
passing from $G$ to an image of $G$ through a morphism.
\begin{lem}
\label{lem.C3.1}
Let
$\mA = (A,\alpha,\nu,\gamma)$
such that the group $S \subseteq \Aut A$ generated by $\alpha(G)$, $\nu(\Pi)$  is locally compact.
Then $\mA$ admits a spatial morphism $\eta^\phi : \mA \to \mB$, with $\eta$ injective.
\end{lem}

\begin{proof}
The inclusion map $\iota : S \to \Aut A$ makes $A$ an $S$-dynamical system, 
and we consider the crossed product $A \rtimes^\iota S$ \cite[\S 2.3-4]{Wil}. 
We set $B:= M(A \rtimes^\iota S)$.
By \cite[Prop.2.34]{Wil} there are a faithful *-morphism 
$\eta : A \to B$
and a strictly continuous group morphism
$\upsilon : S \to U(B)$,
such that $\upsilon_{id} = 1$ and $\upsilon_r \eta(t) \upsilon_r^* = \eta(r(t))$ for all $r \in S$, $t \in A$.
Then we define
$u_l := \upsilon_{\nu_l} \in$ $U(B)$, $l \in \Pi$.
Since $\upsilon$ is continuous and $\nu$ is Borel we have that $u$ is Borel, and $u_1 = 1$.
This makes $\mB = (B,\ad,\ad_u,du)$ a spatial twisted $C^*$-dynamical system. 
Setting 
$\phi(g) := \upsilon_{\alpha_g}$ for all $g \in G$,
we find
\[
\phi(\gamma(l,m)) \, = \, 
\upsilon_{ \alpha_{\gamma(l,m)} } \, = \, 
\upsilon_{\nu_l} \upsilon_{\nu_m} \upsilon_{\nu_{lm}}^* \, = \,
du(l,m)
\ ,
\]
and this yields the desired spatial morphism
$\eta^\phi : \mA \to \mB$.
\end{proof}

\begin{lem}
\label{lem.C3.2}
Let
$\mA = (A,\alpha,\nu,\gamma)$
with $\alpha : G \to \Aut A$ faithful.
Then there are a topological 2--group ${}^2G = (G \tto N)$ such that $\alpha$ extends to a continuous morphism
$\alpha : {}^2G \to {}^1\Aut A$,
and a cocycle $(\bs\nu,\gamma) \in \check{Z}^2(\Pi,{}^2G)$
such that $\mA$ is isomorphic to ${}^1\mA_\alpha$.
\end{lem}

\begin{proof}
We define $N := N_\alpha A$. Afterwards, we define the maps 
$i : G \to N$, $i(g) := \alpha(g)$,
and
$\beta : N \to \Aut G$, $\beta_u(g) := \alpha^{-1}(u \circ \alpha(g) \circ u^{-1})$, $u \in N$, $g \in G$.
This yields the 2--group ${}^2G = (G \tto N)$, and the morphism
$\alpha_0(g) := \alpha(g)$,
$\alpha_1(u) := u \in \Aut A$.
Then we define $\bs\nu_l := \nu_l \in N$ for all $l \in \Pi$,
and this yields the cocycle $(\bf\nu,\gamma)$.
It is then clear that the identity $\iota \in \Aut A$ defines the desired isomorphism $\mA \simeq {}^1\mA_\alpha$.
\end{proof}

\begin{lem}
\label{lem.C3.3}
Let $G$ be locally compact and 
$\mA = (A,\alpha_0,\alpha_1\circ\bs\nu,\gamma)$
with structure group ${}^2G = (G \tto N)$.
Then there is a morphism $\eta^\phi : \mA \to \mB$, with $\eta$ faithful and
$\mB = (B,\ad,\bar\nu,\bar\gamma)$ 
inner, such that
\begin{equation}
\label{eq.lem.C3.3}
\phi(\wa{\bs\nu}_l(g)) \ = \ \bar\nu_l(\phi(g))
\ \ \ , \ \ \ 
g \in G \, , \, l \in \Pi \ .
\end{equation}
In particular, if $(\bs\nu,\gamma) \in \check{Z}^2(\Pi,{}^2G)$, 
then $\mA = {}^1\mA_\alpha$ and $\mB$ is of Busby-Smith type.
\end{lem}

\begin{proof}
We consider the crossed product $A \rtimes^\alpha G$ \cite[\S 2.3-4]{Wil} (here for simplicity we write $\alpha$ for $\alpha_0$),
and define $B := M(A \rtimes^\alpha G)$.
By construction, there are faithful morphisms
$\eta : A \to B$  and $\phi : G \to U(B)$,
such that 
$\phi(g) \eta(t) \phi(g)^* =$ $\eta(\alpha_g(t))$
for all $t \in A$ and $g \in G$.
For any $l,m \in \Pi$ and $g \in G$ we set
$\nu_l :=$ $\alpha_1(\bs\nu_l) \in \Aut A$,
and
$\bar\gamma(l,m) :=$ $\phi(\gamma(l,m)) \in$ $U(B)$.
Since
\[
\eta( \nu_l(\alpha_g(t)) ) \ = \
\eta \circ \wa{\nu}_l(\alpha_g) \circ \nu_l(t) \ = \
\eta \circ \alpha_0(\wa{\bs\nu}_l(g)) \circ \nu_l(t) \ = \
\ad \phi(\wa{\bs\nu}_l(g)) \circ \eta \circ \nu_l (t)
\ ,
\]
the pair $( \eta \circ \nu_l , \phi \circ \wa{\bs\nu}_l )$ defines a covariant morphism
of $(A,\alpha)$ into $B$, thus by universal property of the crossed product it defines
a *-morphism
$\bar\nu_l : A \rtimes^\alpha G \to B$.
%
The fact that $\bar\nu_l$ actually extends to a *-automorphism of $A \rtimes^\alpha G$ follows by the argument of
\cite[Prop.3.11]{Wil}, having formed the semidirect product $G \rtimes N$.
By standard properties of multiplier algebras, we extend $\bar\nu_l$ to an automorphism of $B$,
that we denote again by $\bar\nu_l$.
Since $\eta$, $\phi$ and $\bs\nu$ are continuous, and $\nu$ is Borel, we have that the mapping
$l \mapsto \bar\nu_l$
is Borel. Moreover, 
\[
\bar\nu_l \circ \bar\nu_m ( \, \eta(t) \phi(g) \, )
\ = \ 
\eta( \alpha_{\gamma(l,m)} \circ \nu_{lm}(t)) \, \, \phi( \ad \gamma(l,m) \circ \wa{\bs\nu}_{lm}(g) )
\ = \
\ad \bar\gamma(l,m) \circ \bar\nu_{lm} \, ( \, \eta(t) \phi(g) \, )
\ .
\]
Thus $\mB = (B,\ad,\bar\nu,\bar\gamma)$ is the desired inner system.
Finally, if $(\bs\nu,\gamma) \in \check{Z}^2(\Pi,{}^2G)$ then (\ref{eq.B3.4}) holds and $\bar\gamma(l,m)$ fulfils (\ref{eq.tds''}),
implying that $\mB$ is Busby-Smith.
\end{proof}

\begin{rem}
\label{rem.BMZ2}
In the situation of the previous Lemma, we have the Green $C^*$-dynamical system
\[
\beta : {}^2G_{\bs\nu} \to {}^2\Aut B \ \ , \ \ \beta_1(g) := \phi(g) \ \ , \ \ \beta_2(\unl{g} \bs\nu_l) := \ad_{\phi(g)} \circ \bar\nu_l \ ,
\]
where ${}^2G_{\bs\nu}$ is defined in Remark \ref{rem.BMZ1}.
That $\beta$ is actually a 2--morphism follows by the definition
$\beta_2(\unl{g}) =$ $\ad_{\phi(g)} =$ $\ad_{\beta_1(g)}$
and by (\ref{eq.lem.C3.3}).
\end{rem}

\begin{cor}
\label{cor.C3.3}
Let
$\mA = (A,\alpha,\nu,\gamma)$
such that $\alpha(G)$ is locally compact.
Then there is a morphism $\eta^\phi : \mA \to \mB$,
with $\mB$ of Busby-Smith type and $\eta$ injective.
\end{cor}

\begin{proof}
By Remark \ref{rem.sg} we have that ${}^2G_\alpha$ is a structure 2--group for $\mA$
and $(\nu,\alpha_\gamma) \in \check{Z}^2(\Pi,{}^2G_\alpha)$. Thus we can apply the previous Lemma.
\end{proof}

\paragraph{Covariant representations.}
Let $\mA = (A,\alpha,\nu,\gamma)$ be a twisted $C^*$-dynamical system.
Given a Hilbert space $H$ and a Borel map $\upsilon : \Pi \to U(H)$, $\upsilon(1) = 1$,
a \emph{covariant representation} of $\mA$, denoted by $\pi^{\rho,\upsilon}$, is by definition a morphism
\[
\pi^\rho : \mA \to \mB(\mH)
\ \ \ \ , \ \ \ \
\mB(\mH) = ( B(H) , \ad , \ad_\upsilon , d\upsilon )
\ ,
\]
with $\pi$ non-degenerate. Explicitly, we have a representation
$\rho : G \to U(H)$
such that, for all $l \in \Pi$, $g \in G$,
\begin{equation}
\label{eq.covrep2}
\ad_{\upsilon_l} \circ \pi = \pi \circ \nu_l
\ \ \ , \ \ \ 
\ad_{\rho(g)} \circ \pi = \pi \circ \alpha_g
\ .
\end{equation}
%
%
We say that $\pi^{\rho,\upsilon}$ is \emph{spatial} whenever
\begin{equation}
\label{eq.covrep1}
d\upsilon(l,m) \, := \, \upsilon_l \upsilon_m \upsilon_{lm}^* \, = \, \rho(\gamma(l,m)) \ ,
\end{equation}
for all $l,m \in \Pi$. Some remarks follow.
\begin{itemize}
      %
\item Let $\mA = ( A , \ad , \ad_u , du )$ be spatial and $\pi : A \to B(H)$ a non-degenerate representation.
      Then defining 
      $\mB(\mH) = ( B(H) , \ad , \ad_\upsilon , d\upsilon )$,
      where $\upsilon_l := \pi(u_l)$, $l \in \Pi$, and $\rho := \pi | U(A)$, one gets the spatial covariant representation 
      $\pi^{\rho,\upsilon} : \mA \to \mB(\mH)$.
      As a consequence, if a twisted $C^*$-dynamical system has non-trivial (faithful) spatial morphisms,
      then it has non-trivial (faithful) spatial covariant representations.
\item Let $\mA$ be of Busby-Smith type and $\pi^{\rho,\upsilon}$ a spatial covariant representation with $\rho = \pi | U(A)$.
      Then $\pi$ and $\upsilon$ define a covariant representation in the sense of \cite[Def.2.3]{PR89}.
\end{itemize}
For Busby-Smith twisted $C^*$-dynamical systems it is customary  to consider more generally representations
over Hilbert modules rather than Hilbert spaces. A recent account on this topic is \cite{BC}.
Following this line, we define a (spatial) representation of $\mA = (A,\alpha,\nu,\gamma)$ over the Hilbert module $X$ as a (spatial) morphism
\[
\pi^\phi : \mA \to \mB(\mX) \ \ \ , \ \ \ \mB(\mX) = (B(X),\ad,\ad_u,du) \ ,
\]
where $B(X)$ the $C^*$-algebra of adjointable, right module operators on $X$ and 
$u : \Pi \to U(X)$ is a unital Borel map. 
Covariant representations over Hilbert modules in the sense of \cite{BC} correspond to the case where $\mA$ is of Busby-Smith type, 
$\pi^\phi$ is spatial, and $\phi = \pi | U(A)$.

\begin{ex}[The regular representation]
\label{ex.regrep}
For simplicity here we assume that $\Pi$ is unimodular
(otherwise, the following definition of $u_l$ would be affected by the modular function, see \cite{BC} and related references).
Let $\mA = ( A , \ad , \nu , \gamma )$ be an inner $C^*$-dynamical system. 
We consider the Hilbert $A$-module $X := L^2_\lambda(\Pi,A)$ of functions $f : \Pi \to A$, square summable with respect to the left Haar measure $\lambda$;
the $A$-valued scalar product on $X$ is given by 
$(f,f') :=$ $\int f(l)^* f'(l) \, d\lambda(l)$, $f,f' \in X$.
We then define the *-morphism
\[
\pi : A \to B(X) 
\ \ \ : \ \ \  
\{ \pi(a)f \}(m) := \nu_m^{-1}(a) f(m) 
\ \ \ , \ \ \ 
a \in A \, , \, f \in X \, , \, m \in \Pi \ ,
\]
that is faithful by standard arguments. 
We also set $\rho := \pi | U(A)$, obtaining a norm-continuous map with values in the group $U(X)$ of unitary operators on $X$.
Then we define the map
\[
u : \Pi \to U(X) 
\ \ \ : \ \ \  
\{ u_l f \}(m) := \nu_m^{-1}( \gamma(l , l^{-1}m) ) f(l^{-1}m) 
\ \ \ , \ \ \ 
l,m \in \Pi \, , \, f \in X \ .
\]
The adjoint $u_l^*$ is dictated by the condition
$( f , u_l f' ) = (u_l^* f , f')$,
%
%
and is given by
\[
\{ u_l^*f \}(m) \, = \, \nu_{lm}^{-1}( \gamma(l,m) )^* f(lm) \ \ \ , \ \ \ l,m \in \Pi \, , \, f \in X \ .
\]
This implies $u_l^* u_l = u_l u_l^* = 1$, and $u$ takes actually values in $U(X)$.
By definition, we have
\[
\{ \pi(\nu_l(a)) f \}(m) \, = \, \nu_m^{-1}(\nu_l(a)) f(m)
\]
and
\[
\begin{array}{lcl}
\{ u_l \pi(a) u_l^* f \}(m) & = &
\nu_m^{-1}( \gamma(l , l^{-1}m) ) \cdot \nu_{l^{-1}m}^{-1}(a) \cdot \nu_m^{-1}( \gamma(l,l^{-1}m) )^* f(m) \ , \\ & = &
\{ \nu_m^{-1} \circ \ad_{ \gamma(l , l^{-1}m) } \circ \nu_m \circ \nu_{l^{-1}m}^{-1}(a) \} \, f(m) \ .
\end{array}
\]
Keeping in mind that $\nu_l \circ \nu_{l^{-1}m} = \ad_{ \gamma(l , l^{-1}m) } \circ \nu_m$, we conclude that
\[
\{ u_l \pi(a) u_l^* f \}(m) \, = \, \{ \pi(a)f \}(m) \ ,
\]
thus the first of (\ref{eq.covrep2}) is fulfilled.
Now, we have
\[
\{ u_1 f \}(m) \, = \, \nu_m^{-1}( \gamma(1,m) ) f(m) \ \ \ , \ \ \ m \in \Pi \ ,
\]
and, in order to get a covariant representation, we want $u$ unital. We have two situations:
\begin{enumerate}
\item $\mA$ is a generic inner twisted $C^*$-dynamical system. In this case we may apply the isomorphism
      of Remark \ref{rem.gamma}, say $\iota : \mA \to \bar\mA$, $\bar\mA =$ $(A,\ad,\nu,\bar{\gamma})$,
      with $\bar\gamma(1,m) \equiv 1$. Thus we define $\bar\pi$ in the same way as $\pi$, with $\bar\gamma$ in place of $\gamma$,
      and set $\pi' := \bar\pi \circ \iota$. In this way, $\pi'^{\rho,u}$ is a faithful covariant representation.
\item $\mA$ is of Busby-Smith type. In this case, the standard construction sketched in \cite{BC} holds,
      and $\pi^{\rho,u}$ is a \emph{spatial} covariant representation of $\mA$ over $X$.
\end{enumerate}
\end{ex}

\paragraph{The convolution algebra.}
We now give the definition of convolution algebra for a twisted $C^*$-dynamical system
$\mA = (A,\alpha_0,\alpha_1\circ\bs\nu,\gamma)$
with structure 2--group ${}^2G = (G \tto N)$.
Note that in general $(\bs\nu,\gamma)$ is not a cocycle, anyway applying Remark \ref{rem.gamma} we assume that 
the third of (\ref{eq.B3.4}) holds and the only property that is not ensured is the second of (\ref{eq.B3.4}).
The standard assumptions that $A$ is separable and $\Pi$ (locally compact and) second countable are made;
when $\mA$ is not inner, $G$ is assumed to be (locally compact and) second countable too.

Our construction relies on the property that the $G$-action can be "made inner". 
Thus, if $\mA$ is not already inner then we consider  the crossed product $A \rtimes^\alpha G$, and we have a twisted, \emph{inner} action
\[
\bar\nu_l \in \Aut (A \rtimes^\alpha G)
\ \ \ , \ \ \ 
\bar\gamma(l,m) \in U(A \rtimes^\alpha G)
\ \ \ , \ \ \ 
l,m \in \Pi
\ ,
\]
see the proof of Lemma \ref{lem.C3.3}.
Instead, if $\mA$ is inner, then the properties of the present and successive paragraphs
hold with $A$ in place of $A \rtimes^\alpha G$.
Note that $\bar\gamma$ fulfils the third of (\ref{eq.B3.4}), but not necessarily the second;
this last is fulfilled when $(\bs\nu,\gamma)$ is a cocycle, implying that 
$( A \rtimes^\alpha G , \ad , \bar\nu , \bar\gamma )$
is Busby-Smith (Lemma \ref{lem.C3.3}).

To simplify the exposition and avoid modular functions, we assume that $\Pi$ is unimodular.
This agrees with our future use of $\Pi$ as a (discrete countable, hence unimodular) homotopy group.

As a first step we consider the vector space $C_c(\Pi,A \rtimes^\alpha G)$ of compactly supported, continuous functions
from $\Pi$ into $A \rtimes^\alpha G$. Given $f,f' \in C_c(\Pi,A \rtimes^\alpha G)$, following \cite[\S 2]{BS70} we define
\begin{equation}
\label{eq.conv1}
\{ f \times f' \}(l) \, := \, 
\int_\Pi f(m) \, \bar\nu_m(f'(m^{-1}l)) \, \bar\gamma(m,m^{-1}l) \, dm \ ,
\end{equation}
\begin{equation}
\label{eq.conv2}
f^*(l) \, := \, 
\bar\gamma(l,l^{-1})^* \, \bar\nu_l(f(l^{-1}))^*  \ .
\end{equation}
We note that since $C_c(G,A)$ densely embeds in $A \rtimes^\alpha G$ \cite[Lemma 2.27]{Wil}, 
we may write our convolution algebra in terms of functions in $C_c( \Pi \times G , A )$,
anyway we prefer to use $C_c(\Pi,A \rtimes^\alpha G)$
since this approach allows a compact notation coherent with the case in which $\mA$ is inner.

The computations in \cite[\S 6]{BS70} show that $\times$ and * make $C_c(\Pi,A \rtimes^\alpha G)$ 
an algebra with involution.
Nevertheless when $\bar\gamma$ does not fulfils the second of (\ref{eq.B3.4}) the product $\times$ is not associative
and the involution *, even if it is idempotent, fails to be anti-multiplicative.
To get an ordinary *-algebra we need the second of (\ref{eq.B3.4}), that is, 
$(\bar\nu,\bar\gamma) \in$ $\check{Z}^2(\Pi,{}^2\Aut(A \rtimes^\alpha G))$,
and clearly this condition is implied by 
$(\bs\nu,\gamma) \in$ $\check{Z}^2(\Pi,{}^2G)$.

Let now
$\pi^{\rho,\upsilon} : \mA \to \mB(\mH)$, $\mB(\mH) = ( B(H) , \ad , \ad_\upsilon , d\upsilon )$,
denote a covariant representation. Then by the second of (\ref{eq.covrep2}) $\pi^\rho$ is a covariant
representation of $(A,\alpha)$, and by \cite[Prop.2.23]{Wil} we have the integrated representation
\[
\pi^{\rtimes \rho} : A \rtimes^\alpha G \to B(H)
\ \ , \ \
\pi^{\rtimes \rho}(h) \, := \, \int_G \pi(h(g)) \, \rho(g) \, dg
\ \ , \ \
h \in C_c(G,A)
\ .
\]
By \cite[Prop.2.34]{Wil}, $\pi^{\rtimes \rho}$ extends to $U(A \rtimes^\alpha G)$, with 
$\pi^{\rtimes \rho}( \bar\gamma(l,m) ) = \rho( \gamma(l,m) )$.
Iterating this procedure, we set
\begin{equation}
\label{eq.int-rep}
\pi^{\rtimes \rho,\upsilon}(f) \, := \, \int_\Pi \pi^{\rtimes \rho}(f(l)) \, \upsilon_l \, dl
\, \in B(H) \ , \ \ \ 
f \in C_c(\Pi,A \rtimes^\alpha G) 
\ .
\end{equation}
Now,
\[
\begin{array}{lcl}
\pi^{\rtimes \rho,\upsilon}(f) \, \ \pi^{\rtimes \rho,\upsilon}(f')
& = &
\int_{\Pi \times \Pi}  
\pi^{\rtimes \rho}(f(m)) \, \upsilon_m \, 
\pi^{\rtimes \rho}(f'(l)) \, \upsilon_l \, dmdl
= \\ \\ & = &
\int_{\Pi \times \Pi}
\pi^{\rtimes \rho}(f(m)) \, 
\ad \upsilon_m(\pi^{\rtimes \rho}(f'(l))) \, 
d\upsilon(m,l) \, \upsilon_{ml} \, 
dmdl 
= \\ \\ & = &
\int_{\Pi \times \Pi}
\pi^{\rtimes \rho}(f(m)) \, 
\pi^{\rtimes \rho}(\bar\nu_m(f'(m^{-1}l'))) \, 
d\upsilon(m,m^{-1}l') \, \upsilon_{l'} \, 
dmdl'
\ ,
\end{array}
\]
whilst
\[
\begin{array}{lcl}
\pi^{\rtimes \rho,\upsilon}(f \times f')
& = &
\int_{\Pi \times \Pi}
\pi^{\rtimes \rho}(f(m)) \, 
\pi^{\rtimes \rho}(\bar\nu_m(f'(m^{-1}l))) \, \rho(\gamma(m,m^{-1}l)) \, 
\upsilon_l \, 
dmdl
\, ,
\end{array}
\]
thus (\ref{eq.int-rep}) preserves the product when $\pi^{\rho,\upsilon}$ is spatial.
Analogously, using the fact that (\ref{eq.tds}) implies 
${\bar\nu}_l^{-1} =$ ${\bar\nu}_{l^{-1}} \circ \ad \bar\gamma(l,l^{-1})^*$,
we find
\[
\begin{array}{lcl}
\pi^{\rtimes \rho,\upsilon}(f)^*
& = &
\int_\Pi \upsilon_l^* \, \pi^{\rtimes \rho}(f(l))^*  \, dl
\, = \, 
\int_\Pi \pi^{\rtimes \rho}({\bar\nu}_l^{-1}(f(l)))^* \, \upsilon_l^* \, dl \, 
= \\ \\ & = &
\int_\Pi \pi^{\rtimes \rho}({\bar\nu}_{l^{-1}}^{-1}(f(l^{-1})))^* \, \upsilon_{l^{-1}}^* \, dl
\, = \, 
\int_\Pi \pi^{\rtimes \rho}( \ad \bar\gamma(l,l^{-1}) \circ \bar\nu_l (f(l^{-1})))^* \, \upsilon_{l^{-1}}^* \, dl \, 
= \\ \\ & = &
\int_\Pi  \pi^{\rtimes \rho}(\bar\gamma(l,l^{-1}))^* \, \pi^{\rtimes \rho}(\bar\nu_l (f(l^{-1})))^* \, 
\pi^{\rtimes \rho}(\bar\gamma(l,l^{-1})) \, \upsilon_{l^{-1}}^* \, dl \, ;
\end{array}
\]
if $\pi^{\rho,\upsilon}$ is spatial then 
$\pi^{\rtimes \rho}(\bar\gamma(l,l^{-1})) =$ $\upsilon_l \upsilon_{l^{-1}}$,
and the * is preserved,
\[
\pi^{\rtimes \rho,\upsilon}(f)^* 
\ = \ 
\int_\Pi  \pi^{\rtimes \rho}(\bar\gamma(l,l^{-1}))^* \, \pi^{\rtimes \rho}(\bar\nu_l (f(l^{-1})))^* \, 
\upsilon_l \, dl 
\ = \
\int_\Pi \pi^{\rtimes \rho}(f^*(l)) \, \upsilon_l \, dl 
\, .
\]
We note that the image of $C_c(\Pi,A \rtimes^\alpha G)$ through $\pi^{\rtimes \rho,\upsilon}$ is a *-algebra 
(associative and with anti-multiplicative involution),
thus when $\pi^{\rho,\upsilon}$ is spatial $\pi^{\rtimes \rho,\upsilon}$ kills the terms preventing $C_c(\Pi,A \rtimes^\alpha G)$ from being a *-algebra.
We have proved:
\begin{lem}
\label{lem.C3.4}
Let $\mA = (A,\alpha_0,\alpha_1\circ\bs\nu,\gamma)$ with structure 2--group ${}^2G$
and $\pi^{\rho,\upsilon} : \mA \to \mB(\mH)$ a covariant representation.
Then (\ref{eq.int-rep}) is a linear map, and a *-morphism when $\pi^{\rho,\upsilon}$ is spatial.
\end{lem}

\begin{rem}[Integrated maps]
\label{rem.intmap}
Let $B$ be a $C^*$-algebra and $v : \Pi \to U(B)$ a unital Borel map.
We consider a morphism
$\iota^u : \mA \to \mM(\mB)$
with $\mM(\mB) := (M(B),\ad,\ad_v,dv)$ and $\iota : A \to B$ non-degenerate.
Then
$\ad u_g \circ \eta = \eta \circ \alpha_g$
for all $g \in G$, and by universality of $A \rtimes^\alpha G$, \cite[Prop.2.39]{Wil}, there is a *-morphism
$\iota^{\rtimes u} : A \rtimes^\alpha G \to M(B)$.
Defining 
\[
\iota^{\rtimes u,v}(f) \, := \, \int_\Pi \iota^{\rtimes u}(f(l)) \, v_l \, dl 
\ \ \ , \ \ \ 
f \in C_c(\Pi,A \rtimes^\alpha G)
\ ,
\]
we obtain a linear map
$\iota^{\rtimes u,v} : C_c(\Pi,A \rtimes^\alpha G) \to M(B)$.
We call $\iota^{\rtimes u,v}$ the \emph{integrated map of $\iota^u$}.
\end{rem}

\paragraph{Crossed products by twisted actions.} 
Maintaining the assumptions of the previous paragraph, we proceed with the construction of the crossed product of a twisted $C^*$-dynamical system.

A \emph{crossed product} of $\mA$ is given by a $C^*$-algebra $B$ with a unital Borel map 
$\tbv : \Pi \to U(B)$,
fulfilling the following properties:
\begin{itemize}
\item[(\textbf{cp1})] There is a morphism
      $\tbi^\tbu : \mA \to \mM(\mB)$,
      where $\mM(\mB) := (M(B),\ad,\ad_\tbv,d\tbv)$ and $\tbi : A \to B$ is non-degenerate.
      The associated integrated map $\tbi^{\rtimes \tbu,\tbv}$ is a *-morphism when $(\bs\nu,\gamma) \in \check{Z}^2(\Pi,{}^2G)$.
\item[(\textbf{cp2})] $\tbi^{\rtimes \tbu,\tbv}$ takes values in $B$ (rather than, more generally, in $M(B)$),
     and the *-algebra generated by $\tbi^{\rtimes \tbu,\tbv}(C_c(\Pi,A \rtimes^\alpha G))$ is dense in $B$.
\item[(\textbf{cp3})] For any covariant representation $\pi^{\rho,\upsilon}$ of $\mA$ over a Hilbert space $H$, there is a representation 
      $\eta : B \to B(H)$
      such that, for all $l,m \in \Pi$ and $g \in G$,
      \begin{equation}
      \label{eq.cp2}
      \eta \circ \tbi \, = \, \pi
      \ \ \ , \ \ \ 
      \eta(\tbv_l) \, = \, \upsilon_l
      \ \ \ , \ \ \ 
      \eta(\tbu_g) \, = \, \rho(g) \ .
      \end{equation}
\end{itemize}

Note that (\textbf{cp1}) says that $\tbu : G \to U(B)$ is Borel, and the relations
\begin{equation}
\label{eq.cp}
\ad \tbv_l \circ \tbi \, = \, \tbi \circ \nu_l
\ \ \ , \ \ \ 
\ad \tbu_g \circ \tbi \, = \, \tbi \circ \alpha_g
\ ,
\end{equation}
hold for all $l \in \Pi$ and $g \in G$. 
Instead, a \emph{spatial} crossed product of $\mA$ is a $C^*$-algebra $B'$ fulfilling (\textbf{cp1}) - (\textbf{cp3})
with the conditions that $\tbi^\tbu$ must be spatial, 
\begin{equation}
\label{eq.cp2s}
d\tbv(l,m) \, := \, \tbv_l \tbv_m \tbv_{lm}^* \, = \, \tbu(\gamma(l,m))  \ \ , \ \ l,m \in \Pi \ ,
\end{equation}
and that (\textbf{cp3}) holds for any \emph{spatial} covariant representation $\pi^{\rho,\upsilon}$.
In this case, we require that $\tbi^{\rtimes \tbu,\tbv}$ is a *-morphism even if $(\bs\nu,\gamma) \notin \check{Z}^2(\Pi,{}^2G)$.

\begin{thm}
\label{thm.cp}
Let $\mA = (A,\alpha,\nu,\gamma)$ with structure 2--group ${}^2G = (G \tto N)$.
Then there are, unique up to isomorphism, a crossed product and a spatial crossed product denoted by, respectively,
\[
A \rtimes^{\alpha,\nu} (G,\Pi)
\ \ \ , \ \ \ 
A \rtimes^{\alpha,\nu,\gamma} (G,\Pi)
\ .
\]
\end{thm}

\begin{proof}
The proof is essentially the one of \cite[Prop.2.7]{PR89}, thus here we give a sketch of the argument.
We consider a set
$\mS = \{ \pi^{\rho_s,\upsilon_s}_s : \mA \to \mB(\mH_s) \}_{s \in I}$
of cyclic covariant representations of $\mA$ such that every
cyclic covariant representation of $\mA$ is equivalent to an element of $\mS$.
We have, respectively, a *-morphism, a (continuous) group morphism and a unital Borel map
\begin{equation}
\label{eq.t.cp1}
\tbi : A \to B(\oplus_s H_s)
\ \ \ , \ \ \ 
\tbu : G \to U(\oplus_s H_s)
\ \ \ , \ \ \ 
\tbv : \Pi \to U(\oplus_s H_s)
\ ,
\end{equation}
defined by 
$\tbi(t) := \{ \oplus_s \pi_s(t) \}$, $\tbu_g := \{ \oplus_s \rho_s(g) \}$ and $\tbv_l := \{ \oplus_s \upsilon_s(l) \}$.
By construction, (\ref{eq.cp}) holds.
By well-known properties of group representations, $\tbu$ and $\tbv$ extend to maps
\[
\tbu(w) := \int_G w(g) \tbu_g dg
\ \ \ , \ \ \ 
\tbv(w') := \int_\Pi w'(l) \tbv_l dl
\ ,
\]
for all $w \in C_c(G)$, $w' \in C_c(\Pi)$, 
with $\| \tbu(w) \| \leq \| w \|_1$ and $\| \tbv(w') \| \leq \| w' \|_1$.
Note that $\tbu$ is a *-morphism on $C_c(G)$ with respect to convolution and *, 
whilst $\tbv$ in general is only linear on $C_c(\Pi)$ due to the considerations of the previous paragraph.
Then we define $B$ as the $C^*$-subalgebra of $B(\oplus_s H_s)$ generated by operators of the type
\[
\tbv(w') \, \tbi(t) \, \tbu(w)
\ \ \ , \ \ \ 
t \in A \, , \, w \in C_c(G)  \, , \, w' \in C_c(\Pi)
\ .
\]
By (\ref{eq.cp}), elements of $\tbv(\Pi)$, $\tbi(A)$, $\tbu(G)$, and their products, 
are elements of $B$ or multipliers of $B$.
Thus, if $h \in C_c(G,A)$, then
$\tbi^{\rtimes \tbu}(h) :=$ $\int_G \tbi(h(g)) \tbu_g \, dg$
belongs to $M(B)$.
By density, we extend $\tbi^{\rtimes \tbu}$ to a *-morphism
\begin{equation}
\label{eq.t.cp2}
\tbi^{\rtimes \tbu} : A \rtimes^\alpha G \to M(B) \ .
\end{equation}
If $f \in C_c(\Pi,A \rtimes^\alpha G)$, then we define
\begin{equation}
\label{eq.t.cp3}
\tbi^{\rtimes \tbu,\tbv}(f) \, := \, \int_\Pi \tbi^{\rtimes \tbu}(f(l)) \, \tbv_l \, dl 
\ \ \ , \ \ \ 
f \in C_c(\Pi,A \rtimes^\alpha G)
\ .
\end{equation}
Since $f$ is limit of terms of the type 
$\sum_k w'_k \otimes t_k \otimes w_k $, with $w'_k \in C_c(\Pi)$, $t_k \in A$, $w_k \in C_c(G)$,
with $\tbi^{\rtimes \tbu,\tbv}(w'_k \otimes t_k \otimes w_k) =$ $\tbv(w'_k) \tbi(t_k) \tbu(w_k) \in$ $B$,
and since $\| f \|_1 \leq \| \tbi^{\rtimes \tbu,\tbv}(f) \|$,
we have $\tbi^{\rtimes \tbu,\tbv}(f) \in B$.
Again, in general $\tbi^{\rtimes \tbu,\tbv}$ does not define a *-morphism.
The argument of the proof of \cite[Prop.2.7]{PR89} now applies with the obvious modifications, 
and $B$, with $\tbi , \tbu , \tbv$, is the desired crossed product $A \rtimes^{\alpha,\nu} (G,\Pi)$.

The argument for the construction of the spatial crossed product is analogous, by considering a set
$\mS' = \{ \pi^{\rho_s,\upsilon_s}_s : \mA \to \mB(\mH_s) \}_{s \in I}$
of \emph{spatial} cyclic covariant representations.
Note that in this case by Lemma \ref{lem.C3.4} the map (\ref{eq.t.cp3}) is always a *-morphism.
\end{proof}

Some remarks follow.
\begin{itemize}
\item With the notation of the above proof we have $\mS' \subseteq \mS$, thus the $C^*$-norm on $C_c(\Pi,A \rtimes^\alpha G)$
      induced by $\mS'$ is $\leq$ the $C^*$-norm induced by $\mS$. We conclude that there is a *-morphism
      \[
      A \rtimes^{\alpha,\nu} (G,\Pi)
      \, \to \,  
      A \rtimes^{\alpha,\nu,\gamma} (G,\Pi)
      \ .
      \]
\item If $\mA = (A,\ad,\nu,\gamma)$ is inner then the argument of the previous proof applies with $A$ in place of
      $A \rtimes^\alpha G$. Of course, we do not need to define the integrated representation (\ref{eq.t.cp2}).
      The crossed product and the spatial crossed product are written 
      $A \rtimes^\nu \Pi$ and $A \rtimes^{\nu,\gamma} \Pi$
      respectively. Note that at least $A \rtimes^\nu \Pi$ is not trivial, 
      because we have the regular representation Example \ref{ex.regrep} that we compose with a representation of $B(X)$.
      At the present time, we do not have results ensuring non-triviality of the spatial crossed product,
      anyway when $\mA$ is Busby-Smith $A \rtimes^{\nu,\gamma} \Pi$ agrees with the crossed product \cite[Def.2.4]{PR89}
      and is non-trivial.
      %
      %
\item For a generic system $\mA = (A,\alpha,\nu,\gamma)$ with structure 2--group ${}^2G$, Lemma \ref{lem.C3.3} ensures 
      that there is a morphism $\eta^\phi : \mA \to \mB$ with $\mB$ inner and $\eta$ faithful,
      thus by the previous point $A \rtimes^{\alpha,\nu} (G,\Pi)$ is not trivial. 
      What is not ensured is the fact that $\phi$ is faithful, thus in general we do not know whether 
      the maps $\tbu$ and, as consequence, $\tbi^{\rtimes \tbu}$ are always faithful.
      Of course, this turns out to be the case when $\alpha$ is faithful (Lemma \ref{lem.C3.2} and Lemma \ref{lem.C3.3}).
      
      For the same reason, $A \rtimes^{\alpha,\nu,\gamma} (G,\Pi)$ is not trivial 
      when $\mA$ admits a non-trivial morphism into a Busby-Smith system.
      This is the case
      when $\alpha$ is faithful (Lemma \ref{lem.C3.2} and Lemma \ref{lem.C3.3}),
      when $(\bs\nu,\gamma) \in \check{Z}^2(\Pi,{}^2G)$ (Lemma \ref{lem.C3.3}),
      or when $G$ is compact (Corollary \ref{cor.C3.3}).
\end{itemize}

\begin{rem}
\label{rem.BMZ3}
Let $\mA = (A,\alpha,\nu,\gamma)$ with structure 2--group ${}^2G = (G \tto N)$
and $G$ locally compact.
By Remark \ref{rem.BMZ2}, we have the Green $C^*$-dynamical system
$\beta : {}^2G_{\bs\nu} \to {}^2\Aut (A \rtimes^\alpha G)$.
Let now 
$\eta : A \to M(A \rtimes^\alpha G)$, $\phi : G \to U(A \rtimes^\alpha G)$
denote the canonical morphisms.
We consider covariant (non-degenerate) representations of $\beta$ in the sense of \cite{BMZ12}, 
given by pairs
$\pi : A \rtimes^\alpha G \to B(H)$,
$V : G \vee_{\bs\nu} \Pi \to U(H)$,
such that
\[
\pi \circ \beta(\bs\nu_l) \, = \, \ad V(\bs\nu_l) \circ \pi
\ \ \ , \ \ \ 
V(\unl{g}) \, = \, \pi(\phi_g) \, ,
\]
$g \in G$, $l \in \Pi$.
The above relations imply that 
$\{ \pi \circ \eta \}^{ \pi \circ \phi , V \circ \bs\nu }$
is a covariant representation of $\mA$. Since
$V(\bs\nu_l \bs\nu_m \bs\nu_{lm}^{-1}) =$
$V(\unl{\gamma(l,m)}) =$
$\pi(\phi(\gamma(l,m)))$,
the above representation is spatial.
Thus taking the seminorm on $A \rtimes^{\alpha,\nu,\gamma} (G,\Pi)$ defined by these representations yields a surjective *-morphism
\[
A \rtimes^{\alpha,\nu,\gamma} (G,\Pi) \to (A \rtimes^\alpha G) \rtimes_\beta {}^2G_{\bs\nu} \, ,
\]
where the crossed product $\rtimes_\beta$ is in the sense of \cite{BMZ12}.
A more general version of this crossed product in the setting of $C^*$-correspondences has been discussed in \cite{BM17}.
\end{rem}

\subsection{Twisted $C^*$-dynamical systems as holonomy maps}
\label{sec.C4}

Needless to say, our notion of twisted $C^*$-dynamical system is designed to give the twisted holonomy of a bundle $C^*$-gerbe,
with the idea that $\Pi$ is the fundamental group.

Throughout this section we assume that $\Delta$ is a poset such that $\pi_1(\Delta)$ is countable.
This covers the cases in which we are interested, where $\Delta$ is a good base of a manifold.
Thus we endow $\pi_1(\Delta)$ with the discrete topology, making it a locally compact group.

Given a topological group $G$, we consider the category $\Cgrb^u_{\Delta,G}$
having objects bundle $C^*$-gerbes over $\Delta$ with fibres $G$-dynamical systems;
the superscript $u$ indicates that for the moment we restrict our attention to untwisted morphisms
$\eta : \check{\mA} \to \check{\mA}'$,
such that the terms $\beta_b$ in (\ref{eq.B2.1a}) are the identity;
this allows us to state the next result in all its conceptual simplicity.
\begin{thm}
\label{thm.C4.1}
Let $\Delta$ be a poset with countable fundamental group $\Pi := \pi_1(\Delta)$ and $G$ a locally compact group.
Then assigning the twisted holonomy to bundle $C^*$-gerbes over $\Delta$ defines a functor
\begin{equation}
\label{eq.thm.C4.1}
\Cgrb^u_{\Delta,G} \, \to \, \tdyn_{\Pi,G} \ .
\end{equation}
\end{thm}

\begin{proof}
Given a bundle $C^*$-gerbe $\check{\mA} = (A,\jmath)_\Delta$, we pick a standard fibre
$\alpha_a : G \to \Aut A_*$, $A_* = A_a$, and a section
$\varsigma_{ \check{\mA} , a } : \alpha_a(A_*) \to G$, $\alpha \circ \varsigma_{ \check{\mA} , a } = id$.
In this way, (\ref{eq.B3.3}) and (\ref{eq.B3.5}) yield the cocycle $( \bar{\jmath}^{\, a} , \gamma^a )$, where 
$\bar{\jmath}^{\, a}_l :=$ $\jmath_{s(l)} \in$ $\Aut A_*$ and $\gamma^a(l,m) \in \alpha_a(G)$ 
for all $l,m \in \Pi$; here,
$s(l) : a \to a$
is the choice of a path with homotopy class $l$.
We set 
$\gamma_*(l,m) := \varsigma_{ \check{\mA} , a } \circ \gamma^a(l,m) \in G$,
and note that both $\bar{\jmath}^{\, a}$ and $\gamma_*$ are obviously Borel.
Moreover, by (\ref{eq.B3.5}),
\[
\bar{\jmath}^{\, a}_1 \, = \, id
\ \ \ , \ \ \ 
\bar{\jmath}^{\, a}_l \circ \bar{\jmath}^{\, a}_m \, = \, 
\gamma^a(l,m) \circ \bar{\jmath}^{\, a}_{lm} \, = \, 
\alpha_{\gamma_*(l,m)} \circ \bar{\jmath}^{\, a}_{lm}
\ \ \ , \ \ \ 
l,m \in \Pi
\ .
\]
Thus $\mA_* = ( A_* , \alpha_a , \bar{\jmath}^{\, a} , \gamma_* )$ is a twisted $C^*$-dynamical system.

Let now $\eta : \check{\mA} \to \check{\mA}' = (A',\jmath^{ \,'})_\Delta$ be a gerbe morphism.
We set $\eta_* := \eta_a$.
Since by definition
$\bar{\jmath}^{\, a}_l = \jmath_{s(l)}$
with $s(l) : a \to a$, applying repeatedly the first of (\ref{eq.B2.1a}) with $\beta = id$, we obtain
$\eta_* \circ \jmath_{s(l)} = \jmath_{s(l)} \circ \eta_*$,
implying
$\eta_* \circ \bar{\jmath}^{\, a}_l = \bar{\jmath}^{\, a}_l \circ \eta_*$.
On the other hand, applying the second of (\ref{eq.B2.1a}) we have
$\alpha'_{a,g} \circ \eta_*  =  \eta_* \circ \alpha_{a,g}$,
$g \in G$.
Thus $\eta_* : \mA_* \to \mA'_*$ is a morphism in $\tdyn_{\Pi,G}$. 
Given a morphism $\eta' : \check{\mA}' \to \check{\mA}''$, we have 
$(\eta' \circ \eta)_* =$ $\eta'_a \circ \eta_a =$ $\eta'_* \circ \eta_*$,
and (\ref{eq.thm.C4.1}) is a functor.
\end{proof}

\begin{rem}
In the case of inner twisted $C^*$-dynamical systems we have $G = U(A)$ for a fixed $C^*$-algebra $A$,
thus (\ref{eq.thm.C4.1}) is relative to bundle $C^*$-gerbes with standard fibre $A$.
A class of such $C^*$-gerbes is defined as follows. Let
$q=(u,g) \in \check{U}^2(\Delta,{}^2\Aut A)$
be a twisted connection; then we have the bundle $C^*$-gerbe $\check{\mA}_q$ associated to $q$,
with constant fibre $A$ and gerbe structure $u_b : A \to A$, $b \in \Sigma_1(\Delta)$.
We then say that $\check{\mA}_q$ is a \emph{Busby-Smith} $C^*$-gerbe.
Busby-Smith $C^*$-gerbes give rise to "flat" Morita bundle gerbes, Ex.\ref{ex.pennig}, \cite{Pen11}.
%
%
\end{rem}


\paragraph{Generic gerbe morphisms.}
We now consider the category $\Cgrb_{\Delta,G}$ with the same objects as $\Cgrb^u_{\Delta,G}$
and arrows generic gerbe morphisms $\eta_\beta : \check{\mA} \to \check{\mA}'$, Def.\ref{def.gmor}.
Here, $\check{\mA}' = (A',\jmath^{ \,'})_\Delta$ is a bundle $q'$-$C^*$-gerbe with $q'=(u',g') \in \check{Z}^2(\Delta,{}^2G)$.
By hypothesis, given $b_1 \in \Sigma_1(\Delta)$ we have $\beta_{b_1} = \alpha'_{\partial_0b_1}(g_1) \in \Aut A'_{\partial_0b_1}$ for some $g_1 \in G$;
thus, by the second of (\ref{eq.B2.1}), for any $b_2 \in \Sigma_1(\Delta)$, $\partial_1b_2 = \partial_0b_1$, we have
\begin{equation}
\label{eq.beta}
\jmath^{ \,'}_{b_2} \circ \beta_{b_1} \, = \, \beta_{b_2,b_1} \circ \jmath^{ \,'}_{b_2}
\ \ \ , \ \ \ 
\beta_{b_2,b_1} := \alpha'_{\partial_0b_2}(\wa{u'}_{b_2}(g_1)) \ .
\end{equation}
Let $p : a \to a$, $p = b_n * \ldots * b_1$; for any $k=1,\ldots,n$ we set $\beta_{b_k} = \alpha'_{\partial_0b_k}(g_k)$. 
Using (\ref{eq.beta}) and the first of (\ref{eq.B2.1a}) we find
\[
\jmath^{ \,'}_{b_2} \circ \jmath^{ \,'}_{b_1} \circ \eta_a \, = \, 
\beta_{b_2,b_1} \circ \beta_{b_2} \circ \eta_{\partial_0b_2} \circ \jmath_{b_2} \circ \jmath_{b_1}
\ ,
\]
and, by iteration,
\begin{equation}
\label{eq.beta'}
\jmath^{ \,'}_p \circ \eta_a \, = \, \beta_p \circ \eta_a \circ \jmath_p \ ,
\end{equation}
where $\beta_p \in \alpha'_a(G)$ is given by
\begin{equation}
\label{eq.defbeta}
\beta_p 
\, := \,
\alpha'_a( \wa{u'}_{ b_n * \ldots * b_2 }(g_1) \,
          \ldots \,
          \wa{u'}_{b_n}(g_{n-1}) \,
          g_n )
\ .
\end{equation}
Given a section $\varsigma_{ \check{\mA}' , a } : \alpha'_a(G) \to G$ we set 
$\kappa_*(p) := \varsigma_{ \check{\mA}' , a }(\beta_p)$, so that $\alpha'_{\kappa_*(p)} = \beta_p$;
moreover, given the section $s : \pi_1(\Delta) \to \Pi_1^a(\Delta)$ used in Theorem \ref{thm.C4.1},
we set
$\kappa(l) := \kappa_*(s(l))$, $l \in \Pi$.
Recalling the proof of Theorem \ref{thm.C4.1}, (\ref{eq.beta'}) reads
\begin{equation}
\label{eq.beta''}
\bar{\jmath}^{\,' \, a}_l \circ \eta_* \, = \, \alpha'_{\kappa(l)} \circ \eta_* \circ \bar{\jmath}^{\, a}_l
\ \ \ , \ \ \ 
\forall l \in \Pi
\ ,
\end{equation}
which, besides the $G$-equivariance of $\eta_*$, implies that (\ref{eq.C3.mk}) is fulfilled by
\[
\{ \eta_\beta \}_* \ := \ \eta_{*,\kappa} \ .
\]
\begin{rem}
\label{rem.ek}
Once that a choice of the sections $\varsigma_{ \check{\mA}' , a }$ and $s$ is made,
the morphism $\eta_{*,\kappa}$ is uniquely determined by $\eta_a$ and the mapping $\beta$
defined by (\ref{eq.defbeta}).
\end{rem}

Thus generic gerbe morphisms induce morphisms of the type (\ref{eq.C3.mk}) defining the category $\tdyn^c_{\Pi,G}$.
Yet a problem arises with compositions of morphisms
\[
\eta_\beta : \check{\mA} \to \check{\mA}'
\ \ \ , \ \ \ 
\eta'_{\beta'} : \check{\mA}' \to \check{\mA}''
\, ,
\]
where $\beta'_b = \alpha''_\bo(g'_b)$ and $\beta_b = \alpha'_\bo(g_b)$.
To illustrate this point, we consider the composition $\{ \eta' \circ \eta \}_{\beta^2} :=$ $\eta'_{\beta'} \circ \eta_\beta$, 
so that, by the first of (\ref{eq.B2.1a}), the map $\beta^2$ is defined by
$\beta^2_b := \alpha''_\bo(g'_b g_b)$.
Then, with the obvious notation, iterating (\ref{eq.beta}) and (\ref{eq.beta'}) we find the analogue of (\ref{eq.beta'}) for
$\eta'_{*,\kappa'} \circ \eta_{*,\kappa}$ and $\{ \eta'_{\beta'} \circ \eta_\beta \}_*$
respectively:
\begin{equation}
\label{eq.beta1}
\jmath^{ \,''}_p \circ \eta'_a \circ \eta_a
\, = \, 
\beta^1_p \circ \eta'_a \circ \eta_a \circ \jmath_p \, ,
\ \ \ , \ \ \ 
\jmath^{ \,''}_p \circ \{ \eta' \circ \eta \}_a
\, = \, 
\beta^2_p \circ \{ \eta' \circ \eta \}_a \circ \jmath_p \, ,
\end{equation}
where
\begin{equation}
\label{eq.beta'''}
\beta^1_p \ = \ 
\alpha''_a( \wa{u''}_{ b_n * \ldots * b_2 }(g'_1) \,
          \ldots \,
          \wa{u''}_{b_n}(g'_{n-1}) \,
          g'_n \, \,
          \wa{u'}_{ b_n * \ldots * b_2 }(g_1) \,
          \ldots \,
          \wa{u'}_{b_n}(g_{n-1}) \,
          g_n           
          )
\end{equation}
and
\begin{equation}
\label{eq.beta''''}
\beta^2_p \ = \ 
\alpha''_a( \wa{u''}_{ b_n * \ldots * b_2 }(g'_1g_1) \,
          \ldots \,
          \wa{u''}_{b_n}(g'_{n-1} g_{n-1}) \,
          g'_n g_n \, 
          )
\, .
\end{equation}
Note that the two left hand sides of (\ref{eq.beta1}) agree;
composing to the right with $\jmath_p^{-1}$, we have
\begin{equation}
\label{eq.beta1a}
\beta^1_p \circ \eta'_a \circ \eta_a
\ = \ 
\beta^2_p \circ \{ \eta' \circ \eta \}_a
\ = \
\beta^2_p \circ \eta'_a \circ \eta_a\, .
\end{equation}
The relations (\ref{eq.beta'''}) and (\ref{eq.beta''''}) say that in general $\beta^1_p \neq \beta^2_p$.
On the other hand, by Remark \ref{rem.ek} we have functoriality if, and only if, $\beta^1_p = \beta^2_p$.
Thus in general we cannot expect that twisted holonomy preserves composition of morphisms of the type $\eta_\beta$.
Yet we have two favorable cases.
\begin{cor}
\label{cor.C4.2}
Let $G$ be abelian and $\Cgrb_{\Delta,G}^{ab}$ denote the full subcategory of $\Cgrb_{\Delta,G}$
with objects bundle $q$-$C^*$-gerbes where $q \in \check{Z}^2(\Delta,{}^2G)$ and 
${}^2G = (G \tto N)$ is such that $N$ acts trivially on $G$.
Then twisted holonomy defines a functor
\begin{equation}
\label{eq.cor.C4.2}
\Cgrb_{\Delta,G}^{ab} \, \to \, \tdyn^c_{\Pi,G} \ .
\end{equation}
\end{cor}

\begin{proof}
Comparing (\ref{eq.beta'''}) and (\ref{eq.beta''''}) we find $\beta^1_p = \beta^2_p$ for all $p$.
\end{proof}

\begin{cor}
\label{cor.C4.3}
Let $\Cgrb_{\Delta,G}^{su}$ denote the subcategory of $\Cgrb_{\Delta,G}$ having arrows morphisms 
of the type $\eta_\beta$ with $\eta$ surjective. Then twisted holonomy defines a functor
\begin{equation}
\label{eq.cor.C4.2}
\Cgrb_{\Delta,G}^{su} \, \to \, \tdyn^c_{\Pi,G} \ .
\end{equation}
\end{cor}

\begin{proof}
If $\eta'_{\beta'}$ and $\eta_\beta$ are composable arrows in $\Cgrb_{\Delta,G}^{su}$ then
any $t \in A''_a$ is of the type $t = \eta'_a(\eta_a(t_0))$ for some $t_0 \in A_a$,
and (\ref{eq.beta1a}) implies $\beta^1_p(t) =$ $\beta^1_p(\eta'_a(\eta_a(t_0))) =$ $\beta^2_p(t)$ 
for all loops $p$.
Thus $\beta^1_p = \beta^2_p$.
\end{proof}

\paragraph{Uniqueness of (\ref{eq.thm.C4.1}).}
In the proof of Theorem \ref{thm.C4.1} several choices are made to construct the desired functor (\ref{eq.thm.C4.1}).
We now investigate the dependence of the twisted holonomy on these choices.
In the sequel, we use the notation $\check{\mA} = (A,\jmath)_\Delta$, $q=(u,g)$, for a bundle $q$-$C^*$-gerbe
and, analogously, $\check{\mA}' = (A',\jmath^{ \,'})_\Delta$, $q'=(u',g')$, and so on.

\

A first choice is given by the sections $\varsigma_{ \check{\mA} , a }$. 
Picking different sections $\varsigma'_{ \check{\mA} , a }$ does not affect the automorphisms 
$\bar{\jmath}^{\, a}_l$, $l \in \Pi$,
but yields new maps $\gamma'_* := \varsigma'_{ \check{\mA} , a } \circ \gamma_a$.
This defines the new functor
$*' : \Cgrb^u_{\Delta,G} \, \to \, \tdyn_{\Pi,G}$, $\mA_{*'} :=$ $( A_* , \alpha_a , \bar{\jmath}^{\, a} , \gamma'_* )$,
$\eta_{*'} := \eta_*$.
Now, by definition we have $\alpha_{\gamma_*} = \alpha_{\gamma'_*}$, and
Remark \ref{rem.gamma} implies that the identity yield the isomorphism
$\iota_{\check{\mA}} : \mA_* \to \mA_{*'}$
such that
$\iota_{\check{\mA}'} \circ \eta_* = \eta_* \circ \iota_{\check{\mA}}$
for all morphisms $\eta : \check{\mA} \to \check{\mA}'$.
Thus $\iota$ is a natural transformation, making $*$ and $*'$ isomorphic functors.

\

A second choice is the section $s$.
As we saw in the proof of Prop.\ref{prop.B1.1}, picking a different section $r : \Pi \to \Pi_1^a(\Delta)$ we get a new twisted holonomy functor
\[
\mA_{*''} \, := \, ( A , \alpha_a , \bar{\jmath}^{\, a,r} , \gamma_* )
\ \ \ , \ \ \ 
\eta_{*''} = \eta_*
\ .
\]
Now,
\[
\kappa_*(l) \, := \, 
\bar{\jmath}^{\, a,r}_l \circ (\bar{\jmath}^{\, a}_l)^{-1}  \ = \
\jmath_{r(l)} \circ \jmath_{s(l)}^{-1} \,
\in \alpha_a(G) \, ,
\]
for all $l \in \Pi$ (in fact, $r(l)$ and $s(l)$ have the same homotopy class $l$),
and defining $\kappa_l := \varsigma_{ \check{\mA}' , a }(\kappa_*(l))$, $l \in \Pi$, we get a map $\kappa : \Pi \to G$ with 
$\alpha_{\kappa(l)} = \kappa_*(l)$.
By construction,
\[
\bar{\jmath}^{\, a,r}_l \, = \, \alpha_{\kappa(l)} \circ \bar{\jmath}^{\, a}_l \ \ \ , \ \ \  \forall l \in \Pi \ ,
\]
thus the identity $id \in \Aut A$ defines the isomorphism $\iota'_{\check{\mA}} : \mA_* \to \mA_{*''}$,
$\iota'_{\check{\mA}} := id_\kappa$, of the type (\ref{eq.C3.mk}).
For all morphisms $\eta : \check{\mA} \to \check{\mA}'$ we have
$\kappa'_*(l) := \bar{\jmath}^{\, ' \, a,r}_l \circ (\bar{\jmath}^{\, ' \, a}_l)^{-1}$
and
\[
\eta_* \circ \kappa_*(l) \ = \
\eta_a \circ \jmath_{r(l)} \circ \jmath_{s(l)}^{-1} \ = \
\jmath^{\, '}_{r(l)} \circ (\jmath^{\, '}_{s(l)})^{-1} \circ \eta_a \ = \
\kappa'_*(l) \circ \eta_* 
\, .
\]
Thus
$\iota'_{\check{\mA}'} \circ \eta_* = \eta_* \circ \iota'_{\check{\mA}}$,
and $\iota'$ is a natural transformation making $*$ and $*''$ isomorphic functors.

\

Finally, the last choice is the standard fibre. 
To correctly handle this case, it is convenient to regard (\ref{eq.thm.C4.1}) as a functor
$\Cgrb^u_{\Delta,G} \to \tdyn_\Pi$,
so that now the target category $\tdyn_\Pi$ has arrows more general morphisms of the type (\ref{eq.C3.1}).
Picking $a' \in \Delta$ we get the functor
\begin{equation}
\label{eq.thm.C4.1'}
* ''' : \Cgrb^u_{\Delta,G} \, \to \, \tdyn_\Pi 
\ \ , \ \
\mA_{*'''} = ( A'_* , \alpha_{a'} , \bar{\jmath}^{\, a',r} , \gamma_* )
\ \ , \ \ 
\eta_{*'''} := \eta_{a'}
\ ,
\end{equation}
where $A'_* := A_{a'}$ and $\bar{\jmath}^{\, a',r}_l := \jmath_{r(l)}$, $l \in \Pi$.
Here, for convenience we use the section $r(l) :=$ $p_{a'a} * s(l) * p_{aa'}$, $l \in \Pi$,
where $p_{aa'}$ is a path from $a'$ to $a$ and $p_{a'a} := \ovl{p_{aa'}}$.
In this way, by (\ref{eq.B2.1}) we get
\[
\bar{\jmath}^{\, a',r}_l \, = \, 
\jmath_{ p_{a'a} * s(l) * p_{aa'} } \, = \,
\iota''_{\check{\mA}} \circ \bar{\jmath}^{\, a}_l \circ (\iota''_{\check{\mA}})^{-1}
\ \ \ , \ \ \ 
\iota''_{\check{\mA}} \circ \alpha_{a,g} \, = \, \alpha_{a',\phi(g)} \circ \iota''_{\check{\mA}}
\ ,
\]
having defined $\iota''_{\check{\mA}} := \jmath_{p_{a'a}}$ and $\phi(g) := \wa{u}_{ p_{a'a} }(g)$, $g \in G$.
Thus a change of the fibre affects $\mA_*$ by the isomorphism $(\iota''_{\check{\mA}})^\phi$ of the type (\ref{eq.C3.1}).
Given the usual morphism $\eta : \check{\mA} \to \check{\mA}'$, the equalities
\[
\iota''_{\check{\mA}'} \circ \eta_* \ = \ 
\jmath^{\, '}_{p_{a'a}} \circ \eta_a \ = \
\eta_{a'} \circ \jmath_{p_{a'a}} \ = \
\eta_{*'''} \circ \iota''_{\check{\mA}} 
\]
say that $\iota''$ is a natural transformation making $*$ and $*'''$ isomorphic functors.

\paragraph{The range of (\ref{eq.thm.C4.1}).}
In contrast with the case of $C^*$-bundles, it is not evident whether or not (\ref{eq.thm.C4.1}) is an equivalence.
In particular, it remains an open question whether any twisted $C^*$-dynamical system is isomorphic to a twisted holonomy.
The main point is that a twisted holonomy $\bar{\jmath}^{\, a}$ preserves the inverse,
\[
(\bar{\jmath}^{\, a}_l)^{-1} \, = \, 
\jmath_{s(l)}^{-1} \, \stackrel{*}{=} \,
\jmath_{ \ovl{s(l)} } \, = \, 
\bar{\jmath}^{\, a}_{l^{-1}} 
\ \ \ , \ \ \ 
\forall l \in \Pi
\ ,
\]
where in the equality marked by * we used (\ref{eq.tc2}).
Instead, given a generic twisted $C^*$-dynamical system $\mA = (A,\alpha,\nu,\gamma)$, by (\ref{eq.tds}) we have
\[
\nu_{l^{-1}} \ = \ \nu_l^{-1} \circ \alpha_{ \gamma( l , l^{-1} ) } \ \ \ , \ \ \  l \in \Pi \ .
\]
We note that there may be the possibility that terms of the type $\alpha_{ \gamma( l , l^{-1} ) }$ are dropped by an isomorphism
of the type (\ref{eq.C3.mk}), anyway we leave this point as an open question.


\section{Duality}
\label{sec.D}

In the present section we prove our duality for group gerbes.
Given a presheaf of DR-categories (\emph{DR-presheaf}),
we show that its category of sections embeds, essentially in a unique way,
in the one of Hilbert gerbes associated to a non-abelian cocycle Ex.\ref{ex.Hilb}.

\subsection{Presheaves of categories}
\label{sec.D1}

Let $\Delta$ be a poset.
A \emph{presheaf of categories over $\Delta$} is given by a pair $\mS = (S,r)^\Delta$, where
$S = \{ S_o \}_{o \in \Delta}$
is a collection of categories and $r$ is a family of \emph{restriction functors}
\[
r_{o\omega } : S_\omega  \to S_o
\ \ \ : \ \ \ 
r_{o\omega } \circ r_{\omega \xi } = r_{o\xi }
\ \ , \ \
\forall o \leq \omega  \leq \xi  \ .
\]
Presheaf morphisms
$\eta : \mS \to \mS' =$ $(S',r')^\Delta$
are defined by families of functors preserving the presheaf structures,
$\eta_o \circ r_{o\omega } = r'_{o\omega } \circ \eta_\omega $, $o \leq \omega $.

We will focus on the case where each $S_o$ is a DR-category and each $r_{o\omega }$ is a tensor *-functor.
In this case we say that $\mS$ is a \emph{DR-presheaf}.

We say that $\mS$ is a \emph{presheaf bundle} whenever each $r_{o\omega }$ is a (strictly invertible) isomorphism,
and in this case we write $r_{\omega o} := r_{o\omega}^{-1}$;
we also use the notations 
$r_b : S_\bl \to S_\bo$, $b \in \Sigma_1(\Delta)$, and 
$r_p : S_a \to S_o$, $p : a \to o$, 
analogous to those relative to bundles.
Some remarks follow.
First, any presheaf bundle $\mS = (S,r)^\Delta$ is isomorphic to a presheaf bundle $\mS' = (S',r')^\Delta$
where any $S'_o$ equals a fixed category $F$ 
(\emph{i.e.}, $\mS'$ is a \emph{choice of the standard fibre} of $\mS$, analogously to Remark \ref{rem.standardfibre}).
Secondly, reasoning as in (\ref{eq.A8}) and subsequent remarks we find the following characterization of 
the set $\pbun(\Delta,F)$ of isomorphism classes of presheaf bundles with standard fibre $F$,
\begin{equation}
\label{eq.D1.1}
\pbun(\Delta,F) \, \simeq \, H^1( \pi_1(\Delta) , \Aut F ) \ ,
\end{equation}
\cite[Eq.3.21]{Vas12}.
When $F$ carries an additional structure, by $\Aut F$ we mean the group of automorphisms that preserves it:
for example, when $\mS$ is a DR-presheaf $\Aut F$ denotes the group of tensor *-automorphisms of $F$.
Given a reference $a \in \Delta$, we denote the corresponding holonomy by
\[
\Hol_a\mS  \, := \, \{ r_p \in \Aut S_a \, : \, p : a \to a  \} \ .
\]
\begin{ex}
\label{ex.D1.1}
Let $d \in \bN$ and $G \subseteq \bSU(d)$ a compact group. 
Then the inclusion map defines a unitary representation $\pi : G \to \bU(d)$,
and we denote the category of tensor powers of $\pi$ by $\wa{\pi}$.
Since $\wa{\pi}$ is a tensor subcategory of $\wa{G}$, it is a DR-category.
We have $\Aut \wa{\pi} \simeq N/G$, where $N$ is the normalizer of $G$ in $\bU(d)$.
Thus for any poset $\Delta$ we find $\pbun(\Delta,\wa{\pi})  \simeq  H^1( \pi_1(\Delta) , N/G )$,
see \cite[Theorem 4.5]{Vas12}.
\end{ex}

\paragraph{Sections.}
Let $\mS = (S,r)^\Delta$ be a presheaf of categories.
The \emph{category of sections} $\wt{\mS}$ of $\mS$ is defined by
\[
\left\{
\begin{array}{ll}
\obj \wt{\mS} \, := \, \{ \varrho = \{ \varrho_o \in \obj S_o \}_o \, : \, r_{o\omega }(\varrho_\omega ) = \varrho_o \, , \, \forall o \in \Delta  \} \ ,
\\ \\
(\varrho,\varsigma) \, := \, \{ t = \{ t_o \in (\varrho_o,\varsigma_o) \}_o \, : \, r_{o\omega }(t_\omega ) = t_o \, , \, \forall o \in \Delta \} 
\ \ \ \ , \ 
\varrho , \varsigma \in \obj \wt{\mS} \ .
\end{array}
\right.
\]
\begin{rem}
\label{rem.pb}
A useful property is that for any presheaf $\mS$ there is a presheaf bundle 
${}_\beta\mS = ( {}_\beta S , {}_\beta r )^\Delta$
with a canonical isomorphism ${}_\beta\wt{\mS} \simeq \wt{\mS}$ \cite[Prop.3.4]{Vas12},
thus we may reduce to study presheaf bundles for what concerns the categories of sections.
\end{rem}

We remark that $\wt{\mS}$ carries additional structure inherited by the fibres of the presheaf.
Thus in particular $\wt{\mS}$ is a DR-category when $\mS$ is a DR-presheaf 
($\Delta$ being assumed connected, the identity section is simple \cite[\S 3]{Vas12}).
In this case Doplicher-Roberts duality says that $\wt{\mS}$ is isomorphic to the dual of a compact group $\wt{G}$
{\footnote{Indeed, it is not difficult to prove that $\wt{G}$ is the dual of ${}_\beta \ovl{S}$,
           where ${}_\beta \ovl{S}$ is the fixed point category under the holonomy action $\pi_1^a(\Delta) \to \Aut {}_\beta S_a$
           defined by (\ref{eq.D1.1}).}},
but this does not give us detailed informations on the structure of $\wt{\mS}$.
Our aim is to prove, instead, a duality that takes account of the nature of fields over $\Delta$ of the objects and arrows of $\wt{\mS}$.
This shall be done in the following section.

\subsection{The embedding theorem}
\label{sec.D2}

We start with the following remarks, that are consequences of \cite[Theorems 6.1 and 6.9]{DR89}.
Let $T$ be a DR-category. Then there is an embedding functor $H : T \to \Hilb$ assigning, in particular, Hilbert spaces $H(\rho)$, $\rho \in \obj T$.
If $H' : T \to \Hilb$ is an other embedding functor, then there is a tensor unitary natural transformation $v$ from $H$ to $H'$,
and we write $v \in U_\otimes(H,H')$. This means that $v$ is a family of unitary operators 
$v^\rho : H(\rho) \to H'(\rho)$, $\rho \in \obj T$,
fulfilling certain compatibility properties (details are given in Eq.\ref{eq.D2.1} below).
In particular, $U_\otimes(H,H')$ is a compact group when $H=H'$, under the topology induced by the product of the unitary groups
$U(H(\rho))$, $\rho \in \obj T$.
We write $U_\otimes(H) := U_\otimes(H,H)$.
By \cite[Theorem 6.1]{DR89}, the dual of $U_\otimes(H)$ is isomorphic to $T$.

Let now $\mS = (S,r)^\Delta$ be a DR-presheaf bundle. 
By the above remarks, for any $o \in \Delta$ there is an embedding functor
$H_o : S_o \to \Hilb$,
and for any $\omega \geq o$ there is a tensor unitary natural transformation
\begin{equation}
\label{eq.D2.1a}
v_{\omega o} 
\, = \, 
\{ v_{\omega o}^\rho \}_{\rho \in \obj S_o} \in U_\otimes( H_o \, , \, H_\omega  \circ r_{\omega o} )
\end{equation}
(for $o=\omega$ we take $v_{oo} := 1_o$, where $1_o$ is the family of identities $1_o^\rho \in U(H_o(\rho))$).
Let $a \in \Delta$. We consider the compact groups
\[
G_\bullet \, := \, U_\otimes(H_a) \ \ , \ \ 
G_o \, := \, U_\otimes(H_o) \ \ , \ \ 
\forall o \in \Delta \ .
\]
By the above remarks, for any $o \in \Delta$ there is an isomorphism $S_o \simeq \wa{G}_o$.
To be concise we write 
$H^\rho_o := H_o(\rho)$.
Thus we can write the properties defining $v_{\omega o}$ and $G_o$ as follows:
\begin{equation}
\label{eq.D2.1}
\left\{
\begin{array}{lll}
v_{\omega o}^\rho \in U(H^\rho_o \, , \, H^{\rho' }_\omega  ) \ , &
v_{\omega o}^{\rho\sigma} = v_{\omega o}^\rho \otimes v_{\omega o}^\sigma \ , &
v_{\omega o}^{\sigma'} \circ H_o(t) \, = \, H_\omega (r_{\omega o}(t)) \circ v_{\omega o}^\rho \ ,
\\ \\
y^\rho \in U(H^\rho_o) \ , &
y^{\rho\tau} \, = \, y^\rho \otimes y^\tau \ , &
y^\sigma \circ H_o(t) \, = \, H_o(t) \circ y^\rho \ ,
\end{array}
\right.
\end{equation}
for all
$\rho,\sigma,\tau \in \obj S_o$,
$\rho'  := r_{\omega o}(\rho) , \sigma' := r_{\omega o}(\sigma)  \in \obj S_\omega $, 
$y \in G_o$,
$t \in (\rho,\sigma)$.

\begin{rem}
A family of embeddings $H = \{ H_o : S_o \to \Hilb \}$ is said to be \emph{coherent} whenever
$H_\omega = H_o \circ r_{o\omega}$ for all $o \leq \omega$.
If such a family exists, then $\wt{\mS}$ can be embedded in the category $\Hilb(\Delta)$ of Hilbert bundles over $\Delta$ \cite[Theorem 5.6]{Vas12}.
Examples of DR-presheaves that cannot be embedded in $\Hilb(\Delta)$ are given by presheaves with standard fibre $\wa{\pi}$, Example \ref{ex.D1.1},
such that the corresponding holonomy $r_* : \pi_1(\Delta) \to N/G$ does not admit a lift $\nu : \pi_1(\Delta) \to N$, ${\mathrm{mod}}G \circ \nu = r_*$,
\cite[remarks after Theorem 5.7]{Vas12}.
\end{rem}

\paragraph{The dual 2--group.}
The groups $G_o$ and the natural transformations $v_{\omega o}$
are key ingredients to construct the dual non-abelian cocycle of $\mS$,
as we explain in the following.

Let 
$\tbf : S_a \to S_o$, $\tbf' : S_o \to S_e$ 
be isomorphisms, and
$v \in$ $U_\otimes( H_a , H_o \circ \tbf )$, $v' \in$ $U_\otimes( H_o , H_e \circ \tbf' )$
natural transformations. Then we define natural transformations
\[
v' \cdot v \in U_\otimes( H_a , H_e \circ \tbf' \circ \tbf )
\ \ \ \ , \ \ \ \
v^{-1} \in U_\otimes( H_o , H_a \circ \tbf^{-1} ) \ ,
\]
by setting
\begin{equation}
\label{eq.N0}
\left\{
\begin{array}{ll}
v' \cdot v  \, := \, \{ \, v'^{\tbf(\rho)} \circ v^\rho : H^\rho_a \to H^{\tbf' \circ \tbf(\rho)}_e  \, \}_\rho \ ,
\\ \\
v^{-1} := \{ \, v^{\tbf^{-1}(\tau),*} : H^\tau_o \to H^{ \tbf^{-1}(\tau) }_a  \, \}_\tau \ .
\end{array}
\right.
\end{equation}
In particular, defining for any $o \in \Delta$
\[
\Nat_o  \, := \,  \{ v \in U_\otimes( H_o , H_o \circ \tbf ) \, : \, \tbf \in \Hol_o\mS  \} \ ,
\]
we obtain a group with identity $1_o$.
The interesting point is that any
$G_o$, $o \in \Delta$,
is a normal subgroup of $\Nat_o$ generated by those $y \in \Nat_o$ with $\tbf = id$.
Thus we have the 2--groups ${}^2G_o := \{ G_o \tto \Nat_o \}$.
We keep in evidence a reference $a \in \Delta$ and write
\begin{equation}
\label{eq.dual2group}
G_\bullet := G_a \ \ \ , \ \ \
\Nat_\bullet := \Nat_a \ \ \ , \ \ \
{}^2G_\bullet := (G_\bullet \tto \Nat_\bullet) \, .
\end{equation}
Let $v \in U_\otimes( H_o , H_o \circ \tbf )$ and $v' \in U_\otimes( H_o , H_o \circ \tbf' )$.
By (\ref{eq.N0}) we have $v^{-1} \cdot v' \in G_o$ if, and only if, $\tbf = \tbf'$.
Thus we make the identification 
\begin{equation}
\label{eq.dual2group.a}
\Nat_o / G_o \ \equiv \ \Hol_o\mS \ .
\end{equation}
%
%

\paragraph{The embedding.}
By \cite[Theorem 6.1]{DR89}, representations of $G_o$ are in one-to-one correspondence with objects $\rho \in \obj S_o$,
by writing 
\begin{equation}
\label{eq.D2.3aa}
\pi^\rho : G_\bullet \to U(H^\rho_o) \ \ , \ \  \pi^\rho(g) \, = \, g^\rho \ .
\end{equation}
%
%
%
%
The following theorem is expressed under our usual assumption that $\Delta$ is connected, 
anyway it holds in the general case with obvious modifications.

\begin{thm}[Embedding of DR-presheaves]
\label{thm.D2.1}
Let $\mS = (S,r)^\Delta$ be a DR-presheaf.
Then there is a unique (up to isomorphism) 2--group ${}^2G_\bullet$ and a unique (up to equivalence)
$q = (u,g) \in \check{Z}^2(\Delta,{}^2G_\bullet)$
such that there is an embedding
$\wt{\mS} \to \Hilb_q(\Delta)$.
\end{thm}

\begin{proof}
Without loss of generality we assume that $\mS$ is a presheaf bundle, Remark \ref{rem.pb}.

As a \textbf{first step} we construct the 2-cocycle $q = (u,g)$.
To this end we consider a path frame $p_a = \{ p_{ao} : o \to a \}$, the unitary tranformations (\ref{eq.D2.1a}),
and for any $o \in \Delta$ write
\begin{equation}
\label{eq.barv}
\left\{
\begin{array}{ll}
\bar{r}_{ao} \, := \, r_{p_{ao}} \, = \, r_{a |b_n|} \circ r_{ |b_n| \partial_0 b_{n-1} } \circ \ldots \circ r_{ |b_1| o  } \ , \\ \\
\bar{v}_{ao} \, := \, v_{ |b_n| a }^{-1} \cdot v_{ |b_n| \partial_0 b_{n-1}  }  \cdot \ldots \cdot  v_{ |b_1| o } \ .
\end{array}
\right.
\end{equation}
Note that $\bar{r}_{ao} :  S_o \to S_a$ is an isomorphism, whilst $\bar{v}_{ao} \in U_\otimes( H_o , H_a \circ \bar{r}_{ao} )$
and in particular $\bar{v}_{aa} = 1_a$.
For convenience we write
$\bar{v}_{oa} := \bar{v}_{ao}^{-1}$ and $\bar{r}_{oa} := \bar{r}_{ao}^{-1}$.
Now, for any $b \in \Sigma_1(\Delta)$ we set
\begin{equation}
\label{def.fb}
\tbfu_b \, := \, \bar{r}_{a \bo} \circ r_b \circ \bar{r}_{\bl a} \ \stackrel{ (\ref{eq.ellb}) }{ = } r_{\ell_b} \, \in \Hol_a\mS \equiv N_\bullet / G_\bullet \ ,
\end{equation}
and
\begin{equation}
\label{def.ub}
u_b \, := \,  \bar{v}_{a\bo} \cdot v_{|b|\bo}^{-1} \cdot v_{|b|\bl} \cdot \bar{v}_{\bl a} \, \in 
U_\otimes ( H_a , H_a \circ \tbfu_b )  \subseteq  N_\bullet
\ .
\end{equation}
To verify the cocycle relations (\ref{eq.B1.01}), we note that
\[
u_\co \cdot u_\cz \cdot u_\cl^{-1} \, \in 
U_\otimes( \, H_a \, , \, H_a \circ \bar{r}_{a \partial_{00}c} \circ r_\co \circ r_\cz \circ r_\cl^{-1} \circ \bar{r}_{\partial_{01}c a} \, )
\ \ \ \ , \
\forall c \in \Sigma_2(\Delta)
\, .
\]
Now, by (\ref{eq.A3}) we have $\partial_{01}c = \partial_{00}c$; moreover, since $r$ fulfils the presheaf relations, we have
$r_\co \circ r_\cz \circ r_\cl^{-1} = id$.
This implies
$\bar{r}_{a \partial_{00}c} \circ r_\co \circ r_\cz \circ r_\cl^{-1} \circ \bar{r}_{\partial_{01}c a}  =  id$,
that is equivalent to saying 
\begin{equation}
\label{eq.N0b}
g_c \, := \, u_\co \cdot u_\cz \cdot u_\cl^{-1} \, \in U_\otimes(H_a) \equiv G_\bullet
\ \ \ \ , \
\forall c \in N_2(\Delta) \ .
\end{equation}
In this way the first of (\ref{eq.B1.01}) is fulfilled.
The second of (\ref{eq.B1.01}) follows by injectivity of the inclusion $G_\bullet \subseteq \Nat_\bullet$ (see remarks after (\ref{eq.B1.01})).
Thus
$q := (u,g) \in \check{Z}^2(\Delta,{}^2G_\bullet)$
as desired.

As a \textbf{second step} we assign an associated $q$-Hilbert gerbe Ex.\ref{ex.Hilb} to any section $\varrho \in \obj \wt{\mS}$.
Given $o \in \Delta$, we write $H^\varrho_o := H_o(\varrho_o)$.
Any $H^\varrho_o$ is endowed with the unitary representation
\begin{equation}
\label{eq.N0c}
\pi^\varrho_o : G_\bullet \to U(H^\varrho_o)
\ \ \ , \ \ \ 
\pi^\varrho_o(g)  :=   \{ \bar{v}_{oa} \cdot g \cdot \bar{v}_{ao} \}^{\varrho_o} \ .
\end{equation}
which extends to the morphism
\begin{equation}
\label{eq.N0d}
\pi^\varrho_o : {}^2G_\bullet \to {}^1U(H^\varrho_o)
\ \ \ , \ \ \ 
\pi^\varrho_o(v) := \{ \bar{v}_{oa} \cdot v \cdot \bar{v}_{ao} \}^{\varrho_o} \ ,
\end{equation}
where ${}^1U(H_\pi)$ is the 2--group defined by the identity of $U(H_\pi)$ as in (\ref{eq.1G}).
In fact, if $v \in U_\otimes( H_a , H_a \circ \tbf )$, $\tbf \in \Hol_a\mS$, 
then $\pi^\varrho_o(v)$ is a unitary operator from $H^\varrho_o$ onto $H_o( \bar{r}_{oa} \circ \tbf \circ \bar{r}_{ao}(\varrho_o) )$;
since $\varrho$ is a section, we have $\bar{r}_{oa} \circ \tbf \circ \bar{r}_{ao}(\varrho_o) = \varrho_o$ and $\pi^\varrho_o(v) \in U(H^\varrho_o)$.
Since $G_\bullet$ is normal in $N_\bullet$, this suffices to prove that (\ref{eq.N0d}) is a 2--group morphism.
Note that in particular $\pi^\varrho_a(v) = v^{\varrho_a}$.

\noindent 
%
%
Now, we define the unitary operators
\[
U^\varrho_b : H^\varrho_\bl \to H^\varrho_\bo
\ \ \ \ , \ \ \ \
U^\varrho_b \ := \  \{ v_{ |b| \bo }^{-1}  \cdot  v_{ |b| \bl } \}^\veps
\ \ \ , \
\forall b \in \Sigma_1(\Delta) \, , \, \veps := \varrho_\bl
\ .
\]
We note that (\ref{def.ub}) implies 
\begin{equation}
\label{eq.N0e}
U^\varrho_b  \ = \ 
\{ \bar{v}_{\bo a} \cdot u_b \cdot \bar{v}_{a \bl} \}^\veps 
\ ,
\end{equation}
This relation can be expressed in an equivalent way by defining the unitary operators
$U^\varrho_{ao} : H^\varrho_o \to H^\varrho_a$, $U^\varrho_{ao} := (\bar{v}_{ao})^{\varrho_o}$, $o = \bl , \bo$:
\begin{equation}
\label{eq.N0e'}
\pi^\varrho_a(u_b) \ = \ (u_b)^{\varrho_a} \ = \ U^\varrho_{a\bo} \circ U^\varrho_b \circ U^\varrho_{\bl a} \ .
\end{equation}
Thus, given $c \in \Sigma_2(\Delta)$ and $\veps := \varrho_{ \partial_{01}c }$ we have
\[
\begin{array}{lcl}
U^\varrho_\co \circ U^\varrho_\cz \circ U^{\varrho , *}_\cl & = & 
\{ \bar{v}_{\partial_{01}c a} \cdot u_\co \cdot u_\cz \cdot u_\cl^{-1} \cdot \bar{v}_{a \partial_{01}c } \}^\veps \ = \\ & = &
\{ \bar{v}_{\partial_{01}c a} \cdot g_c \cdot \bar{v}_{a \partial_{01}c } \}^\veps \ = \\ & = &
\pi^\varrho_{ \partial_{01}c }(g_c)
\ .
\end{array}
\]
This proves that the first of (\ref{eq.B2.2}) is verified for the pair $\check{\mH}^\varrho := (H^\varrho,U^\varrho)_\Delta$.
Let now $g \in G_\bullet$ and $b \in \Sigma_1(\Delta)$. We set $\veps := \varrho_\bl$ and compute
\[
\begin{array}{lcl}
U^\varrho_b \circ \pi^\varrho_\bl(g) & = & 
\{ \bar{v}_{\bo a} \cdot u_b \cdot \bar{v}_{a \bl} \cdot \bar{v}_{\bl a} \cdot g \cdot \bar{v}_{a \bl} \}^\veps \ = \\ & = &
\{ \bar{v}_{\bo a} \cdot \wa u_b(g) \cdot u_b \cdot \bar{v}_{a \bl} \}^\veps \ = \\ & = &
\{ \bar{v}_{\bo a} \cdot \wa u_b(g) \cdot \bar{v}_{a \bo} \cdot \bar{v}_{\bo a} \cdot u_b \cdot \bar{v}_{a \bl} \}^\veps \ = \\ & = &
\pi^\varrho_\bo(\wa u_b(g)) \circ U^\varrho_b
\ .
\end{array}
\]
This proves that the second of (\ref{eq.B2.2}) is verified, and $\check{\mH}^\varrho$ is a $q$-Hilbert gerbe.
It remains to prove that $\check{\mH}^\varrho$ is associated to $q$. 
To this end we note that we already have the required 2-group morphism (\ref{eq.N0d}), 
and we just have to prove that the cocycle $\tau^a = ( U^{\varrho,a} , \delta^a )$ defined by $\check{\mH}^\varrho$ in the sense of Prop.\ref{prop.B3.1} is equivalent to 
$\pi^\varrho_{a,*}q := ( \pi^\varrho_a \circ u , \pi^\varrho_a \circ g )$.
Now, (\ref{eq.N0e'}) and the definition of $U^{\varrho,a}_b :=$ $U^\varrho_{a\bo} \circ U^\varrho_b \circ U^\varrho_{\bl a}$
imply $U^{\varrho,a}_b = \pi^\varrho_a(u_b)$; this also implies 
$\delta^a_c :=$ 
$U^{\varrho,a}_\co \circ U^{\varrho,a}_\cz \circ (U^{\varrho,a}_\cl)^{-1} =$ 
$\pi^\varrho_a( u_\co \cdot u_\cz \cdot u_\cl^{-1} ) =$
$\pi^\varrho_a(g_c)$,
$c \in \Sigma_2(\Delta)$.
That is, $\tau^a = \pi^\varrho_{a,*}q$ and $\check{\mH}^\varrho$ is associated to $q$ by means of (\ref{eq.N0d}).

The \textbf{third step} is to construct the embedding $\wt{\mS} \to \Hilb_q(\Delta)$.
To this end, given $t \in (\varrho,\varsigma)$ we assign a family $\check{t}$ of linear operators by defining
$\check{t}_o :=$ $H_o(t_o) : H^\varrho_o \to H^\varsigma_o$, $o \in \Delta$. 
By definition we have
$\check{t}_o \circ \pi^\varrho_o(g) =$ $\pi^\varsigma_o(g) \circ \check{t}_o$ for all $g \in G_\bullet$. 
Moreover, defining $\tau_0 := \varrho_\bo$, $\tau_1 := \varsigma_\bl$, $p := p_{\bo a} * p_{a \bl}$,
and applying (\ref{eq.D2.1}), we get 
\[
\begin{array}{lcl}
U^\varsigma_b \circ \check{t}_\bl & = &
\{ \bar{v}_{\bo a} \cdot u_b \cdot \bar{v}_{a \bl} \}^{\tau_1} \circ H_\bl(t_\bl) \ = \\ & = &
H_\bo(r_p(t_\bl)) \circ \{ \bar{v}_{\bo a} \cdot u_b \cdot \bar{v}_{a \bl} \}^{\tau_0} \ = \\ & = &
H_\bo(t_\bo) \circ U^\varrho_\bl \ = \\ & = &
\check{t}_\bo \circ U^\varrho_\bl
\ .
\end{array}
\]
This proves that $\check{t}$ is a gerbe morphism, and allows us to define the functor
\begin{equation}
\label{eq.emb}
\varrho \mapsto \check{\mH}^\varrho \ \ \ \ , \ \ \ \ t \mapsto \check{t} \ .
\end{equation}
Since each $H_o$, $o \in \Delta$, is a *-functor, we have that (\ref{eq.emb}) is actually a *-functor.
Moreover, (\ref{eq.D2.1}) and the fact that any $H_o$ is a tensor functor imply that (\ref{eq.emb}) preserves the tensor product.
Since each $H_o$ is an embedding
{\footnote{An embedding is a functor injective on the objects and defining isomorphisms on the spaces of arrows.}},
we have that (\ref{eq.emb}) is an embedding.

As a \textbf{final step} we discuss uniqueness of ${}^2G_\bullet$ and $q$.
To this end, we remark that our construction makes use of three choices:
(1) the family of embeddings $H = \{ H_o \}$;
(2) the natural transformations $v_{\omega o} \in U_\otimes( H_o , H_\omega \circ r_{\omega o} )$, $o \leq \omega$;
(3) the path frame $p_a = \{ p_{ao} : o \to a \}$.
Chosing a different family of embeddings $H' = \{ H'_o \}$ we get the 2--groups
${}^2G'_o = ( G'_o \tto N'_o )$, $o \in \Delta$,
where 
$N'_o :=$ $\{ u \in U_\otimes( H'_o , H'_o \circ \tbf ) , \tbf \in \Hol_o\mS \}$
and
$G'_o :=$ $U_\otimes(H'_o)$.
Now, applying as usual \cite[Theorem 6.9]{DR89} we have that for any $o \in \Delta$ there is 
$w_o \in U_\otimes( H_o , H'_o )$.
Thus we define the map
$\beta_o : N_o \to N'_o$, $\beta_o(u) := w_o \cdot u \cdot w_o^{-1}$.
It is easily seen that $\beta_o$ is an isomorphism and that $\beta_o(G_o) = G'_o$, thus we get 2--group isomorphisms
$\beta_o : {}^2G_o \to {}^2G'_o$.
Let now $v'_{\omega o} \in U_\otimes(H'_o,H'_\omega \circ r_{\omega o})$, $o \leq \omega$, be natural transformations as in (\ref{eq.D2.1a})
and $q' = (u',g')$ denote the ${}^2G'_\bullet$-cocycle defined by $\{ v'_{\omega o} \}$ as in the first step of the present proof.
Then by construction we have 
$w_a \cdot u_b \cdot w_a^{-1} \in$ $U_\otimes( H'_a , H'_a \circ \tbfu_b )$, $b \in \Sigma_1(\Delta)$,
and since ${u'_b}^{-1} \in$ $U_\otimes( H'_a , H'_a \circ \tbfu_b^{-1} )$, we find
\[
w_a \cdot u_b \cdot w_a^{-1} \cdot {u'_b}^{-1} \, =: \, h_b^{-1} \in U_\otimes(H'_a) = G'_a \ .
\]
From the above identity we obtain
$\beta_a(u_b) \ = \ h_b^{-1} \cdot u'_b$.
Keeping in mind (\ref{eq.B1.01a}), we conclude that the pair $(1,h)$ makes $\beta_{a,*}q$ cohomologous to $q'$
{\footnote{
Note that since the maps $G_a \to N_a$ and $G'_a \to N'_a$ are inclusions, the second of (\ref{eq.B1.01a}) is redundant and there is no need to verify it.
}}.
Finally, let us consider $e \in \Delta$ and a path frame $p'_e = \{ p'_{oe} : e \to o \}$.
Given $o \in \Delta$ we consider 
$\bar{r}'_{eo} := r_{p'_{oe}} : S_e \to S_o$
and
$\bar{v}'_{oe} := v_{p'_{oe}} \in$ $U_\otimes(H_o , H_e \circ \bar{r}'_{eo})$,
defined as in (\ref{eq.barv}) with the paths of $p'_e$ in place of those of $p_a$.
The corresponding cocycle is given by 
$u'_b :=$ $\bar{v}'_{e\bo} \cdot v_{|b|\bo}^{-1} \cdot v_{|b|\bl} \cdot \bar{v}'_{\bl e}$,
with $g'_c \in G_e$ defined analogously to $g_c$. 
We set $q' := (u',g') \in$ $\check{Z}^2(\Delta,{}^2G_e)$.
Then we define the isomorphism
$\beta : {}^2G_2 \to {}^2G_\bullet$, $\beta(u') := \bar{v}'_{ae} \cdot u' \cdot \bar{v}'_{ea}$,
and
$v_o :=$ $\bar{v}_{ao} \cdot \bar{v}'_{oe} \cdot \bar{v}'_{ea} \in$ $N_\bullet$,
for all $o \in \Delta$, so that we find
$v_\bo \cdot \beta(u'_b) \cdot v_\bl^{-1} =$ $u_b$,
implying that $\beta_*q'$ is cohomologous to $q$.
Combining the above constructions of isomorphisms and equivalences,
we get uniqueness of ${}^2G_\bullet$ up to isomorphism and uniqueness of $q$ up to equivalence.
\end{proof}

We call ${}^2G_\bullet$ the \emph{2--group defined by $\mS$} and $q$ the \emph{cocycle defined by $\mS$}.

\begin{thm}
\label{cor.ThmEmb.1}
Given DR-presheaf bundles $\mS = (S,r)^\Delta$ and $\mS' = (S',r')^\Delta$, the following are equivalent:
$(i)$  There is an isomorphism $\eta : \mS \to \mS'$.
$(ii)$ $\mS$ and $\mS'$ define the same 2--group ${}^2G_\bullet$ (up to isomorphism)
       and the same cocycle $q \in \check{Z}^2(\Delta,{}^2G_\bullet)$ (up to equivalence).
\end{thm}

\begin{proof}
We prove $(i) \Rightarrow (ii)$.
Let $H'_o : S'_o \to \Hilb$, $o \in \Delta$, denote a family of embeddings for $\mS'$. 
Then $H_o :=$ $H'_o \circ \eta_o$, $o \in \Delta$, is a family of embeddings for $\mS$.
We consider a family 
$v' = \{ v'_{\omega o} \in U_\otimes( H'_o \, , \, H'_\omega \circ r'_{\omega o} ) \}_{o \leq \omega}$ 
as in (\ref{eq.D2.1a}). 
From the proof of the previous theorem we know that $v'$ determines the 2--group ${}^2G'_\bullet$
and the cocycle $q' \in \check{Z}^2(\Delta,{}^2G'_\bullet)$ defined by $\mS'$.
On the other hand we have
$H'_\omega \circ r'_{\omega o} \circ \eta_o = H_\omega \circ r_{\omega o}$,
and checking the definition (\ref{eq.D2.1}) we find 
$v'_{\omega o} \in U_\otimes( H_o \, , \, H_\omega  \circ r_{\omega o} )$
for all $o \leq \omega$. Thus we conclude that $\mS$ defines ${}^2G'_\bullet$ as 2--group and $q'$ as cocycle.

We prove $(ii) \Rightarrow (i)$.
Let $\beta : {}^2G_\bullet \to {}^2G'_\bullet$ be a 2--group isomorphism with ${}^2G_\bullet$ (${}^2G'_\bullet$) defined by $\mS$ ($\mS'$).
Then $S_o \simeq \wa{G}_\bullet$ and $S'_o \simeq \wa{G}'_\bullet$, 
thus performing a choice of the standard fibre we assume without loss of generality that 
$S'_o = S_o = \wa{G}_\bullet$ for all $o \in \Delta$.
Let now $q=(u,g)$ and $q'=(u',g') \in \check{Z}^2(\Delta,{}^2G_\bullet)$ denote the cocycles defined by $\mS$ and $\mS'$ respectively.
Then by construction we have 
$u_b \in U_\otimes( H_a , H_a \circ \tbfu_b )$
and
$u'_b \in U_\otimes( H'_a , H'_a \circ \tbf'_b )$,
where $\tbfu_b$ and $\tbf'_b$ are defined as in (\ref{def.fb}).
Thus passing to classes modulo $G_\bullet$ we get
\begin{equation}
\label{eq.uGb}
u_{b,G_\bullet} = \tbfu_b = \bar{r}_{a \bo} \circ r_b \circ \bar{r}_{\bl a}
\ \ \ , \ \ \
u'_{b,G_\bullet} = \tbf'_b = \bar{r}'_{a \bo} \circ r'_b \circ \bar{r}'_{\bl a}
\ .
\end{equation}
Now, by hypothesis there is a pair $(v,h)$ such that (\ref{eq.B1.01a}) holds.
Passing to classes modulo $G_\bullet$ and keeping in mind (\ref{eq.uGb}) we get
\begin{equation}
\label{eq.uGb'}
\bar{r}_{a \bo} \circ r_b \circ \bar{r}_{\bl a} 
\ = \ 
v_{\bo,G_\bullet} \circ \bar{r}'_{a \bo} \circ r'_b \circ \bar{r}'_{\bl a} \circ v_{\bl,G_\bullet}^{-1}
\ .
\end{equation}
For any $o \in \Delta$ we set $\eta_o := \bar{r}'_{oa} \circ v_{o,G_\bullet}^{-1} \circ \bar{r}_{oa}^{-1}$;
by construction, $\eta_o \in \Aut \wa{G}_\bullet$ can be regarded as an isomorphism $\eta_o : S_o \to S'_o$,
and (\ref{eq.uGb'}) says that the relations
$r'_b \circ \eta_\bl = \eta_\bo \circ r_b$, $b \in \Sigma_1(\Delta)$,
hold, that is, $\eta$ is a presheaf isomorphism. This concludes our proof.
\end{proof}

\paragraph{The DR-presheaf of a non-abelian cocycle.}
Let $G$ be a compact group. Then we have an obvious embeding 
$H : \wa{G} \to \Hilb$,
defined by assigning to any unitary representation $\pi : G \to U(H_\pi)$ the underlying Hilbert space $H_\pi$.
We define
\[
N \, := \, \{ v \in U_\otimes( H , H \circ \tbf ) \, , \, \tbf \in \Aut \wa{G} \} \ .
\]
We endow $N$ with the product (\ref{eq.N0}).
In this way $N$ becomes a group having $U_\otimes(H) \simeq G$ as a normal subgroup.
This yields the 2--group ${}^2G = (G \tto N)$.
Reasoning as in (\ref{eq.dual2group.a}), we conclude that $N/G \simeq \Aut \wa{G}$.

Let now $q=(u,g) \in \check{Z}^2(\Delta,{}^2G)$. For any $o \leq \omega$, we define the functors
\[
r_{o \omega} : \wa{G} \to \wa{G}
\ \ \ , \ \ \ 
r_{o \omega} \, := \, u_{b,G} \ \ \ \ , \ b := ( \omega , o ; \omega) \in \Sigma_1(\Delta) \ .
\]
By construction $r_\co \circ r_\cz = \unl{g_c}_G \circ r_\cl =$ $r_\cl$ for all $c \in \Sigma_2(\Delta)$,
implying the presheaf relations
$r_{o\omega} \circ r_{\omega\xi} =$ $r_{o\xi}$, $o \leq \omega \leq \xi$.
Thus the pair $\mS^q := (S,r)^\Delta$, where $S_o \equiv \wa{G}$ for all $o \in \Delta$,
is a DR-presheaf bundle, that we call the \emph{dual of $q$}.

Comparing the construction of $\mS^q$ with the first step of the proof of Theorem \ref{thm.D2.1},
we conclude that $q$ is the cocycle defined by $\mS^q$ (up to a reduction of $N$ to $N_\bullet \subseteq N$).
By Theorem \ref{cor.ThmEmb.1}, $q' \in \check{Z}^2(\Delta,{}^2G)$ is cohomologous to $q$ 
if, and only if, there is a presheaf isomorphism
$\eta : \mS^q \to \mS^{q'}$.

Let now $u \in Z^1(\Delta,N)$. Then the group $u$-bundle $\mG_u$ is defined, Ex.\ref{ex.Gb}.
We set
\begin{equation}
\label{eq.waG}
\wa{\mG}_u \, := \, \mS^{du} \ ,
\end{equation}
where $du = (u,1) \in \check{Z}^2(\Delta,{}^2G)$ is defined as in \S \ref{sec.B1}.
We call $\wa{\mG}_u$ the \emph{dual of $\mG_u$}. 
Let $u' \in Z^1(\Delta,N)$ such that $du$ is cohomologous to $du'$ in $\check{Z}^2(\Delta,{}^2G)$
by means of a pair $(v,h)$. 
By the previous considerations, this is equivalent to saying that there is an isomorphism 
$\wa{\mG}_u  \simeq  \wa{\mG}_{u'}$ of DR-presheaves.
The maps
$\eta_o(y) := \wa{v}_o(y)$ and $\beta_b(y) := h_b v h_b^{-1}$,
at varying of $y \in G$, $o \in \Delta$, $b \in \Sigma_1(\Delta)$, define a gerbe isomorphism 
$\eta_\beta : \mG_u \to \mG_{u'}$.

\paragraph{Presheaves of DR-algebras and twisted C*-dynamical systems.}
Let $\mS = (S,r)^\Delta$ be a DR-presheaf bundle and $\varrho$ a section of $\mS$. 
Then for any $o \in \Delta$ a $C^*$-algebra $O_{\varrho,o}$ is defined,
generated by the spaces of arrows $(\varrho_o^l,\varrho_o^m)$, $l,m \in \bN$, where $\varrho_o^l$ is the $l$-fold tensor power
\cite[\S 4]{DR89}.
Since the construction of $O_{\varrho,o}$ is functorial, the functors $r_{eo}$, $e \leq o$, define *-isomorphisms 
\begin{equation}
\label{eq.DR1}
r^\varrho_{eo,*} : O_{\varrho,o} \to O_{\varrho,e} \ ,
\end{equation}
with 
$r^\varrho_{eo,*} \circ r^\varrho_{o\omega,*} = r^\varrho_{e\omega,*}$ for all $e \leq o \leq \omega$.
Of course, the inverses $r^\varrho_{\omega o , *} := (r^\varrho_{eo,*})^{-1}$ fulfil the precosheaf relations,
thus we have the $C^*$-bundle $\mO_\varrho = (O_\varrho,r^\varrho_*)_\Delta$.
Always by functoriality, given embeddings $H_o : S_o \to \Hilb$, $o \in \Delta$, there are *-monomorphisms
\begin{equation}
\label{eq.DR2}
H^\varrho_{o,*} : \mO_{\varrho,o} \to \mO_d \ ,
\end{equation}
where $d \in \bN$ is the dimension of $\varrho_o$ and $\mO_d$ is the Cuntz algebra. 
Now, we fix $a \in \Delta$ and set $G_\bullet := U_\otimes(H_a)$. 
Then, the morphism (\ref{eq.N0d}), composed with the action
$\bU(d) \to \Aut \mO_d$ (\ref{eq.Od.dyn}),
defines the morphism
\begin{equation}
\label{eq.DR3}
\alpha^\varrho_o : {}^2G_\bullet \to {}^1\Aut \mO_d \ .
\end{equation}
The fixed-point $C^*$-algebra of $\mO_d$ under the $G_\bullet$-action is
$H^\varrho_{o,*}(\mO_{\varrho,o}) \simeq \mO_{\varrho,o}$, see \cite[\S 4]{DR89}.
This implies that $\mO_\varrho$ is a bundle of fixed-point algebras,
thus the results in \S \ref{sec.C4} yield:
\begin{prop}
\label{prop.DR.tw}
Any section $\varrho$ of a DR-presheaf bundle defines:
$(i)$   A bundle $(\mO_\varrho,r^\varrho_*)_\Delta$ of fixed-point algebras of $\mO_d$;
$(ii)$  A gerbe $\check{\mF}_\varrho = ( F_\varrho , \jmath )_\Delta$, with fibres isomorphic to $\mO_d$;
$(iii)$ A twisted $C^*$-dynamical system $\mO_* = ( \mO_d \, , \, \alpha^\varrho_a \, , \, \bar{\jmath}^a \, , \, \gamma_* )$.
\end{prop}

\subsection{The gerbe of a Haag-Kastler net}
\label{sec.D4}

Let $M$ denote a globally hyperbolic spacetime and $\Delta$ its base of diamonds.
A \emph{Haag-Kastler net} over $M$ is given by a Hilbert space $H$ and a family of Von Neumann algebras 
$A_o \subset B(H)$, $o \in \Delta$,
fulfilling certain physically motivated properties \cite{GLRV01,BR08}.
Among these ones, we just need to mention the \emph{isotony property} 
$A_o \subseteq A_\omega$, $o \subseteq \omega$;
this defines inclusion *-morphisms $\jmath_{\omega o} : A_o \to A_\omega$, making $\mA := (A,\jmath)_\Delta$ a $C^*$-precosheaf.
The basic idea is that any $A_o$ is generated by gauge-invariant field operators localized in $o \in \Delta$.
An example is given by the fixed-point $C^*$-precosheaf under the $\bU(1)$-gauge action of the Dirac precosheaf \S \ref{sec.C1}.

An important invariant of $\mA$ is the category $Z^1(\mA)$, defined as follows \cite{BR08}.
The objects are given by \emph{cocycles} $z = \{ z_b \in U(A_{|b|}) \}$ fulfilling the relations
\begin{equation}
\label{eq.zb}
z_\co z_\cz = z_\cl \ \ \ , \ \ \ \forall c \in \Sigma_2(\Delta) \ ;
\end{equation}
the spaces of arrows $(z,z')$ are given by families $t = \{ t_o \in A_o \}$ such that
\[
t_\bo z_b \, = \, z'_b t_\bl \ \ \ , \ \ \ \forall b \in \Sigma_1(\Delta) \ .
\]
Any cocycle $z$ defines a Hilbert bundle $\mH^z = (H^z,U^z)_\Delta$, and a precosheaf morphism
$\mA \to \mB(\mH^z)$
playing the role of a sector. It turns out that $Z^1(\mA)$ is a DR-category.

From the above definitions it is evident that objects and arrows of $Z^1(\mA)$ have a nature of "vector fields" over $\Delta$.
This property is confirmed by the result that $Z^1(\mA)$ is the category of sections of a DR-presheaf $\mS(\mA)$ \cite{VasQFT}.
Thus applying Theorem \ref{thm.D2.1} we find:
\begin{thm}
\label{thm.HK}
Let $\mA = (A,\jmath)_\Delta$ be a Haag-Kastler net.
Then there are, unique up to isomorphism and equivalence respectively:
$(i)$   a compact group $G$ with a 2--group structure ${}^2G := (G \tto N)$;
$(ii)$  a cocycle $q = (u,g) \in \check{Z}^2(\Delta,{}^2G)$.
Any $z \in Z^1(\mA)$ defines a $q$-Hilbert gerbe $\check{\psi}^z$, and this assignment defines an embedding
$\check{\psi} : Z^1(\mA) \to \Hilb_q(\Delta)$.
\end{thm}

\begin{rem}
Higher category structures already appeared in quantum field theory on curved spacetimes,
see \cite{BS17} and related references. 
The scenario is the one of locally covariant quantum field theory,
and it could be of interest to understand whether there is some relation between 
the gerbe ${}^2\check{\mG}_q$ defined by a Haag-Kastler net and the "local gauge groupoids"
mentioned in the above cited paper.
\end{rem}

As we mentioned in \S \ref{sec.intro} we expect that $q$ is trivial,
and this would mean that the dual object of $Z^1(\mA)$ should be the group $G$, instead of a group gerbe or a group bundle. 
A brief argument supporting this hypothesis is given in the following lines.
It certainly works in the case of the Minkowski spacetime,
nevertheless a work in progress, involving $\mS(\mA)$, suggests that it should apply to generic spacetimes.

The embedding $\check{\psi}$ is realized by enlarging the local algebras $A_o$ by means of isometries
$\psi^z_{o,k}$, $k = 1, \ldots , d$,
where $d$ is the dimension of $z$ as an object of $Z^1(\mA)$.
The so-obtained $C^*$-algebra is denoted by $F_o$ and is interpreted as the one generated by (not necessarily gauge-invariant)
fields localized in $o \in \Delta$. In this way, for any $o \in \Delta$ we have finite-dimensional vector spaces 
\[
\psi^z_o \ = \ {\mathrm{span}} \,  \{ \psi^z_{o,k} \, , \, k = 1, \ldots , d \, \} \, \subset F_o \ .
\]
By construction one has $z_b \Psi \in \psi^z_\bo$ for all $b \in N_1(\Delta)$ and $\Psi \in \psi^z_\bl$.
Thus multiplication by $z_b$ defines a map
$u^z_b : \psi^z_\bl \to \psi^z_\bo$
that turns out to be unitary with respect to the scalar product 
$\left\langle \Psi , \Psi' \right\rangle {\bf 1} :=$ $\Psi^*\Psi'$, $\Psi,\Psi' \in \psi^z_\bl$,
see \cite[\S 3.4]{Rob0}. In conclusion,
\begin{itemize}
\item By (\ref{eq.zb}), the maps $u^z$ fulfill the precosheaf relations $u^z_\co \circ u^z_\cz = u^z_\cl$, $c \in N_2(\Delta)$,
      instead of the more general (\ref{eq.B2.1}). Thus any $( \psi^z , u^z )_\Delta$ is a Hilbert bundle;
\item Any $F_o$ acts on a Hilbert space $H^\psi_o$ carrying a representation $\pi_o : G \to U(H^\psi_o)$. 
      We have $\pi_o(g)\Psi \in \psi^z_o$, $\Psi \in \psi^z_o$, thus we get unitaries
      $\pi_o^z(g) \in U(\psi^z_o)$.
      Since operators in $A_o$ must commute with $\pi_o(g)$, we have
      $\pi_{|b|}(g)z_b =$ $z_b \pi_{|b|}(g)$ for all $b \in \Sigma_1(\Delta)$.
      It can be shown that this implies
      $\pi_\bo^z(g) \circ u^z_b =$ $u^z_b \circ \pi_\bl^z(g)$,
      thus $( \psi^z , u^z )_\Delta$ is a $G$-precosheaf in the sense of \S \ref{sec.A3}.
\end{itemize}
The above remarks suggest that the dual object is $G$, acting on the Hilbert bundles $( \psi^z , u^z )_\Delta$.

\paragraph{Acknowledgments.}
The author would like to thank an anonymous referee for several useful remarks
on a previous version of the present work, and for the references \cite{BMZ13,MP16,BM17,BS17}.

%


\end{document}